\documentclass[12pt, oneside]{amsart}
\usepackage{command}
\title[Hochschild cohomology of intersection complexes and BV algebras]{Hochschild cohomology of intersection complexes and Batalin-Vilkovisky algebras}
\author{Ismaïl RAZACK}
\address{LAMFA, CNRS UMR 7352, Université de Picardie Jules Verne,  33, rue Saint-Leu, 80000, Amiens, France.}
\email{ismail.razack@u-picardie.fr}
\begin{document}
\begin{abstract}
    Let $X$ be a compact, oriented, second countable pseudomanifold. We show that $HH^\ast_\bullet(\widetilde N^\ast_\bullet(X;\mathbb{Q}))$, the Hochschild cohomology of the blown-up intersection cochain complex of $X$, is well defined and endowed with a Batalin-Vilkovisky algebra structure. Furthermore, we prove that it is a topological invariant. More generally, we define the Hochschild cohomology of a perverse differential graded algebra $A_\bullet$ and present a natural Gerstenhaber algebra structure on it. This structure can be extended into a Batalin-Vilkovisky algebra when $A_\bullet$ is a derived Poincaré duality algebra. 
\end{abstract}
\maketitle
{\small\tableofcontents}

\section*{Introduction}
\indent It is well known that the Hochschild cohomology $HH^\ast(A)$ of a differential graded algebra $A$ over a commutative ring $R$ has the structure of a \emph{Gerstenhaber algebra} \cite{Ger63}. 

\begin{defi*}
    An associative, graded-commutative algebra $(A^\ast, \cup)$ is a \emph{Gerstenhaber algebra} if there exists a degree -1 linear map $[-,-]:(A\otimes A)^\ast \to A^{\ast-1}$ such that
    \begin{enumerate}[label=\alph*)]
        \item the suspension of $A$, $(sA, [-,-])$, is a Lie algebra. This means that, for $a,b,c\in A^\ast$, the following identities hold
        \begin{itemize}
        \item graded skew-commutativity: $[a,b]=-(-1)^{(\Real{a}-1)(\Real{b}-1)}[a,b]$,
        \item Jacobi identity: $[[a,b],c]=[a,[b,c]]-(-1)^{(\Real{a}-1)(\Real{b}-1)}[b,[a,c]]$,
         \end{itemize}   
    \item and both operations satisfy the Leibniz rule: 
    \[[a, b\cup c]=[a,b]\cup c + (-1)^{(\Real{a}-1)\Real{b}}b\cup [a,c].\]
    \end{enumerate}
\end{defi*}

Furthermore, this structure can be extended into a \emph{Batalin-Vilkovisky algebra} (BV algebra) when $A$ satisfies some form of duality. 
\begin{defi*}
    A Gerstenhaber algebra $(A^\ast, \cup, [-,-])$ is a \emph{Batalin-Vilkovisky algebra} if there exists a linear map $\Delta:A^\ast\to A^{\ast-1}$ of degree -1 such that $\Delta\circ\Delta=0$ and
    \[(-1)^{\Real{a}}[a,b]=\Delta(a\cup b)-\Delta(a)\cup b - (-1)^{\Real{a}}a\cup \Delta(b)\]
    for any $a,b\in A^\ast$.
\end{defi*}

Various cases are collected in \cite{Abb15} and among them, we can find the notion of \emph{derived Poincaré duality algebra}. The singular cochain complex of a manifold is an example of such an algebra. Menichi shows the following result.
\begin{thm*}[{\cite[Theorem 22]{Men09}}]
    Let $M$ be a compact, simply-connected, oriented, smooth manifold. Then, there is a Batalin-Vilkovisky algebra structure on $HH^\ast(C^\ast(M; \mathbb{F}))$, the Hochschild cohomology of the singular cochains of $M$ with coefficients in a field $\mathbb{F}$.
\end{thm*}

Menichi's proof relies on Poincaré duality and the existence of an isomorphism between $C^\ast(M; \mathbb{F})$ and its linear dual in the derived category of $C^\ast(M; \mathbb{F})$-bimodules. We can ask ourselves whether or not there is a similar result for spaces with singularities, for instance pseudomanifolds. The aim of this paper is to answer this question.

Since a pseudomanifold $X$ doesn't necessarily satisfy Poincaré duality, one cannot directly apply Menichi's approach. In order to restore Poincaré duality, Goresky and MacPherson have introduced the $\overline{p}$-\emph{intersection chain complex} \cite{GM80}. This complex, denoted $I^{\overline{p}}C_\ast(X;R)$, is obtained by considering $\overline{p}$-\emph{allowable} singular chains i.e. chains that verify geometric conditions with respect to the parameter $\overline{p}$ which is called a \emph{perversity}. We denote its homology by $I^{\overline{p}}H_\ast(X;R)$. One can associate a poset of perversities to each pseudomanifold $X$  and the intersection chain complex $I^\bullet C_\ast(X;R)$ is then seen as a family of chain complexes indexed by this poset. This poset admits a greatest element which is denoted by $\overline{t}$. Goresky and MacPherson state Poincaré duality as follows.

\begin{thm*}[\textbf{- Generalized Poincaré duality} - {\cite[Section 3.3]{GM80}}] Let $X$ be an $n$-dimensional oriented PL-pseudomanifold. If $i$ and $j$ are complementary dimensions ($i+j=n$) and if $\overline{p}$ and $\overline{q}$ are complementary perversities ($\overline{p}+\overline{q}=\overline{t}$), then there exists a pairing
    \[I^{\overline{p}}H_i(X;\mathbb{Z})\times I^{\overline{q}}H_j(X; \mathbb{Z})\to \mathbb{Z}\]
    which is non-degenerate when the groups are tensored with $\mathbb{Q}$.
\end{thm*}
To lift this result at the cochain level, one needs to consider variants of intersection (co)homology. We will work with the \emph{blown-up intersection cohomology}, defined by Chataur, Saralegui and Tanré \cite{CST18ration} and denoted by $\widetilde N^\ast_\bullet(X; R)$. One of its advantages is that it is a \emph{perverse differential graded algebra} (pDGA) in Hovey's terminology \cite{Hov09}. The algebraic notions we presented before all have a perverse analogue. 

In this paper, we define the Hochschild cochain complex of a pDGA,  and prove in Theorem \ref{thm:Hoch_is_Gerst} that it is endowed with a \emph{perverse Gerstenhaber algebra} structure.
Furthermore, if $A_\bullet$ is a \emph{perverse derived Poincaré duality algebra} and under some additional assumptions, one can extend this structure into a \emph{perverse Batalin-Vilkovisky algebra}. The detailed statement is given in Theorem \ref{thm:DPDA_is_BV}. The commutative case is easier and is treated in Theorem \ref{thm:commuDPDA_is_BV}. 

As a consequence of all these theorems, we show the following result which is similar to the one given by Menichi. 
\begin{proposition*}[\textbf{\ref{prop:BV_on_BUP}}]
Let $X$ be a compact, oriented, second countable pseudomanifold. Then, there exists a perverse Batalin-Vilkovisky algebra structure on the Hochschild cohomology $HH^\ast_\bullet(\widetilde N^\ast_\bullet(X; \mathbb{Q}))$.
\end{proposition*}

We also give two important properties of this Hochschild cohomology.
\begin{itemize}
    \item \textbf{Corollary \ref{coro:top_invar_Gerst}} It is a topological invariant when the coefficient ring is a field $\mathbb{F}$. If $X$ and $Y$ are two homeomorphic (or stratified homotopy equivalent) pseudomanifolds then we have an isomorphism of perverse Gerstenhaber algebras
    \[HH^\ast_\bullet(\widetilde N^\ast_\bullet(X; \mathbb{F})) \simeq HH^\ast_\bullet(\widetilde N^\ast_\bullet(Y; \mathbb{F})).\]
    Furthermore, in corollary \ref{coro:top_invar_BV}, when $\mathbb{F}=\mathbb{Q}$, we give sufficient conditions on $X$ and $Y$ to have an isomorphism of perverse BV algebras .
    \item \textbf{Proposition \ref{prop:HH_tensorprod_BViso}} Under finiteness assumptions on commutative pDGAs $A_\bullet$ and $B_\bullet$ , we have an isomorphism of perverse BV-algebras
    \[HH^\ast_\bullet(A\boxtimes B) \simeq (HH^\ast(A)\boxtimes HH^\ast(B))_\bullet.\]
\end{itemize} 
Using these results, we show that there exist pseudomanifolds for which the Hochschild cohomology $HH^\ast_\bullet(\widetilde N^\ast_\bullet(X; \mathbb{Q}))$ has non trivial perverse BV algebra structure.

\begin{center}
    \textbf{Outline of the paper}
\end{center}

In section \ref{sect:pDGA}, we present some intersection (co)homology theories and interpret them in Hovey's framework \cite{Hov09}. They are \emph{perverse chain complexes} by his terminology. More generally, we consider the category of \emph{perverse objects} on a fixed target category and give its main properties. If the target category is abelian or closed symmetric monoidal, then so is the new category of perverse objects. For instance, there is a well-defined tensor product $\boxtimes$ on the category of perverse chain complexes and hence, one can deal with monoids (called pDGAs for perverse chain complexes), and modules over a monoid.  

\indent The Hochschild cohomology $HH^\ast(A)$ of an algebra $A$ over a commutative ring $R$ can be defined via the \emph{bar construction} or in terms of derived functors \cite{Lod98}. Section \ref{sect:Hoch-coho} will be about adapting these constructions for pDGAs. First, we define the Hochschild (co)chain complex of a pDGA $A_\bullet$ using the bar construction and show that the Hochschild cohomology is naturally endowed with a perverse Gerstenhaber algebra structure. Later, we present the model categorical framework. A characterisation of cofibrant perverse chain complexes over a field is given in Appendix \ref{appendix:chara-cofib-field}. In Theorem \ref{thm:char_cofib_Amod}, we characterize cofibrant objects for the model category structure, given by Hovey \cite{Hov09}, on modules over a pDGA. We end the section by showing that the Hochschild (co)homology can be interpreted with the $\Tor$ and $\Ext$ functors when working on a field or when $A_\bullet$ is cofibrant. 

\indent In section \ref{sect:DPDA}, we give a sufficient condition on a pDGA $A_\bullet$, namely being a perverse derived Poincaré duality algebra, to extend the perverse Gerstenhaber algebra structure of $HH^\ast_\bullet(A)$ into a perverse BV-algebra. We then show that this condition is indeed satisfied by the blown-up intersection cochain complex. To do so, we restrict ourselves to $\mathbb{Q}$ and use a perverse analogue of Sullivan's polynomial forms, which appears in \cite{CST18ration}. 

\indent The last two sections are dedicated to the properties of this Hochschild cohomology. In section \ref{sect:top_invar}, we prove the invariance of Hochschild cohomology under quasi-isomorphisms of pDGAs and in section \ref{sect:tensor}, we study its good behaviour with the tensor product under some assumptions. These results offer tools to compute $HH^\ast_\bullet(N^\ast_\bullet(X; \mathbb{Q}))$ for specific cases which we present as examples. 

\begin{center}
    \textbf{Acknowledgments}
\end{center}
This preprint is part of my PhD thesis on \emph{Hochschild cohomology of intersection algebras}. I would like to thank my advisor David Chataur for taking the time to read various drafts of this paper and for his insights on several proofs. I am also grateful to Sylvain Douteau for discussions regarding the characterisation of cofibrant perverse chain complexes over a field.

\textbf{Notations and conventions:}
\begin{enumerate}[label=\roman*)]
    \item $\mathbb{F}$ denotes a field and $R$ denotes a commutative ring.
    \item Let $M$ be an $R$-module, its linear dual is denoted $M^\vee:=\Hom_R(M,R)$.
    \item $Ch(R)$ denotes the category of chain complexes. A chain complex $(Z_\ast, d)$ is a sequence of lower $\mathbb{Z}$-graded $R$-modules $\{Z_i\}_{i\in\mathbb{Z}}$ with a degree $-1$ differential. The chain complex whose components are all trivial is denoted $0$.
    \item A cochain complex $(Z^\ast, d)$ is a sequence of upper $\mathbb{Z}$-graded $R$-modules $\{Z^i\}_{i\in\mathbb{Z}}$ with a degree $+1$ differential. With the \emph{classical convention} of \cite[pp 41-42]{FHT01}, we define a chain complex $(Z_\ast, d)$ by setting for $i\in\mathbb{Z}$
    \[Z_i:=Z^{-i}.\]
    Hence, we can think of a cochain complex as an object in $Ch(R)$.
     \item Let $n\in \mathbb{Z}$. The $n$-\emph{sphere chain complex} $\mathbb{S}^n$ is defined as the chain complex concentrated in degree $n$ with $(\mathbb{S}^n)_n=R$ with trivial differential.\\
    The $n$-\emph{disk chain complex} is denoted $\mathbb{D}^n$. It is the chain complex whose components are $R$ in degrees $n$ and $n-1$ and $0$ elsewhere. The only non-trivial differential is $(\mathbb{D}^n)_n\to (\mathbb{D}^n)_{n-1}$ and it is given by the identity.  
    \item The \emph{suspension} $sC_\ast$ of a chain complex $C_\ast$ is the shift in degree by 1 i.e. for $n\in \mathbb{N}$, $(sC)_n=C_{n-1}$. For a cochain complex $C^\ast$, this corresponds to a shift in degree by -1 i.e. $(sC)^n=C^{n+1}$.
\end{enumerate}

\section{Categorical framework for intersection homology}\label{sect:pDGA}
    In this section, we're going to present some intersection homology theories, namely the original theory introduced by Goresky and MacPherson \cite{GM80} and the \emph{blown-up intersection cohomology} studied by Chataur, Saralegui and Tanré. They will be treated as \emph{perverse chain complexes} (i.e. functors from a poset of perversities into the category of chain complexes) and some of their properties are given in this formalism. 
    \subsection{Intersection homology theories}\label{subsect: interhom}
    Let's first define the spaces we will be studying throughout this paper. We will follow \cite[Section 1]{CST18BUP-Alpine}, another good reference for these notions is \cite[Chapter 2]{Fri20}.
    \begin{defi}\label{def:filtered_space}
    A \emph{filtered space} is a Hausdorff space $X$ equipped with a filtration by closed subspaces 
    \[\emptyset = X_{-1} \subseteq X_0 \subseteq \ldots \subseteq X_n=X,\]
    such that $X_n\setminus X_{n-1}$ is non-empty. The \emph{formal dimension} of $X$ is $\dim(X)=n$. 
    
    The \emph{strata} of $X$ are the non-empty connected components of $X_i\setminus X_{i-1}$. 
    The \emph{formal dimension} of a stratum $S\subset X_i\setminus X_{i-1}$ is $\dim(S)=i$. Its \emph{formal codimension} is $\codim(S)=\dim(X)-\dim(S)$.
    The \emph{regular strata} are the strata of dimension $n$ and the other strata are \emph{singular}. The set of strata of $X$ is denoted $\mathcal{S}_X$. 
    \end{defi}

    \begin{defi}
        A filtered space  $X$ is \emph{stratified} if it verifies the \emph{Frontier condition}:
        \begin{center}
        \emph{For any pair $S,T \in \mathcal{S}_X$, if $S\cap \overline{T}\neq \emptyset$ then $S\subset T$.}
        \end{center}
        A continuous map $f:X\to Y$ between stratified spaces is \emph{stratified} if for every $S\in \mathcal{S}_X$ there exists a stratum $T\in \mathcal{S}_Y$ such that
        \[f(S)\subset T \text{ and } \codim(S)\geq \codim(T).\]
        A statified map is a \emph{stratified homeomorphism} if it is an homeomorphism and its inverse map is also stratified. \\        
        We denote by $\strat$ the category whose objects are stratified spaces and whose morphisms are stratified maps.
    \end{defi}
    \begin{defi}
        Let $f,g:X\to Y$ be stratified maps between stratified spaces. They are \emph{stratified homotopic} if there exists a stratified map $H:[0,1]\times X \to Y$ such that $f=H|_{\{0\}\times X}$ and $g=H|_{\{1\}\times X}$. The map $H$ is called a \emph{stratified homotopy}.

        Two stratified spaces $X$ and $Y$ are \emph{stratified homotopy equivalent} if there exist stratified maps $f:X\to Y$ and $g:Y\to X$ such that
        \begin{itemize}
            \item $f\circ g$ and $g\circ f$ are stratified homotopic to $\id_Y$ and $\id_X$ respectively,
            \item for each stratum $S\in \mathcal{S}_X$, we have $\codim(f(S))=\codim(S)$,
            \item and for each stratum $T\in \mathcal{S}_Y$, we have $\codim(g(T))=\codim(T)$.
        \end{itemize}
        We say that the maps $f$ and $g$ are \emph{stratified homotopy equivalences}.
    \end{defi}
    Filtered spaces are enough to define intersection homology but one needs additional structure in order to get Poincaré duality. 
    \begin{defi}
        A filtered space $X$ of dimension $n$, with no codimension one strata, is an \emph{$n$-dimensional pseudomanifold} if for any $i\in\{0,1,\ldots, n\}$, $X_i\setminus X_{i-1}$ is an $i$-dimensional manifold. Furthermore, for any $i\in\{0,1,\ldots, n-1\}$ and for every $x\in X_i\setminus X_{i-1}$, there exist
        \begin{itemize}
            \item an open neighborhood $V$ of $x$ in $X$ endowed with the induced filtration,
            \item an open neighborhood $U$ of $x$ in $X_i\setminus X_{i-1}$,
            \item a \emph{link} i.e. a compact pseudomanifold $L$ of dimension $n-i-1$, whose open cone $\mathring{c}L$ is endowed with the filtration $(\mathring{c}L)_i=\mathring{c}L_{i-1},$
            \item a homeomorphism, $\phi: U \times \mathring{c}L \to V$ such that
                \begin{itemize}
                    \item $\phi(u,v)=u$ for any $u\in U$ and $v$ the apex of $\mathring{c}L$,
                    \item $\phi(U\times \mathring{c}L_j)=V\cap X_{i+j+1}$, for all $j\in\{0, \ldots, n-i-1\}.$
                \end{itemize}
        \end{itemize}
    \end{defi}    
    \begin{rem}
        A pseudomanifold (more generally, a \emph{CS set}) is a stratified space, see \cite[Theorem G]{CST18ration}.
    \end{rem}
    One can endow every filtered space with a poset of \emph{perversities}. The \emph{intersection chain complex} consists of chains that verify some conditions with respect to these perversities. We now give Goresky and MacPherson's original definition of a \emph{perversity}. 
    \begin{defi}
    Let $n\in \mathbb{N}$, a \emph{Goresky-MacPherson $n$-perversity} (GM-perversity) is a map 
    \[\overline{p}:\{0, 1, 2\ldots, n\}\to \mathbb{N}\] 
    such that $\overline{p}(0)=\overline{p}(1)=\overline{p}(2)=0$ and 
    \[\overline{p}(i)\leq \overline{p}(i+1)\leq \overline{p}(i)+1\] 
    for any $i\in \{1,\ldots, n-1\}$.
    \end{defi}
    \begin{rem}\label{rem:genperv}
        Let $X$ be a filtered space of formal dimension $n$. Note that a Goresky-MacPherson perversity induces a map
        \[\overline{p}:\mathcal{S}_X\to \mathbb{N}\]
        by setting for a singular stratum $S$, $\overline{p}(S)=\overline{p}(\codim(S))$ (and 0 on the regular strata). In the following, a \emph{perversity on $X$} will refer to a GM $n$-perversity if  $X$ is a filtered space of formal dimension $n$.\\
        In MacPherson's work \cite{MacP91}, a perversity is defined as a function
        \[\overline{p}:\mathcal{S}_X\to \mathbb{Z}\]
        that takes the value $0$ on the regular strata. This general notion of perversity has been used in papers dealing with non-GM intersection homology. (\cite{Fri20}, \cite{CST18BUP-Alpine} for example).
    \end{rem}
    
    \textbf{Poset of GM-perversities:} In the following, we fix $n\in \mathbb{N}$ and we consider the set of GM $n$-perversities, $\pGM$. It is a poset. The partial order is given by $\leq$: 
    \[\overline{p}\leq \overline{q} \Longleftrightarrow \overline{p}(i)\leq \overline{q}(i), ~\forall i\in \{0,\ldots, n\}.\]
    It can be seen as a category where there is at most one morphism between any pair of perversities.\\
    Notice that it has a minimal element the \emph{zero perversity} $\overline{0}$ and a maximal element, the \emph{top perversity} $\overline{t}$, given by 
    \[\overline{t}(0)=\overline{t}(1)=0 \text{ and } \overline{t}(i)=i-2 \text{ for } i\in \{2,\ldots, n\}.\]
    
    \textbf{Partial symmetric monoidal structure:} In general, the point-wise sum of two GM-perversities is not a GM-perversity but one can define a partial symmetric monoidal structure on $\pGM$. For $\overline{p}, \overline{q}\in \pGM$ such that $\overline{p}+\overline{q}\leq \overline{t}$, we denote by $\overline{p}\oplus\overline{q}$ the smallest GM-perversity which is greater or equal to $\overline{p}+\overline{q}$. Dually, when $\overline{p}\leq \overline{q}$, let $\overline{q}\ominus \overline{p}$ be the biggest GM-perversity which is lower or equal to $\overline{q}-\overline{p}$. \\
    The \emph{dual perversity} (or \emph{complementary perversity}) of $\overline{p}$ is $D\overline{p}:=\overline{t}\ominus\overline{p}$. Note that it is exactly $\overline{t}-\overline{p}$. 
    \\ When $\oplus$ and $\ominus$ are defined, they satisfy the properties expected from a closed symmetric monoidal category.
    
    Let $X$ be a filtered space of formal dimension $n$ and $\overline{p}\in \pGM$. We will follow Chataur, Saralegui and Tanré's exposition of intersection homology given in \cite{CST18ration}. They present $I^{\overline{p}}C_\ast(X;R)$, the \emph{$\overline{p}$-intersection chain complex} of $X$ and define the \emph{blown-up complex of $\overline{p}$-intersection cochains} $\widetilde{N}^\ast_{\overline{p}}(X ;R)$ (also known as the \emph{Thom-Whitney complex}). The reader may refer to \cite{CST18BUP-Alpine}, \cite{CST18ration} and \cite{Fri20} for explicit constructions, we're just going to give the main properties of these complexes.

    \begin{proposition}
        Let $\overline{p}$ be a GM $n$-perversity. There is a functor
        \[\begin{array}{rrcl}
        I^{\overline{p}} C_\ast: & \nstrat & \to & Ch(R) \\
        ~& X & \mapsto & I^{\overline{p}} C_\ast(X ; R). \\
        \end{array}\]   
        from the category of stratified spaces of formal dimension $n$ to the category of chain complexes. Furthermore, the homology of $I^{\overline{p}} C_\ast(X;R)$, called \emph{$\overline{p}$-intersection homology} and denoted $I^{\overline{p}} H_\ast(X;R)$, verifies the following properties.
        \begin{description}[style=unboxed,leftmargin=0cm]
            \item[Mayer-Vietoris] If $\{U, V\}$ is an open cover of $X$, then there is a long exact sequence
            \[\ldots \to I^{\overline{p}}H_\ast(U\cap V) \to I^{\overline{p}}H_\ast(U)\oplus I^{\overline{p}}H_\ast(V)\to I^{\overline{p}}H_\ast(X)\to I^{\overline{p}}H_{\ast-1}(U\cap V)\to \ldots\]
            \item[Cone formula] If $X$ is compact, we have the \emph{cone formula}. For any $k\in \mathbb{N}$,
            \[I^{\overline{p}}H_k(\mathring{c}X; R)=\begin{cases}
                I^{\overline{p}}H_k(X; R) & \text{if }k\leq n-\overline{p}(n+1)-1, \\
                0 & \text{if } k \geq n-\overline{p}(n+1).
            \end{cases}\]
            \item[Topological invariance] If $X$ and $Y$ are two pseudomanifold of formal dimension $n$ which are homeomorphic (not necessarily stratified homeomorphic) then for every $\overline{p}\in\pGM$ we have an isomorphism
            \[I^{\overline{p}} H_\ast(X;R) \simeq I^{\overline{p}} H_\ast(Y;R).\]
            \item[Invariance under stratified homotopy equivalence] Let $f:X\to Y$ be a stratified homotopy equivalence between $X\in \nstrat$ and $Y\in \mstrat$. Let $\overline{p}$ and $\overline{q}$ be perversities of rank $n$ and $m$ respectively such that for any $S\in \mathcal{S}_X$, we have $\overline{p}(\codim(S))=\overline{q}(\codim(T))$  when $f(S)\subset T$.
            Then, we have an isomorphism
            \[I^{\overline{p}} H_\ast(X;R) \simeq I^{\overline{q}} H_\ast(Y;R).\]
        \end{description}
    \end{proposition}
    \begin{defi}
        Let $X$ be a filtered space of formal dimension $n$ and $\overline{p}\in \pGM$. By taking the linear dual of the $\overline{p}$-intersection chain complex, we define the \emph{dual $\overline{p}$-intersection cochain complex}, $I_{\overline{p}}C^\ast(X;R):=\Hom_R(I^{\overline{p}}C_\ast(X;R), R)$ and its homology is the \emph{dual $\overline{p}$-intersection cohomology} $I_ {\overline{p}}H^\ast(X;R)$.
    \end{defi}
    There exists several variants to intersection (co)homology. For instance, we have non-GM intersection (co)homology theories which are presented in \cite{Fri20} and the \emph{blown-up intersection cohomology} which we now consider. Note that the blown-up cochain complex appears in a several papers  \cite{CST18BUP-Alpine}, \cite{CST18ration},  \cite{CST18DP}, \cite{CST20Deligne} and \cite{ST20} to name a few.
    \begin{proposition}\label{prop:bup_prop}
        Let $\overline{p}$ be a GM $n$-perversity. There is a functor
        \[\begin{array}{rrcl}
        \widetilde{N}^\ast_{\overline{p}}: & (\nstrat)^{op} & \to & Ch(R) \\
        ~& X & \mapsto & \widetilde{N}^\ast_{\overline{p}}(X ; R). \\
        \end{array}\]    
        Furthermore, the homology of $\widetilde{N}^\ast_{\overline{p}}(X;R)$, called \emph{blown-up $\overline{p}$-intersection cohomology} and denoted $\mathscr{H}^{\overline{p}}_\ast(X;R)$, verifies the following properties.
        \begin{description}[style=unboxed,leftmargin=0cm]
            \item[Mayer-Vietoris] Suppose that $X$ is paracompact and has an open cover $\{U, V\}$, then there is a long exact sequence
            \[\ldots \to \mathscr{H}_{\overline{p}}^\ast(X;R) \to \mathscr{H}_{\overline{p}}^\ast(U;R)\oplus \mathscr{H}_{\overline{p}}^\ast(V;R)\to \mathscr{H}_{\overline{p}}^\ast(U\cap V; R)\to \mathscr{H}_{\overline{p}}^{\ast+1}(X;R)\to \ldots \]
            \item[Cone formula] If $X$ is a compact $n$-pseudomanifold, we have,
            \[\mathscr{H}^k_{\overline{p}}(\mathring{c}X; R)=\begin{cases}
                \mathscr{H}^k_{\overline{p}}(X; R) & \text{if }k\leq\overline{p}(n+1), \\
                0 & \text{if } k> \overline{p}(n+1)
            \end{cases}\]
            where $\overline{p}$ is a perversity on $\mathring{c}X$.
            \item[Topological invariance] If $X$ and $Y$ are two stratified spaces of formal dimension $n$ which are homeomorphic (not necessarily stratified homeomorphic) then for every $\overline{p}\in\pGM$ we have an isomorphism
            \[\mathscr{H}^\ast_{\overline{p}}(X; R) \simeq \mathscr{H}^\ast_{\overline{p}}(Y; R).\]
            \item[Invariance under stratified homotopy equivalence] Let $f:X\to Y$ be a stratified homotopy equivalence between $X\in \nstrat$ and $Y\in \mstrat$. Let $\overline{p}$ and $\overline{q}$ be perversities of rank $n$ and $m$ respectively such that for any $S\in \mathcal{S}_X$, we have $\overline{p}(\codim(S))=\overline{q}(\codim(T))$  when $f(S)\subset T$.
            Then, we have an isomorphism
            \[\mathscr{H}^\ast_{\overline{p}}(X; R) \simeq \mathscr{H}^\ast_{\overline{q}}(Y; R).\]
        \end{description}
    \end{proposition}    
    
    The blown-up intersection cohomology coincides with intersection cohomology as it was defined by Goresky and MacPherson. Furthermore, it corresponds to the dual intersection cohomology when we work on a field and for GM-perversities.
    \begin{thm}[{\cite[Theorem B]{CST18ration}}]\label{thm:field_intercoho_iso}
        Suppose that $R$ is a field and let $X$ be a filtered space of formal dimension $n$. For any $\overline{p}\in\pGM$, we have a quasi-isomorphism
        \[\Inter:\widetilde{N}^\ast_{\overline{p}}(X ;R)\to I_{\overline{t}-\overline{p}}C^\ast(X ;R).\]
    \end{thm}
    In \cite{CST18ration}, Sullivan's approach to rational homotopy theory is extended to intersection cohomology. In particular, there is a filtered version of Sullivan's polynomial forms: the \emph{blown-up of Sullivan's polynomial forms} $\widetilde{A}^{\ast}_{PL,\bullet}$. We have the following result.    
    \begin{proposition}[{\cite[Corollary 1.39]{CST18ration}}]
    Let $X$ be a stratified space of formal dimension $n$ and $\overline{p}\in \pGM$. The integration map
    \[\int\colon \widetilde{A}^{\ast}_{PL,\overline{p}}(X)\to \widetilde{N}^{\ast}_{\overline{p}}(X; \mathbb{Q})\]
    induces an isomorphism in homology.    
    \end{proposition}
    
    We now recall results regarding algebraic structures found on the blown-up intersection complex.
    \begin{proposition}[{\cite[Proposition 4.2]{CST18BUP-Alpine}}]\label{prop:cup_bup}
        Let $X$ be a filtered space of formal dimension $n$. For any $\overline{p}, \overline{q}\in \pGM$ and $i,j\in \mathbb{N}$, there exists an associative multiplication
        \[-\cup-:\widetilde{N}^i_{\overline{p}}(X ;R)\otimes \widetilde{N}^j_{\overline{q}}(X ;R)\to \widetilde{N}^{i+j}_{\overline{p}\oplus \overline{q}}(X ;R).\]
        It induces a graded commutative multiplication in homology called \emph{intersection cup product}
        \[-\cup-:\mathscr{H}^i_{\overline{p}}(X ;R)\otimes \mathscr{H}^j_{\overline{q}}(X ;R)\to \mathscr{H}^{i+j}_{\overline{p}\oplus \overline{q}}(X ;R).\]
    \end{proposition}
    \begin{proposition}[{\cite[Propostion 6.7]{CST18BUP-Alpine}}]\label{prop:cap_bup}
        Let $X$ be a filtered space of formal dimension $n$. For any $\overline{p}, \overline{q}\in \pGM$ and $i,j\in \mathbb{N}$, there is a well defined \emph{intersection cap product}
        \[-\cap-:\widetilde{N}^i_{\overline{p}}(X ;R)\otimes I^{\overline{q}}C_j(X ;R)\to I^{\overline{p}\oplus \overline{q}}C_{j-i}(X ;R).\]
        It induces a morphism in homology
        \[-\cap-:\mathscr{H}^i_{\overline{p}}(X ;R)\otimes I^{\overline{q}}H_j(X ;R)\to I^{\overline{p}\oplus \overline{q}}H_{j-i}(X ;R).\]
    \end{proposition}
    The following result is a consequence of \cite[Theorem 3.1]{ST20}. 
    \begin{thm}[\textbf{- Poincaré duality}]\label{thm:Poincaré}
    Let $X$ be an $n$-dimensional, second countable, compact and oriented pseudomanifold. Then, for every $\overline{p}\in \pGM$ there exists a quasi-isomorphism
    \[\widetilde N^\ast_{\overline{p}}(X;R) \xrightarrow[\simeq]{DP_X} I^{\overline{p}}C_{n-\ast}(X;R)\]
    where $DP_X$ is the cap product with a fundamental cycle $\zeta\in I^{\overline{0}}C_n(X;R)$.
    \end{thm}
    \subsection{Generalities on perverse objects}
    In \cite{Hov09}, Mark Hovey defines perverse modules and chain complexes for Goresky-MacPherson perversities. We extend this approach and consider perverse objects on any abelian category.
    \begin{defi}
    Let $n\in\mathbb{N}$ and  $\mathcal{C}$ be an abelian category. A \emph{perverse object} in $\mathcal{C}$ of \emph{rank} $n$ is a functor \[M:\pGM\to \mathcal{C}.\] The perverse objects in $\mathcal{C}$ of rank $n$ form a category where morphisms are natural transformations. We will denote it $\mathcal{C}^{\pGM}$. In what follows, the rank of an object will be omitted. 
    \end{defi}
    Alternatively, a perverse object is a family of objects $\{M_{\overline{p}}\}_{\overline{p}\in \pGM}$ in $\mathcal{C}$ such that for any pair of GM-perversities $\overline{p} \leq \overline{q}$ we have a morphism in $\mathcal{C}$
    \[\phi_{\overline{p}\leq \overline{q}}: M_{\overline{p}}\to M_{\overline{q}}\]
    which is the identity when $\overline{p}=\overline{q}$ and such that for any perversities $\overline{p}\leq \overline{q}\leq \overline{r}$, we have $\phi_{\overline{p}\leq \overline{r}}=\phi_{\overline{q}\leq \overline{r}}\circ \phi_{\overline{p}\leq \overline{q}}$. We refer to these morphisms as \emph{structure morphisms}.
    \begin{rem}
    Recall that the category of functors from a small category into an abelian category is also abelian \cite[Chapter IX - Prop 3.1]{MacL95} and that limits and colimits are taken point-wise in functor categories. 
    \end{rem}
    \begin{exemple}
    \begin{itemize}
        \item A \emph{perverse module} is a perverse object in the category of $R$-modules. The category of perverse $R$-modules is denoted $(\modu)^{\pGM}$.
        \item A \emph{perverse graded module} $M^\bullet$ is a family of $R$-modules $\{M_i^{\overline{p}}\}_{i\in\mathbb{Z},\, \overline{p}\in \pGM}$ such that for any $\overline{p} \leq \overline{q}$ and $i\in\mathbb{Z}$ we have an $R$-linear map
        \[M^{\overline{p}}_i\to M^{\overline{q}}_i\]
        which is the identity when $\overline{p}=\overline{q}$.\\
        An element $m\in M^{\overline{p}}_i$ is of degree $\Real{m}=i$ and perverse degree $\overline{p}$.
        \item A \emph{perverse chain complex} is a perverse object in the category of chain complexes $Ch(R)$. Equivalently, a perverse chain complex is a chain complex of perverse modules. The category of perverse chain complexes is denoted $(Ch(R))^{\pGM}$.\\
        The degree $i\in \mathbb{Z}$ and perverse degree $\overline{p}\in \pGM$ component of a perverse chain complex $(Z^\bullet, d^\bullet)$ is denoted $Z^{\overline{p}}_i$. Following the convention given at the beginning of this paper, for $i\in \mathbb{Z}$, we set
        \[Z^i_{\overline{p}}:=Z_{-i}^{\overline{p}}.\]
        This defines a \emph{perverse cochain complex} $Z_\bullet$ whose $\overline{p}$-component is the cochain complex $Z^\ast_{\overline{p}}$.
    \end{itemize}
    \end{exemple}
    
    We collect results from \cite{Hov09}. The original statements are given for perverse $R$-modules but one can easily adapt the proofs to perverse objects.   
    \begin{proposition}
    Let $\overline{p}\in\pGM$, there is an exact evaluation functor 
    \[\begin{array}{rrcl}
    Ev_{\overline{p}}: & \mathcal{C}^{\pGM} & \to & \mathcal{C} \\
    ~& M & \mapsto & M_{\overline{p}}. \\
    \end{array}\]   
    The functor $Ev_{\overline{p}}$ possesses a left adjoint $F_{\overline{p}}$ defined for all $N\in \mathcal{C}$ by
    \[F_{\overline{p}}(N)_{\overline{q}}=\begin{cases}
        N & \text{ if } \overline{p}\leq \overline{q}. \\
        0 & \text{ else.} \\
    \end{cases}\]
    \end{proposition}
    Notice that $F_{\overline{0}}(N)$ is a constant diagram. 
    \begin{proposition}
        If $(\mathcal{C}, \otimes_\mathcal{C}, [~,~], \mathcal{I})$ is a closed symmetric monoidal category then so is the category of perverse object $\mathcal{C}^{\pGM}$.
        The monoidal structure is given by:
        \[(M \boxtimes N)_{\overline{r}}:=\colim_{\overline{p}+\overline{q}\leq \overline{r}}M_{\overline{p}}\otimes_\mathcal{C} N_{\overline{q}}\]
        with unit $F_{\overline{0}}(\mathcal{I})$. The closed structure comes from
        \[\Hom_{\mathcal{C}^{\pGM}}(M,N)_{\overline{r}}:=\lim_{\overline{r}\leq \overline{q}-\overline{p}} [M_{\overline{p}},N_{\overline{q}}].\]
    \end{proposition}
    \begin{rem}
        Using adjunctions, we can also think of the closed structure as morphisms in the category of perverse objects (i.e. natural transformations)
        \[\Hom_{\mathcal{C}^{\pGM}}(M,N)_{\overline{r}}\simeq \mathcal{C}^{\pGM}(F_{\overline{r}}(\mathcal{I})\boxtimes M, N).\]
    \end{rem}
    
    \begin{rem}\label{rem:internal_tenhom_adj}
        Since $\mathcal{C}^{\pGM}$ is a symmetric monoidal category, for any perverse objects $X_\bullet, Y_\bullet$ and $Z_\bullet$ we have an isomorphism
        \[\Hom_{\mathcal{C}^{\pGM}}(X\boxtimes Y , Z)_\bullet \simeq \Hom_{\mathcal{C}^{\pGM}}(X , \Hom_{\mathcal{C}^{\pGM}}(Y,Z))_\bullet.\]
        We'll refer to this isomorphism as \emph{internal tensor-hom adjunction}.
    \end{rem}
    
    \begin{exemple}\label{ex:pmonoidal_chcplx}
    Let $\overline{r}$ be a perversity and $Z^\bullet,~Y^\bullet$ two perverse chain complexes. 
    The monoidal structure on $(Ch(R))^{\pGM}$ is given by 
    \[(Z \boxtimes Y)^{\overline{r}}:=\colim_{\overline{p}+\overline{q}\leq \overline{r}}Z^{\overline{p}}\otimes Y^{\overline{q}}\]
    where the tensor product on the right hand side denotes the tensor product of chain complexes.
    For $n\in \mathbb{Z}$, the degree $n$ and perverse degree $\overline{r}$ component of the perverse chain complex is
    \[(Z \boxtimes Y)_n^{\overline{r}}=\colim_{\overline{p}+\overline{q}\leq \overline{r}}\bigoplus_{i+j=n}Z_i^{\overline{p}}\otimes_R Y_j^{\overline{q}}.\]
    The unit is $F_{\overline{0}}(\mathbb{S}^0)^\bullet$.
    
    The internal Hom from $Z^\bullet$ to $Y^\bullet$ is the perverse chain complex whose degree $n$ and perverse degree $\overline{r}$ component is
    \[\Hom_{(Ch(R))^{\pGM}}(Z,Y)_n^{\overline{r}}=\lim_{\overline{r}\leq \overline{q}-\overline{p}} \prod_{j-i=n} \Hom_R(Z_i^{\overline{p}}, Y_j^{\overline{q}}).\]
    In what follows, we will be interested in the linear dual $DZ^\bullet$ of a perverse chain complex, we have:
    \[DZ_n^{\overline{r}}:=\Hom_{(Ch(R))^{\pGM}}(Z, F_{\overline{0}}(\mathbb{S}^0))_n^{\overline{r}}=\Hom_R(Z_{-n}^{\overline{t}-\overline{r}}, R).\]
    \end{exemple}
    \subsection{Perverse differential graded algebra}\label{subsect:pDGA}
    We are now going to present a perverse analogue of differential graded algebras (DGA). Recall that a DGA is a monoid in the monoidal category $Ch(R)$, this justifies the following definition.
    \begin{defi}
        A \emph{perverse differential graded algebra} (pDGA) is a monoid in the category of perverse chain complexes $(Ch(R))^{\pGM}$.
    \end{defi}
    An equivalent explicit description is given in the next definition.
    \begin{defi}
        A \emph{perverse differential graded $R$-algebra} $A^\bullet$ is a perverse chain complex $(A^\bullet, d^\bullet)$ 
        equipped for every $\overline{p},\overline{q}\in \pGM$ and $i,j\in\mathbb{Z}$ with an associative product $\mu: A_i^{\overline{p}} \otimes A_j^{\overline{q}} \to A_{i+j}^{\overline{p}\oplus\overline{q}}$
        compatible with the poset structure of $\pGM$ i.e. it makes the following diagram commute
        \[\begin{tikzcd}
        A^{\overline{p_1}}_i\otimes A^{\overline{q_1}}_j \arrow[r, "\mu"] \arrow[d, "\phi_{\overline{p_1}\leq\overline{p_2}}\otimes \phi_{\overline{q_1}\leq\overline{q_2}}"'] & A^{\overline{p_1}+\overline{q_1}}_{i+j} \arrow[d, "\phi_{\overline{p_1}+\overline{q_1}\leq\overline{p_2}+\overline{q_2}}"] \\
        A^{\overline{p_2}}_i\otimes A^{\overline{q_2}}_j \arrow[r, "\mu"]                                                                                                      & A^{\overline{p_2}+\overline{q_2}}_{i+j}  \end{tikzcd}\]
        where $\overline{p_1}\leq \overline{p_2}$ and $\overline{q_1}\leq \overline{q_2}$.
        We will denote $\mu(a \otimes b)$ by $ab$. The multiplication has a unit $1\in A^{\overline{0}}_0$. Furthermore, the differential $d^\bullet$ is a derivation with respect to this product i.e.
        \[d_{i+j}^{\overline{p}\oplus \overline{q}}(ab)=(d_i^{\overline{p}}a)b+(-1)^ia(d_j^{\overline{q}}b) \text{ for } a\in A_i^{\overline{p}}\text{ and }b\in A_j^{\overline{q}}.\]
        We will denote by $\pdgAlg$ the category of perverse differential graded algebras whose morphisms are natural transformations between perverse chain complexes that are compatible with the products.\\
        We say that $(A^\bullet, d^\bullet )$ is a \emph{perverse commutative differential graded $R$-algebra} (pCDGA) if for every $\overline{p},\overline{q}\in \pGM$ and $i,j\in\mathbb{Z}$, the following diagram commutes
        \[\begin{tikzcd}
        A^{\overline{p}}_i\otimes A^{\overline{q}}_j \arrow[rr, "\tau"'] \arrow[rd, "\mu"'] &                                     & A^{\overline{q}}_j\otimes A^{\overline{p}}_i \arrow[ld, "\mu"] \\
                                                                                    & A^{\overline{p}+\overline{q}}_{i+j}. &                                                               
        \end{tikzcd}\]
        where $\tau$ is the usual twisting isomorphism for chain complexes
        \[\begin{array}{rrcl}
        \tau: & A_\ast \otimes B_\ast & \to & B_\ast\otimes A_\ast \\
        ~& a\otimes b & \mapsto & (-1)^{\Real{a}\Real{b}}b\otimes a. \\
        \end{array}\]    
    \end{defi}
    \begin{rem}
        Note that having a perverse graded algebra is equivalent to asking for chain maps, $\mu: (A\boxtimes A)^{\overline{r}}\to A^{\overline{r}}$, for every $\overline{r}\in\pGM$, that are compatible with the poset structure of $\pGM$ and verify the associativity property.
    \end{rem}
    \begin{rem}
        To be precise, we have defined a perverse differential \emph{lower} graded $R$-algebra since we consider chain complexes. Similarly, we can define a perverse differential \emph{upper} graded $R$-algebra when we work on cochain complexes.
    \end{rem}
    \begin{exemple}
        Let $X$ be a filtered space. The blown-up of Sullivan's polynomial forms $\widetilde{A}^{\ast}_{PL,\overline{p}}(X)$ is a pCDGA, see \cite[Section 2.1]{CST18ration}.
    \end{exemple}
    \begin{exemple}\label{ex:tensoralg}
        Let $M^\bullet$ be a perverse graded module. The \emph{tensor algebra} of $M^\bullet$ is the perverse graded algebra (pDGA with trivial differential) \[(TM)^\bullet:= \bigoplus_{k\geq 0} (T^k M)^\bullet\]
        with $(T^0 M)^\bullet = F_{\overline{0}}(\mathbb{S}^0)^\bullet$ and $(T^k M)^\bullet=(M^{\boxtimes k})^\bullet$ for $k\geq 1$. 
        
        The degree of an element $m_1\otimes m_2 \otimes \ldots \otimes m_k$ in $(T^k M)^\bullet_i$ is $i=\sum_{i=1}^k \Real{m_i}$. For $\overline{p},\overline{q}\in \pGM$ and $i,j\in\mathbb{Z}$, the multiplication 
        \[\mu: (TM)_i^{\overline{p}} \otimes (TM)_j^{\overline{q}} \to (TM)_{i+j}^{\overline{p}\oplus\overline{q}}\] is given by the tensor product.
    \end{exemple}
    By a similar process, we adapt the notion of differential graded $A$-module to the context of perverse objects.
    \begin{defi}
        Let $(A^\bullet, d_A^\bullet)$ be a perverse DGA. A left \emph{perverse differential graded $(A^\bullet, d^\bullet_A)$-module} (pDG module) is a perverse chain complex of $R$-modules $(M^\bullet, d^\bullet_M)$ with $R$-linear maps 
        \[\begin{array}{rcl}
         A^{\overline{p}}_i\otimes M^{\overline{q}}_j    & \to & M^{\overline{p}\oplus \overline{q}}_{i+j} \\
           a \otimes m  & \mapsto & a.m
        \end{array} \]
        for every $\overline{p}, \overline{q}\in \pGM$ and $i, j\in \mathbb{Z}$ such that 
        \[d_M(a.m)=d_A(a).m+(-1)^{\Real{a}}a.d_M(m).\] 
        The category of perverse differential graded $A^\bullet$-modules is denoted $\pMod(A)$.
   
    \end{defi}
    \begin{rem}
    \leavevmode
        \begin{itemize}
        \item Analogously, one defines \emph{right perverse differential graded $(A^\bullet, d^\bullet_A)$-module}.
        \item A \emph{perverse differential graded $(A^\bullet, d^\bullet_A)$-bimodule} is a left perverse differential graded $((A^e)^\bullet:=(A\boxtimes A^{op})^\bullet, d^\bullet_A \otimes 1 +1 \otimes d^\bullet_A)$-module where $(A^{op})^\bullet$ is equipped with the opposite multiplication of $A^\bullet$, i.e.\ $\mu^{op}(a\otimes b)=(-1)^{\Real{a}.\Real{b}}\mu(b\otimes a)$. It is a perverse chain complex of $R$-modules $(M^\bullet, d^\bullet_M)$ with left and right pDG $A^\bullet$-module structures such that
        \[d_M(a.m.b)=d_A(a).m.b+(-1)^{\Real{a}}a.d_M(m).b+(-1)^{\Real{a}+\Real{m}}a.m.d_B(b)\]
        for any $a,b\in A^\bullet$ and $m\in M^\bullet.$
        We denote by $\pMod((A^e)^\bullet)$ the category of perverse differential $A^\bullet$-bimodules.
        \end{itemize}    
    \end{rem}
    The category $\pMod(A)$ has an internal Hom functor which we define below.
    \begin{defi}
    Let $A^\bullet$ be a pDGA and consider two pDG $A^\bullet$-modules $M^\bullet, P^\bullet$. The perverse chain complex $\Hom_A(M,P)^\bullet$ is the perverse subcomplex of the internal Hom, $\Hom_{(Ch(R))^{\pGM}}(M, P)^\bullet$, made of maps that commute with the action of $A^\bullet$. In other words, a map $f : M^\bullet \to P^\bullet$ of degree $\Real{f}$ of $\Hom_{(Ch(R))^{\pGM}}(M, P)^\bullet$ is in $\Hom_A(M,P)^\bullet$, if 
    \[f(a.m) = (-1)^{\Real{f}\Real{a}}a.f(m)~ \forall a \in A^\bullet,\, m\in M^\bullet.\]
    Similarly, we define $(M \boxtimes_A P)^\bullet$ the tensor product over $A^\bullet$ when $M^\bullet$ (resp. $P^\bullet$) is a right (resp. left) pDG $A^\bullet$-module. It is generated by the simple tensors $m\otimes p$ of $M\boxtimes P$ such that
    \[m.a\otimes p = m\otimes a.p ~\forall a \in A^\bullet.\]
    In other words, $(M\boxtimes_A P)^\bullet$ is the coequalizer of $f,g:(M\boxtimes A \boxtimes P)^\bullet \to (M\boxtimes P)^\bullet$ where $f(m\otimes a \otimes p)=m.a\otimes p$ and $g(m\otimes a \otimes p)=m\otimes a.p$. 
    \end{defi}
    The results given at the end of subsection \ref{subsect: interhom} concerning algebraic structures found on the blown-up intersection complex can be stated in the following terms.
    \begin{proposition}
        Let $X$ be a stratified space. The blown-up intersection complex $(\widetilde{N}_\bullet^\ast(X; R), \cup)$ is perverse differential (upper) graded algebra.\\
        The intersection chain complex $I^\bullet C_\ast(X;R)$ is a right perverse differential graded $\widetilde{N}_\bullet^\ast(X; R)$-module for the cap product. 
    \end{proposition}

    The integration map is compatible with the product structure.
    \begin{proposition}[{\cite[Proposition 2.33]{CST18ration}}]\label{prop:integ_polyforms}
    Let $X$ be a stratified space of formal dimension $n$. The integration map
    \[\int\colon \widetilde{A}^{\ast}_{PL,\bullet}(X)\to \widetilde{N}^{\ast}_{\bullet}(X; \mathbb{Q})\]
    is a quasi-isomorphism of perverse chain complexes. Furthermore, there exist a pDGA, $(\widetilde{A_{PL}\otimes C})^\ast_\bullet(X)$, and quasi-isomorphisms of pDGAs, $f_1$ and $f_2$, such that the following diagram commutes in homology
    \[
    \begin{tikzcd}
    {\widetilde A^\ast_{PL, \bullet}(X)} \arrow[rr, "\int"] \arrow[rd, "f_1"'] &                                                                              & \widetilde N^\ast_\bullet(X; \mathbb{Q})  \arrow[ld, "f_2"']\\
                                                                                       & (\widetilde{A_{PL}\otimes C})^\ast_\bullet(X). &                                         
    \end{tikzcd}
    \]
    \end{proposition}

    The next result follows from Poincaré duality (Theorem \ref{thm:Poincaré}). 
    \begin{proposition}
    Let $X$ be an $n$-dimensional, second countable, compact and oriented pseudomanifold. Then, there exists a quasi-isomorphism of perverse (right) $\widetilde N^\ast_{\overline{\bullet}}(X;R)$-modules:
    \[\widetilde N^\ast_{\bullet}(X;R) \to \Hom(I_{\bullet}C^{n-\ast}(X;R), R).\]
    It is obtained as the composite of
    \[\widetilde N^\ast_{\bullet}(X;R) \xrightarrow[\simeq]{DP_X} I^{\bullet}C_{n-\ast}(X;R) \xrightarrow[\simeq]{Bid} \Hom(I_{\bullet}C^{n-\ast}(X;R), R)\]
    where
    \begin{itemize}
        \item $DP_X$ is the cap product with a fundamental cycle
        \item and $Bid$ is the injection in the bidual given in \emph{\cite[Proposition A]{CST20Deligne}}.
    \end{itemize}
    \end{proposition}

    
    Recall that, when we work on a field, we have a quasi-isomorphism between the blown-up cochain complex and the dual intersection cochain complex (Theorem \ref{thm:field_intercoho_iso}). The previous result can be stated in the following terms.
    
    \begin{proposition}\label{prop:dual-qiso-bup}
    Let $X$ be an $n$-dimensional, compact, second countable and oriented pseudomanifold. Then, there exists a quasi-isomorphism of perverse (right) $\widetilde N^\ast_{\overline{\bullet}}(X;\mathbb{F})$-modules:
    \[
    \Dual_N:\widetilde N^\ast_{\bullet}(X;\mathbb{F}) \to \mathbb{D}\widetilde N^\bullet_\ast(X;\mathbb{F})\]
    where $\mathbb{D}\widetilde N^\bullet_\ast(X;\mathbb{F}):=\Hom(\widetilde N^{n-\ast}_{\overline{t}-\bullet}(X;\mathbb{F}), \mathbb{F})$.
    \end{proposition}
    Note that $\mathbb{D}\widetilde N^\bullet_\ast(X;\mathbb{F})$ is equal to the linear dual $D\widetilde N^\bullet_\ast(X;\mathbb{F})$ by a shift of $n$ in degree (Example \ref{ex:pmonoidal_chcplx}).
    The image of the unit $1\in\widetilde N^0_{\overline{0}}(X; \mathbb{F})$ by $\Dual_N$ is denoted $\Gamma_X$, it is given by 
    \begin{align*}
        \Gamma_X: \widetilde N_{\overline{t}}^n(X;\mathbb{F}) &\to \mathbb{F}   \\
          \alpha &\mapsto \Inter(1)(\alpha \cap \zeta)
    \end{align*}
    with $\zeta\in I^{\overline{0}}C_n(X;R)$ a fundamental cycle.

\section{Hochschild cohomology for perverse DGA}\label{sect:Hoch-coho}
    We would like to define the Hochschild (co)homology for perverse differential graded algebras. Standard references for Hochschild cohomology are \cite{Lod98} and \cite[Chapter X]{MacL78}. In the survey paper \cite{Abb15}, the Hochschild (co)chain complex of a differential graded algebra is given via the \emph{two-sided bar construction} and using derived functors. We're going to extend this approach to pDGAs in this section. 
    
    Throughout this section, we will consider perverse cochain complexes (objects with upper grading). 
    
    \subsection{Hochschild (co)homology via bar construction}\label{subsect:Hoch_via_bar}
    Let $A_\bullet$ be a perverse differential (upper) graded algebra. It is unital, so it is equipped with a morphism $\eta:F_{\overline{0}}(\mathbb{S}^0)_\bullet\to A_\bullet$. We fix $\overline{A}_\bullet=A_\bullet/\eta(F_{\overline{0}}(\mathbb{S}^0))_\bullet$. 
    
    Since $A_\bullet$ is a monoid in the abelian monoidal category of perverse chain complexes $(Ch(R))^{\pGM}$, we can consider its \emph{normalized two-sided bar construction} $B(A, A, A)_\bullet$. A good reference for the general case is \cite{Zha19}. We give an explicit description below. Our construction will be slightly different from the one given by Zhang as we take the suspension but both complexes are isomorphic. Recall that for a cochain complex $C^\ast$, the suspension corresponds to a shift in degree by -1: $(sC)^k=C^{k+1}$ for any $k\in\mathbb{Z}$.   

    \begin{defi}
    The (normalized) \emph{two-sided bar complex} is the perverse cochain complex given by
    \[\mathbb{B}(A)_\bullet:=(A\boxtimes T(s\overline{A}) \boxtimes A)_\bullet\]
    where $T(s\overline{A})$ is the tensor algebra of $s\overline{A}_\bullet$ (Example \ref{ex:tensoralg}).
    
    More precisely, the component of perverse degree $\overline{r}\in P_X$ and degree $q$ is given by
    \[\mathbb{B}(A)^q_{\overline{r}}=\bigoplus_{k\in \mathbb{N}} \mathbb{B}_k(A)^q_{\overline{r}}\]
    where \[\mathbb{B}_k(A)^q_{\overline{r}}=(A \boxtimes (s\overline{A})^{\boxtimes k} \boxtimes A)^q_{\overline{r}}\] 
    \[=\colim_{\overline{r}_0+\ldots +\overline{r}_{k+1}\leq\overline{r}} \bigoplus_{i_0+\ldots+i_{k+1}=q}A^{i_0}_{\overline{r}_0}\otimes (s \overline{A})^{i_1}_{\overline{r}_1}\otimes \ldots \otimes  (s \overline{A})^{i_k}_{\overline{r}_k}\otimes A^{i_{k+1}}_{\overline{r}_{k+1}}.\] 
    The element $a\otimes s(a_1)\otimes s(a_2)\otimes\ldots\otimes s(a_k)\otimes b$ from $\mathbb{B}_k(A)^q_{\overline{r}}$ is denoted $a[a_1|a_2|\ldots|a_k]b$ and is said to be of \emph{length} $k$ and of \emph{degree} 
    \[q=\Real{a}+\Real{b}+ \sum_{i=1}^{k}\Real{s(a_i)}=\Real{a}+\Real{b}+ \sum_{i=1}^{k}\Real{a_i}-k.\]
    We denote by $\mathbb{B}_k(A)_\bullet$ the perverse cochain complex
    whose degree $q$ component is $\mathbb{B}_k(A)^q_\bullet$ (and with differential $d_0$ given below).
    
    The differential on $\mathbb{B}(A)$, $D=d_0+d_1$, is defined by an internal (or vertical) differential
    \[d_0(a[a_1|a_2|\ldots|a_k]b)=d(a)[a_1|a_2|\ldots|a_k]b-\sum_{i=1}^{k}(-1)^{\eps_i}a[a_1|a_2|\ldots|d(a_i)|\ldots|a_k]b\] 
    \[+(-1)^{\eps_{k+1}}a[a_1|a_2|\ldots|a_k]d(b),\]
    and an external (or horizontal) differential
    \[d_1(a[a_1|\ldots|a_k]b)=(-1)^{\Real{a}}aa_1[a_2|\ldots|a_k]b+\sum_{i=2}^{k}(-1)^{\eps_i}a[a_1|a_2|\ldots|a_{i-1}a_i|\ldots|a_k]b\] 
    \[-(-1)^{\eps_{k+1}}a[a_1|a_2|\ldots|a_{k-1}]a_kb,\]
    where $\eps_i=\Real{a}+\sum_{j<i}\Real{s(a_j)}$.
    Notice that since we took the suspension of $A$, we have $\degr(D)=\degr(d_0+d_1)=1$. 
    \end{defi}

    \begin{defi}
        The (normalized) \emph{Hochschild chain complex} of $A_{\bullet}$ with coefficients in a pDG $A_{\bullet}$-bimodule $M_{\bullet}$ is \[HC^{\bullet}_\ast(A, M):=(M\boxtimes_{A^e}\mathbb{B}(A))^\ast_{\bullet}\simeq (M\boxtimes T(s\overline{A}))^\ast_{\bullet},\]
        equipped with a degree 1 differential $D_\ast=d_0+d_1$ which is defined by
    \[d_0(m[a_1|\ldots|a_k])=d_M(m)[a_1|a_2|\ldots|a_k]-\sum_{i=1}^{k}(-1)^{\eps_i}m[a_1|a_2|\ldots|d_A(a_i)|\ldots|a_k],\] 
    \[d_1(m[a_1|\ldots|a_k])=(-1)^{\Real{m}}ma_1[a_2|\ldots|a_k]+\sum_{i=2}^k(-1)^{\eps_i}m[a_1|a_2|\ldots|a_{i-1}a_i|\ldots|a_k]\] 
    \[-(-1)^{\eps_k\Real{s(a_k)}}a_km[a_1|a_2|\ldots|a_{k-1}],\]
    where $\eps_i=\Real{m}+\sum_{j< i}\Real{s(a_j)}.$ 
    Its homology $HH^{\bullet}(A, M)$ is the \emph{Hochschild homology} of $A_{\bullet}$ with coefficients in $M_{\bullet}$.
    \end{defi}
    
    \begin{defi}\label{def:Hochcocplx}
    The (normalized) \emph{Hochschild cochain complex} of $A_{\bullet}$ with coefficients in a pDG $A_{\bullet}$-bimodule $M_{\bullet}$ is \[HC^\ast_{\bullet}(A, M):=\Hom_{A^e}(\mathbb{B}(A), M)^\ast_{\bullet}\simeq \Hom_{(Ch(R))^{P_X}}(T(s \overline{A}),M)^\ast_{\bullet}. \]
    Explicitly, the degree $q$ and perversity $\overline{r}$ component $HC^q(A, M)_{\overline{r}}$ is given by 
    \[\Hom_{(Ch(R))^{P_X}}(T(s \overline{A}),M)^q_{\overline{r}}=\lim_{\overline{r}\leq \overline{q}-\overline{p}} \prod_{j-i=q} \Hom_R(T(s \overline{A})^i_{\overline{p}}, M^j_{\overline{q}}).\]
    A cochain $\phi\in HC_\bullet(A, M)$ will be identified with the unique element $f$ in the perverse chain complex $\Hom_{(Ch(R))^{P_X}}(T(s \overline{A}),M)_\bullet$ which verifies 
    \[\phi(a_0[a_1|\ldots|a_k]a_{k+1})=a_0f([a_1|\ldots|a_k])a_{k+1}\]
    for any $a_0[a_1|a_2|\ldots|a_k]a_{k+1}$ in $\mathbb{B}_k(A)$.
    We denote the degree of a cochain $f\in HC^\ast_\bullet(A, M)$ by $\Real{f}$. The differential is given by $D^\ast=d_0+d_1$ where
    \[d_0(f)([a_1|\ldots| a_k])=d_Mf([a_1|\ldots|a_k])+\sum_{i=1}^{k}(-1)^{\eps_i+\Real{f}}f([a_1|\ldots|d_A(a_i)|\ldots|a_k]),\] 
    \[d_1(f)([a_1|\ldots|a_k])=-(-1)^{(\Real{a_1}+1)\Real{f}}a_1f([a_2|\ldots|a_k])+(-1)^{\eps_{k}+\Real{f}}f([a_1|\ldots|a_{k-1}])a_k\] \[-\sum_{i=2}^{k}(-1)^{\eps_i+\Real{f}}f([a_1|\ldots|a_{i-1}a_i|\ldots|a_k]),\]
    with $\eps_i=\sum_{j<i}\Real{s(a_j)}.$ Its homology $HH_\bullet(A, M)$ is the \emph{Hochschild cohomology} of $A$ with coefficients in $M$.
    \end{defi}
    \textbf{Notation} We will use the notations $HC_\ast^\bullet(A)$ and $HC^\ast_\bullet(A)$ to refer to the Hochschild chain and cochain complexes of $A_\bullet$ i.e. when $M_\bullet=A_\bullet$. 

    \textbf{Action on Hochschild cohomology}
    We end this subsection by describing the left action of $HH^\ast_\bullet(A)$ on $HH^\ast_\bullet(A,M)$ where $A_\bullet$ is a pDGA and $M_{\bullet}$ is a pDG $A_{\bullet}$-bimodule. First, notice that for any $\overline{p}, \overline{q}\in \pGM$, we have a map of cochain complexes of degree 0 and perverse degree $\overline{0}$
    \[-\boxtimes_A -:HC^\ast_{\overline{p}}(A)\otimes HC^\ast_{\overline{q}}(A,M)\to HC^\ast_{\overline{p}+\overline{q}}(A,A\boxtimes_A M).\]
    For $f\in HC^i_{\overline{p}}(A)$ and $g\in HC^j_{\overline{p}}(A,M)$, let $f\boxtimes_A g$ be the perverse cochain in $HC^{i+j}_{\overline{p}+\overline{q}}(A,A\boxtimes_A M)$ given by
    \[f\boxtimes_A g([a_1|\ldots |a_l])=\sum_{k=0}^l (-1)^{\Real{g}(\Real{a_1}+\ldots + \Real{a_k}-k)}f([a_1|\ldots|a_k])\boxtimes_A g([a_{k+1}|\ldots|a_l]).\]
    This map is natural with respect to $M$. Furthermore, we have an isomorphism of perverse $A_\bullet$-bimodules which is given by the action of $A_\bullet$ on $M_\bullet$:
    \begin{align*}
        (A\boxtimes_A M)_\bullet &\simeq M_\bullet \\
        a\boxtimes_A m &\mapsto a.m
    \end{align*}
    The map induced in homology by the composite of both maps gives the sought action.

    \subsection{Gerstenhaber algebra structure}
    In \cite{Ger63}, Gerstenhaber shows that the classical Hochschild cohomology of an algebra is equipped with a Gerstenhaber algebra structure. This result can be extended to the case of pDGA and we get the following result. 
    
    \begin{letterthm}\label{thm:Hoch_is_Gerst}
    The Hochschild cohomology $HH^\ast_\bullet(A)$ of a pDGA $A_\bullet$ is a \emph{perverse Gerstenhaber algebra}. More precisely, it is equipped with two products
    \begin{itemize}
    \item the \emph{cup product}: $-\cup -: HH^m_\bullet(A) \otimes HH^p_\bullet(A) \to HH^{m+p}_\bullet(A)$
    \item the \emph{Gerstenhaber bracket}: $[-,-]: HH^m_\bullet(A) \otimes HH^p_\bullet(A) \to HH^{m+p-1}_\bullet(A)$
    \end{itemize}    
    of perverse degree $\overline{0}$ such that 
    \begin{enumerate}[label=\alph*)]
        \item $\cup$ is an associative and graded commutative product,
        \item the suspended Hochschild cohomology $(sHH^\ast_\bullet(A), [-,-])$ is a Lie algebra. For $f,g,h\in HH^\ast_\bullet(A)$, the following identities hold
        \begin{itemize}
        \item graded skew-commutativity: $[f,g]=-(-1)^{(\Real{f}-1)(\Real{g}-1)}[g,f]$,
        \item Jacobi identity: $[[f,g],h]=[f,[g,h]]-(-1)^{(\Real{f}-1)(\Real{g}-1)}[g,[f,h]]$,
         \end{itemize}   
    \item and both operations satisfy the Leibniz rule: 
    \[[f, g\cup h]=[f,g]\cup h + (-1)^{(\Real{f}-1)\Real{g}}g\cup [f,h].\]
    In other words, for any cocycle $f$, the operation $[f,-]: HH^\ast_\bullet(A) \to HH^{\ast+\Real{f}-1}_\bullet(A)$ is a derivation with respect to $\cup$.
    \end{enumerate}
    \end{letterthm}

    This theorem is a consequence of a series of results which we state in this subsection. Let's now give explicit descriptions of these products. 
   
    \begin{defi}\label{def:cup_prod}
    The \emph{cup product} is an associative product given for $f,g\in HC^\ast_\bullet(A)$ by 
    \[f\cup g[a_1|\ldots|a_k]:=\sum_{i=1}^{k-1}(-1)^{\Real{g}\overline{\eps}_i}f[a_1| \ldots | a_{i}]g[a_{i+1} | \ldots |  a_{k}]\]    
    where $\overline{\eps_i}=\sum_{j \leq i} \Real{s(a_j)}=\sum_{j \leq i} (\Real{a_j}-1)$.
    \end{defi}   

    \begin{defi}
    The \emph{Gerstenhaber bracket} on $HC^\ast(A)_\bullet$ is given for two cochains $f$ and $g$ by 
    \[[f,g]:=f\circ g - (-1)^{(\Real{f}-1)(\Real{g}-1)}g\circ f\]
    with
    \[f\circ g([a_1|\,\ldots\, | a_k]):=\sum_{0 \leq i < j \leq k} (-1)^{(\Real{g}-1)\overline{\eps_i}}f[a_1|\,\ldots\,| a_i| g[a_{i+1}|\,\ldots\,| a_j]| a_{j+1}|\,\ldots\,| a_k]\]
    where $\overline{\eps_i}=\sum_{j \leq i} \Real{s(a_j)}=\sum_{j \leq i} (\Real{a_j}-1)$ as before. It is clear that the bracket is graded skew-commutative.    
    \end{defi}
      
    The cup product and the Gerstenhaber bracket are particular cases of the \emph{brace operator} which is introduced in \cite{Get93}. We give a description of it below. 
    In what follows, a Hochschild chain $[a_1|\ldots|a_m]$ will be denoted $[a_{1,m}]$.

    \begin{defi}
    To each family $f_0,\ldots, f_k \in HC^\ast_\bullet(A)$ of Hochschild cochains, the \emph{brace operator} associates a new Hochschild cochain $f_0\{f_1,\ldots, f_k\}$ which is defined by
    \[\begin{aligned}
    f_0\{f_1,\ldots, f_k\}[a_{1,m}]:= &  \sum_{0\leq i_1 < j_1 \leq \ldots \leq i_k < j_k \leq m} (-1)^{\eps_{i_1}(\Real{f_1}-1)+\ldots+\eps_{i_k}(\Real{f_k}-1)}\\
      ~   & f_0[a_{1,i_1}|f_1[a_{i_1+1,j_1}]|a_{j_1+1}|\ldots|a_{i_k}|f_k[a_{i_k+1,j_k}]|a_{j_k+1,m}]
    \end{aligned} \]
    with $\eps_i=\sum_{l=1}^i \Real{s(a_l)}$.
    \end{defi}
    \textbf{Interpretation of products and differential using the brace operator.}
    We remark that $f\circ g=f\{g\}$ and $f\cup g = (-1)^{\Real{f}}m\{f,g\}$ where the map
    \[m:T(s\overline{A})_\bullet\to A_\bullet \text{ is defined by }
    m([a_1| \ldots| a_k])=\begin{cases}
        0 & \text{if } k \neq 2 \\
        (-1)^{\Real{a_1}}a_1.a_2 & \text{if } k=2. 
    \end{cases}\]
    Although $m$ is not a Hochschild cochain, the cochain $m\{f,g\}$ is well defined. Furthermore, notice that
    \[\Real{m[a_1|a_2]}=\Real{a_1}+\Real{a_2}=\Real{s(a_1)}+\Real{s(a_2)}+2.\]
    This is why we fix $\Real{m}=2$. \\
    We can also interpret the differential on $HH^\ast_\bullet(A)$ using the bracket. Indeed, recall that the differential is given by $D^\ast=d_0+d_1$ where
    \[d_0(f)([a_1|\ldots| a_k])=d_Af([a_1|\ldots|a_k])+\sum_{i=1}^{k}(-1)^{\eps_i+\Real{f}}f([a_1|\ldots|d_A(a_i)|\ldots|a_k]),\] 
    \[d_1(f)([a_1|\ldots|a_k])=-(-1)^{(\Real{a_1}+1)\Real{f}}a_1f([a_2|\ldots|a_k])+(-1)^{\eps_{k}+\Real{f}}f([a_1|\ldots|a_{k-1}])a_k\] \[-\sum_{i=2}^{k}(-1)^{\eps_i+\Real{f}}f([a_1|\ldots|a_{i-1}a_i|\ldots|a_k]),\]
    with $\eps_i=\sum_{j<i}\Real{s(a_j)}.$ One can see that $D^\ast f=[d_A,f]+[m,f]$ 
    \[\text{where } d_A:T(s\overline{A})_\bullet\to A_\bullet \text{ is given by }
    d_A([a_1| \ldots| a_k])=\begin{cases}
        0 & \text{if } k \neq 1 \\
        d_A(a_1) & \text{if } k=1. 
    \end{cases}\]  
    We have $\Real{d_A[a_1]}=\Real{a_1}+1=\Real{s(a_1)}+2.$ So we fix $\Real{d_A}=2$, even tough $d_A$ is not a cochain. \\    
    The fact that $HH^\ast_\bullet(A)$ is a perverse Gerstenhaber algebra is a consequence of identities verified by the brace operator which are given in \cite[Section 5]{GJ94} and \cite{GV95}.
    \begin{proposition}\label{prop:highJacobi}
    The brace operator satisfies the \emph{higher pre-Jacobi identities}:
    for any family of Hochschild cochains $\phi, f_1,\ldots, f_k, g_1, \ldots, g_l \in HC^\ast_\bullet(A)$ we have
    \[\begin{aligned}
    \phi\{f_1,\ldots, f_k\}\{g_1,\ldots, g_l\}= &  \sum_{0\leq i_1 \leq j_1 \leq \ldots \leq i_k \leq j_k \leq l} (-1)^{\eta_{i_1}(\Real{f_1}-1)+\ldots+\eta_{i_k}(\Real{f_k}-1)}\\
      ~   & \phi\{g_{1,i_1},f_1\{g_{i_1+1,j_1}\},g_{j_1+1},\ldots,g_{i_k},f_k\{a_{i_k+1,j_k}\},g_{j_k+1,l}\}
    \end{aligned} \]
    where $\eta_i=\sum_{i'=1}^i \Real{s(g_{i'})}$, $f\{g_{i+1,j}\}:=f\{g_{i+1},\ldots, g_j\}$  and with the convention $f_n\{g_{i_n+1, i_n}\}=f_n$.
    \end{proposition}
    
    \begin{exemple}\label{ex:highJacobi}
            \begin{itemize}
            \item For $k=1, \,l=2$, we have
            \[\begin{aligned}
            \phi\{f\}\{g,h\}=&\phi\{f,g,h\}+\phi\{f\{g\},h\}+\phi\{f\{g,h\}\}+(-1)^{(\Real{f}-1)(\Real{g}-1)}\phi\{g,f,h\}\\
            ~\               &+(-1)^{(\Real{f}-1)(\Real{g}-1)}\phi\{g,f\{h\}\}+(-1)^{(\Real{f}-1)(\Real{g}+\Real{h})}\phi\{g,h,f\}.
            \end{aligned} \]
            \item For $k=2, \,l=1$, we get
            \[\begin{aligned}
            \phi\{f,g\}\{h\}=&\phi\{f,g,h\}+\phi\{f,g\{h\}\}+(-1)^{(\Real{g}-1)(\Real{h}-1)}\phi\{f,h,g\}\\
            ~\               &+(-1)^{(\Real{g}-1)(\Real{h}-1)}\phi\{f\{h\},g\}+(-1)^{(\Real{f}+\Real{g})(\Real{h}-1)}\phi\{h,f,g\}.
            \end{aligned} \]           
        \end{itemize}    
    \end{exemple}
    
    The following results are stated in \cite{GV95} and follow from the higher pre-Jacobi identities and straightforward computations. They also hold for the Hochschild cochains of a pDGA as the perversities don't appear in our definitions of the products and of the differential. One difference we have with Gerstenhaber and Voronov is that we also have an internal differential but those additional terms cancel between themselves.
    \begin{proposition}
    The cup product and the Gerstenhaber bracket are chain maps i.e.\ for $f,g\in HC^\ast_\bullet(A)$ we have 
    \[D^\ast (f \cup g)=D^\ast(f) \cup g + (-1)^{\Real{f}}f \cup D^\ast(g)\]
    and
    \[D^\ast([f, g]) = [D^\ast(f), g] + (-1)^{\Real{f}+1}[f, D^\ast(g)].\]
    This implies that the cup product and the Gerstenhaber bracket are well defined on the Hochschild cohomology $HH^{\ast}_\bullet(A)$.
    \end{proposition}

    The graded commutativity of the cup product is a consequence of the next result.
    \begin{proposition}
    For any cochains $f,g\in HC^\ast(A)$, we have 
    \[D^\ast(f\circ g) - D^\ast(f)\circ g - (-1)^{\Real{f}+1}f \circ D^{\ast}(g)  = (-1)^{\Real{g}-1}(g\cup f-(-1)^{\Real{f}\Real{g}}f\cup g).\]
    \end{proposition}
    We now move on to the Gerstenhaber bracket.
    \begin{proposition}
     Let $A_\bullet$ be a pDGA. The suspended Hochschild cochain complex equipped with the Gerstenhaber bracket $(sHC^\ast_\bullet(A), [-,-])$ is a Lie algebra.   
    \end{proposition}
    
    We detail the proof of the Leibniz rule.
    \begin{proposition}
    For $f,g,h \in HH^\ast_\bullet(A)$, we have the following equality:
    \[[f, g\cup h]=[f,g]\cup h + (-1)^{(\Real{f}-1)\Real{g}}g\cup [f,h].\]
    \end{proposition}
    \begin{proof}
    Let $f,g,h \in HH^\ast_\bullet(A)$. Recall that $f\cup g=(-1)^{\Real{f}}m\{f,g\}$. We compute each of the terms.
    \begin{itemize}
        \item We begin with the left hand side    
        \[[f, g\cup h]=(-1)^{\Real{g}}f\{m\{g,h\}\}-(-1)^{(\Real{g}+\Real{h}-1)(\Real{f}-1)+\Real{g}}m\{g,h\}\{f\}.\]
    By applying proposition \ref{prop:highJacobi} for $k=1$ and $l=2$, we get:
        \begin{align*}
            [f, g\cup h]&=(-1)^{\Real{g}}f\{m\{g,h\}\}-(-1)^{\Real{g}\Real{f}}m\{g\{f\},h\} \\
                        &-(-1)^{(\Real{g}+\Real{h}-1)(\Real{f}-1)+\Real{g}}m\{g,h\{f\}\}.
        \end{align*}
        \item We move on to the right hand side
        \[[f,g]\cup h= (-1)^{\Real{f}+\Real{g}-1}m\{f\{g\},h\}-(-1)^{\Real{f}\Real{g}}m\{g\{f\},h\}.\]
        \item Finally, we compute the last term
    \[(-1)^{(\Real{f}-1)\Real{g}}g\cup [f,h]=(-1)^{\Real{g}\Real{f}}m\{g,f\{h\}\}+(-1)^{\Real{h}+\Real{f}(\Real{g}+\Real{h}-1)}m\{g,h\{f\}\}.\]
    \end{itemize}
    We are left with proving that
    \[(-1)^{\Real{g}}f\{m\{g,h\}\}=(-1)^{\Real{f}+\Real{g}-1}m\{f\{g\},h\}+(-1)^{\Real{g}\Real{f}}m\{g,f\{h\}\}.\]
    The differential of a cochain $f\in HC^\ast(A)$ is given by $D^\ast f=[m,f]+[d_A,f]$. So, looking at the term $D^\ast(f\{g,h\})$ may provide some insight. \\
    By definition, we have
    \[\begin{aligned}
    D^\ast(f\{g,h\})=&m\{f\{g,h\}\}+(-1)^{\Real{f}+\Real{g}+\Real{h}} f\{g,h\}\{m\}+d_A\{f\{g,h\}\} \\
    ~&+(-1)^{\Real{f}+\Real{g}+\Real{h}}f\{g,h\}\{d_A\},
    \end{aligned}\]
    \[(D^\ast f)\{g,h\}=m\{f\}\{g,h\}+(-1)^{\Real{f}}f\{m\}\{g,h\}+d_A\{f\}\{g,h\}+(-1)^{\Real{f}}f\{d_A\}\{g,h\}\]
    \[f\{D^\ast g,h\}=f\{m\{g\},h\}+(-1)^{\Real{g}}f\{g\{m\},h\}+f\{d_A\{g\},h\}+(-1)^{\Real{g}}f\{g\{d_A\},h\}\]
    and
    \[f\{g,D^\ast h\}=f\{g,m\{h\}\}+(-1)^{\Real{h}}f\{g,h\{m\}\}+f\{g,d_A\{h\}\}+(-1)^{\Real{h}}f\{g,h\{d_A\}\}.\]    
    Using proposition \ref{prop:highJacobi}, we find 
    \[D^\ast(f\{g,h\})-(D^\ast f)\{g,h\}+(-1)^{\Real{f}}f\{D^\ast g,h\}+(-1)^{\Real{f}+\Real{g}-1}f\{g,D^\ast h\}\]
    \[=-m\{f\{g\},h\}-(-1)^{(\Real{f}-1)(\Real{g}-1)}m\{g,f\{h\}\}+(-1)^{\Real{f}-1}f\{m\{g,h\}\}\]
    This means that the term above is trivial when taking the Hochschild cohomology. This proves the result.
    \end{proof}

    This concludes the proof of Theorem \ref{thm:Hoch_is_Gerst}. In section \ref{sect:DPDA}, we will see how to endow this perverse Gerstenhaber algebra with a \emph{perverse Batalin-Vilkovisky algebra} structure (Definition \ref{def:pBV}).
        
    \subsection{Model structure on perverse chain complexes}
    We would like to interpret Hochschild (co)homology in terms of derived functors. To do so, we first present model structures on perverse chain complexes $(Ch(R))^{\pGM}$ and $\pMod(A)$ for $A_\bullet$ a pDGA. Hovey shows that the (projective) model category structure found on $Ch(R)$ can be lifted to $(Ch(R))^{\pGM}$. 
    \begin{thm}[{\cite[Theorem 3.1]{Hov09}}]\label{thm:modelcat_pchain}
        There is a model category structure on the category of perverse chain complexes $(Ch(R))^{\pGM}$ where a morphism $f:Z_\bullet\to Y_\bullet$ is weak equivalence or a fibration if and only if $f_{\overline{p}}:Z_{\overline{p}}\to Y_{\overline{p}}$ is so in $Ch(R)$ for all $\overline{p}\in \pGM$. 
    \end{thm} 
    In Appendix \ref{appendix:chara-cofib-field}, we give a characterization of cofibrant perverse chain complexes over a field.
    \begin{rem}
       By \cite[Rem 5.1.8]{Hov99}, the model category $(Ch(R))^{\pGM}$ is cofibrantly generated by 
       \[\mathcal{I}^{\pGM}:=\bigcup_{\overline{p}\in \pGM}F_{\overline{p}}\mathcal{I} \text{ and } \mathcal{J}^{\pGM}:=\bigcup_{\overline{p}\in \pGM}F_{\overline{p}}\mathcal{J}\] 
       where
       \[\mathcal{I}:= \{\mathbb{S}^{m-1}\hookrightarrow \mathbb{D}^m,\, m\in \mathbb{Z}\} \text{ and }\mathcal{J}:=\{0\to \mathbb{D}^m,\, m\in \mathbb{Z}\}. \]
    \end{rem}
        
    \begin{rem}\label{rem:quillen_adj_pch}
    This newly obtained model structure is \emph{monoidal} \cite[Theorem 3.3]{Hov09}. The reader may refer to \cite[Chapter 4]{Hov99} for a detailed presentation of the notion. This implies, in particular, that we have the following \emph{Quillen adjunctions} in $(Ch(R))^{\pGM}$ when $C$ is cofibrant:
    \begin{itemize}
        \item[--] $(C\boxtimes -)_\bullet\vdash \Hom_{Ch(R)^{\pGM}}(C,-)_\bullet$,
        \item[--] $(-\boxtimes C)_\bullet\vdash \Hom_{Ch(R)^{\pGM}}(C,-)_\bullet$.
    \end{itemize}
    Furthermore, if $D$ is fibrant $\Hom_{Ch(R)^{\pGM}}(-, D)$ preserves quasi-isomorphism between cofibrant objects. 
    \end{rem}
    
    Let $A_\bullet$ be a pDGA. We can endow the category of perverse DG $A_\bullet$-modules with a model category structure.

    \begin{thm}[{\cite[Theorem 3.4]{Hov09}}]
    Let $A_\bullet$ be a pDGA.
    \begin{enumerate}
        \item There exists a model category structure on $pMod(A)$ where a morphism is a weak equivalence (or a fibration) if it is so in the underlying category $(Ch(R))^{\pGM}$. 
        \item If $A_\bullet$ is commutative, this model category structure is monoidal.
    \end{enumerate} 
    \end{thm}
    \begin{rem}
    We denote by $\mathcal{I}^{\pGM}_A$ (resp. $\mathcal{J}^{\pGM}_A$) the image of $\mathcal{I}^{\pGM}$ (resp. $\mathcal{J}^{\pGM}$) by the functor \[A \boxtimes -:(Ch(R))^{\pGM} \to pMod(A).\] The model category of pDG $A^\bullet$-modules is a cofibrantly generated by $\mathcal{I}^{\pGM}_A$ and $\mathcal{J}^{\pGM}_A$.
    \end{rem}
    The following notions, which are adapated from \cite[Definition 7.1]{BMR14}, are well defined now that we have a model category structure on bimodules over a pDGA.
    \begin{defi}
        Let $A_\bullet$ be a pDGA and $M_\bullet$ a perverse DG $A^e_\bullet$-module. For any $N_\bullet\in pMod(A^e)$ we define the Tor and Ext functors by
        \[\Tor_{A^e}(M,N)^\ast_\bullet:=H((P \boxtimes_{A^e} N)^\ast_\bullet)\]
        and
        \[\Ext_{A^e}(M,N)^\ast_\bullet:=H(\Hom_{A^e}(P,N)^\ast_\bullet)\]
        where $P_\bullet\to M_\bullet$ is a cofibrant approximation of the DG $A_\bullet^e$-module $M_\bullet$.
    \end{defi}    
    We would like to show that Hochschild (co)homology can be defined using $\Tor$ and $\Ext$ functors. This amounts to proving that the two-sided bar construction of a pDGA $A_\bullet$ is a cofibrant object in the category of $A_\bullet$-bimodules. To do so, we follow the approach of \cite[Section 9.1]{BMR14} and characterize the cofibrant objects of $\pMod(A)$ in Theorem \ref{thm:char_cofib_Amod}. We now give a series of results that will prove this theorem.
   
    \begin{defi}
    A pDG $A_\bullet$-module $M_\bullet$ is $\emph{semi-projective}$ if its underlying perverse graded $A_\bullet$-module is projective and if $\Hom_A(M,Z)_\bullet$ is acyclic as a perverse chain complex for all acyclic pDG $A$-modules $Z_\bullet$.
    \end{defi}

    \begin{rem}
        For chain complexes, the notion of semi-projective corresponds to the $K$\emph{-projective} chain complexes of Spaltenstein \cite{Spa88}.
    \end{rem}
    
    \begin{defi}
    A injection $i:N_\bullet \to M_\bullet$ of DG $A_\bullet$-modules is a \emph{semi-projective extension} if $(M/N)_\bullet:=M_\bullet/N_\bullet$ is semi-projective.
    \end{defi}
    
    \begin{proposition}\label{prop:sproj-cofib}
    If a map $i : N_\bullet \to M_\bullet$ of pDG $A_\bullet$-modules is a semi-projective extension, then it is a cofibration. In particular, a semi-projective $A_\bullet$-module is cofibrant.    
    \end{proposition}
    
    \begin{proof}
Fix $P_\bullet = (M/N)_\bullet$, by hypothesis it is semi-projective. We must show that $i$ has the LLP with respect to acyclic fibrations. \\
We must find a lift $H$ in this commutative diagram of pDG $A_\bullet$-modules
\[\begin{tikzcd}
N_\bullet \arrow[r, "g"] \arrow[d, "i"']                & E_\bullet \arrow[d, "p"] \\
M_\bullet \arrow[r, "f"'] \arrow[ru, "{H\,? }", dashed] & B_\bullet              
\end{tikzcd}\]
where $p: E_\bullet \to B_\bullet$ is an acyclic fibration.
Since $P_\bullet$ is projective as a graded $A_\bullet$-module, the following short exact sequence is right split
\[\begin{tikzcd}
0 \arrow[r] & N_\bullet \arrow[r, "i", hook] & M_\bullet \arrow[r, "\pi", two heads] & P_\bullet \arrow[r] \arrow[l, dashed, bend left=49] & 0.
\end{tikzcd}\]
We work in an abelian category, by the splitting lemma, we have an isomorphism of $A_\bullet$-modules $M_\bullet\simeq N_\bullet \oplus P_\bullet$. 
Hence, the differential on $M_\bullet$ can be written as
\[d_M(x, y) = (d_N(x) + t(y), d_P(y))\] 
where $d_N$ (resp. $d_P$) denotes the differential on $N_\bullet$ (resp. $P_\bullet$) and $t:P_\bullet \to N_\bullet$ is a map of graded $A_\bullet$-modules of degree $+1$ and perverse degree $\overline{0}$.\\
Note that $t$ is not necessarily a map of pDG $A_\bullet$-modules but we have 
\[d_N\circ t + t\circ d_P = 0.\]
Indeed, for $(x,y)\in N_\bullet \oplus P_\bullet$, we have
\[d_M^2(x,y)=d\big(d_N(x) + t(y), d_P(y)\big)=(d_N\circ t (y) +t\circ d_P(y), 0).\]
Since $d_M^2(x,y)=0$, we get the stated relation.\\  
We write $f = f_1 + f_2$, where $f_1 : N_\bullet \to B_\bullet$ and $f_2 : P_\bullet \to B_\bullet$ are maps of pDG $A_\bullet$-modules, and we decompose $H = H_1 + H_2$ similarly. 
In order to make the two triangles commute in the above diagram, we need to have
\[H_1= H \circ i = g \text{ and } p\circ H_2 = f_2.\]
Furthermore, $H_2$ is a map of pDG $A_\bullet$-modules, this implies that \[d_E\circ H_2(x)= d_E \circ H(0,x) \underset{\text{map } H}{\overset{\text{chain}}{=}} H\circ d_M(0, x) = H(t(x), d_P(x)) = g\circ t(x) + H_2\circ d_P(x).\] 
Let's give a first approximation of $H_2$. Since $P_\bullet$ is a projective graded $A_\bullet$-module, there is a map $f_2' : P_\bullet \to E_\bullet$ of graded $A_\bullet$-modules, such that $p\circ f_2' = f_2$.
\[\begin{tikzcd}
                                                       & E_\bullet \arrow[d, "p", two heads] \\
P_\bullet \arrow[r, "f_2"'] \arrow[ru, "\exists f_2'", dashed] & B_\bullet                          
\end{tikzcd}\]
However, $f_2'$ is not a map of pDG $A_\bullet$-modules: we don't necessarily have 
\[d_E\circ f_2'=g\circ t + f_2'\circ d_P.\] 
We have to tweak $f_2'$. Let $k : P_\bullet \to E_\bullet$ be the graded $A_\bullet$-module map given by
\[k = d_E\circ f_2'- f_2'\circ d_P - g\circ t.\] 
Suppose that we have a degree $0$ map of graded $A_\bullet$-modules $l: P_\bullet \to (\ker p)_\bullet$ such that $d_E\circ l-l\circ d_P = k$. Then, we get
\[d_E\circ (f_2'-l) - (f_2'-l)\circ d_P = g\circ t.\]
Hence, the map $f_2'-l: M_\bullet \to E_\bullet$ is a map of pDG $A_\bullet$-modules and $p\circ (f_2'-l)=f_2$. Let's prove that such a map $l$ exists.
We first remark that $(\Ima k)_\bullet \subset (\ker p)_\bullet$. Since $p$ is a chain map, we have $d_B\circ p = p\circ d_E$ and as a consequence
\[p\circ k = d_B\circ p\circ f_2' - p\circ f_2'\circ d_P - p\circ g\circ t = d_B\circ f_2 - f_2\circ d_P - f_1\circ t=0.\]
The last equality is obtained from $d_B\circ f= f \circ d_M$. 
Hence, $k$ can be viewed as a map $P_\bullet \to Z_\bullet:=(\ker p)_\bullet$ of degree $+1$. Moreover, $k$ is a cycle of degree $+1$ in $\Hom_A(P, Z)^{\overline{0}}$ because 
\[d_E\circ k - (-1)^1 k\circ d_P = -d_E\circ f_2'\circ d_P - d_E\circ g\circ t + d_E\circ f_2'\circ d_P - g\circ t\circ d_P = 0.\]
The last equality follows from $d_N\circ t + t\circ d_P = 0$.
We have a short exact sequence of perverse chain complexes
\[Z_\bullet\hookrightarrow E_\bullet \xtwoheadrightarrow{p} B_\bullet\simeq (E/Z)_\bullet\]
where the last map is a quasi-isomorphism. By considering the induced long exact sequence in homology, we find that $Z_\bullet$ is an acyclic perverse chain complex. By hypothesis, $P_\bullet$ is semi-projective so $\Hom_A(P,Z)_\bullet$ is acyclic. Therefore, the cycle $k$ is also a boundary. Hence, there exists a degree $0$ map of graded  $A_\bullet$-modules $l: P_\bullet \to Z_\bullet \subset E_\bullet$ such that $d_E\circ l-l\circ d_P = k$. By taking $H_2 = f_2'-l$, the map $H=H_1 +H_2$ solves the lifting problem.
    \end{proof}
    We now adapt the notion of \emph{semi-free module} which appears in \cite[Section 2]{FHT95} and \cite[Chapter 6]{FHT01}. 
    \begin{defi}
        A perverse chain complex $Y_\bullet$ is \emph{free} if there exists a family of degree-wise free chain complexes $\{V_{\overline{p}}\}_{\overline{p}\in \pGM}$ such that
        \[Y_\bullet=\bigoplus_{\overline{p}\in \pGM} F_{\overline{p}}(V_{\overline{p}})_\bullet.\]
        A \emph{basis} of $Y_\bullet$ is given by the union of bases of the $V^m_{\overline{p}}$ for $m\in\mathbb{Z}, \overline{p}\in\pGM$.\\ 
        For $A_\bullet$ a pDGA, a perverse chain complex $M_\bullet$ is \emph{free as pDG $A_\bullet$-module} if it is isomorphic as a pDG $A_\bullet$-module to $(A \boxtimes K)_\bullet$ with $K_\bullet$ a free perverse chain complex. A \emph{basis} of the pDG $A_\bullet$-module $M_\bullet$ is a basis of $K_\bullet$. 
    \end{defi}
    
    \begin{defi}
        A pDG $A_\bullet$-module $M_\bullet$ is \emph{semi-free} if it there exists an increasing sequence of pDG $A_\bullet$-submodules 
        \[0=F_{-1, \bullet} \subset F_{0,\bullet}  \subset \ldots \subset F_{k,\bullet} \subset \ldots \]
        such that $M_\bullet=\bigcup_k F_{k,\bullet}$ and each $(F_k/F_{k-1})_\bullet$ is isomorphic as a pDG $A_\bullet$-module to $(A \boxtimes K_k)_\bullet$ where $K_{k,\bullet}$ is a free perverse chain complex with trivial differential.
    \end{defi}
       
    \begin{rem}\label{rk:splitfiltdecomp}
    Note that a semi-free pDG $A_\bullet$-module is free as a graded $A_\bullet$-module, that is to say as a perverse chain complex with an $A_\bullet$-module that doesn't necessarily respect the differential.
    Since $(F_k/F_{k-1})_\bullet$ is free as a pDG $A_\bullet$-module, the following short exact sequence of $A_\bullet$-modules splits
    \[F_{k-1,\bullet} \hookrightarrow F_{k,\bullet} \twoheadrightarrow (F_k/F_{k-1})_\bullet.\]
    By the splitting lemma, \[F_{k,\bullet} \simeq F_{k-1,\bullet} \oplus (F_k/F_{k-1})_\bullet \simeq F_{k-1,\bullet} \oplus (A \boxtimes K_k)_\bullet.\]
    By induction, we get an isomorphism of graded $A_\bullet$-modules \[F_{k,\bullet} \simeq (A \boxtimes (\bigoplus_{i=0}^k K_i))_\bullet.\] Since $M_\bullet=\bigcup_{k \geq 0} F_{k,\bullet}$, it is isomorphic as a graded $A_\bullet$-module to $(A_\bullet \boxtimes (\bigoplus_{i \geq 0} K_i))_\bullet$. 
    \end{rem}
    
    \begin{proposition}\label{prop:semifree-semiproj}
    If a pDG $A_\bullet$-module $M_\bullet$ is semi-free then it is semi-projective.
    \end{proposition}

    \begin{proof}
    Let $M_\bullet$ a semi-free pDG $A_\bullet$-module, then it has a filtration
    \[0=F_{-1,\bullet} \subset F_{0,\bullet}  \subset \ldots \subset F_{k,\bullet} \subset \ldots \]
    As explained in the remark above, $M_\bullet$ is isomorphic as a graded $A_\bullet$-module to 
    \[\bigoplus_{k \geq 0} (F_k/F_{k-1})_\bullet\simeq (A \boxtimes (\bigoplus_{k \geq 0} K_k))_\bullet.\] 
    More generally, for $k\in \mathbb{N}$, the pDG $A_\bullet$-module $(M/F_{k-1})_\bullet$ is isomorphic as a graded $A_\bullet$-module to 
    $\bigoplus_{i \geq k} (F_i/F_{i-1})_\bullet$. Hence, it is graded $A_\bullet$-free and the following short exact sequence splits
    \[0 \to (F_k/F_{k-1})_\bullet\to (M/F_{k-1})_\bullet \to (M/F_k)_\bullet\to 0.\]
    By applying the $\Hom$ functor, we get another short exact sequence
    \[0 \to \Hom_A(M/F_k,Z)_\bullet \to \Hom_A(M/F_{k-1},Z)_\bullet \to \Hom_A(F_k/F_{k-1}, Z)_\bullet \to 0.\]
    where $Z_\bullet$ is a pDG $A_\bullet$-module. 
    Suppose that $Z_\bullet$ is acyclic, we would like to show that the perverse chain complex $\Hom_A(M,Z)_\bullet$ is also acyclic. We follow the proof of \cite[Proposition 9.9]{BMR14}, we're going to show that each $\Hom_A(F_k/F_{k-1}, Z)_\bullet$ is acyclic.
    Consider the following diagram
    \[M_\bullet=(M/F_{-1})_\bullet\to (M/F_0)_\bullet \to (M/F_1)_\bullet \to \ldots\]
    Its colimit is $0$. Since the $\Hom$ functor sends colimits to limits we have
    \[\lim_k \Hom_A(M/F_k, Z)_\bullet \simeq \Hom_A (\colim_k M/F_k, Z)_\bullet =0.\]
    Let $k \in \mathbb{N}$, we have the following isomorphism of perverse chain complexes 
    \[\Hom_A(F_k/F_{k-1},Z)_\bullet\simeq\Hom_A(A \boxtimes K_k, Z)_\bullet\simeq \Hom_{(Ch(R))^{\pGM}}(K_k, UZ)_\bullet\] where $U$ denotes the forget functor. Since $K_{k,\bullet}$ is a degree wise free perverse chain complex with zero differential, we have
    \[K_{k,\bullet} \simeq \bigoplus_{\overline{p},\, m} F_{\overline{p}}(\mathbb{S}^m)_\bullet.\]
    The functor $\Hom_{(Ch(R))^{\pGM}}(-, UZ)$ sends colimits to limits. In particular, we have \[\Hom_{(Ch(R))^{\pGM}} (\bigoplus_{\overline{p},\, m} F_{\overline{p}}(\mathbb{S}^m), UZ)_\bullet \simeq \prod_{\overline{p},\, m}\Hom_{(Ch(R))^{\pGM}}  (F_{\overline{p}}(\mathbb{S}^m), UZ)_\bullet.\]
    So we just need to show that for $m\in \mathbb{N}$ and $\overline{p}\in \pGM$,  the perverse chain complex $\Hom_{(Ch(R))^{\pGM}}(F_{\overline{p}}(\mathbb{S}^m), UZ)_\bullet$ is acyclic. Recall that for perverse chain complexes $X_\bullet, Y_\bullet$, the internal Hom is the perverse chain complex defined for $k \in \mathbb{N}$ and $\overline{r}\in \pGM$ by
    \[\Hom_{(Ch(R))^{\pGM}}(X,Y)_k^{\overline{r}}=\lim_{\overline{r}\leq \overline{q}-\overline{s}} \prod_{j-i=k} \Hom_R(X_i^{\overline{s}}, Y_j^{\overline{q}}).\]
    Here, we have
    \[\Hom  (F_{\overline{p}}(\mathbb{S}^m)_\bullet, UZ)_k^{\overline{r}}=\lim_{\overline{r}\leq \overline{q} -\overline{p}} \Hom_R(R, UZ_{m+k}^{\overline{q}})\simeq \Hom_R(R, UZ_{m+k}^{\overline{r}+\overline{p}}).\]
    If $\overline{p}=\overline{0}$ and $m=0$, the perverse chain complex $\Hom  (F_{\overline{p}}(\mathbb{S}^m), UZ)_k^{\overline{r}}$ is isomorphic to $UZ_\bullet$, hence it is acyclic. For arbitrary $m\in \mathbb{N}$ and $\overline{p} \in \pGM$, it is isomorphic to $UZ_\bullet$ but shifted by $m$ (degree wise) and by $\overline{p}$ (perversity wise), it is also acyclic. 
    We come back to our short exact sequence
    \[0 \to \Hom_A(M/F_k,Z)_\bullet \to \Hom_A(M/F_{k-1},Z)_\bullet \to \Hom_A(F_k/F_{k-1}, Z)_\bullet \to 0,\]
    By taking the long exact sequence induced in homology, for any $\overline{p}\in\pGM$ we get isomorphisms
    \[H_\ast (\Hom_A(M/F_{k}, Z)_{\overline{p}}) \to H_\ast(\Hom_A(M/F_{k-1}, Z)_{\overline{p}}).\]
    We consider the spectral sequence which arises from the following exact couple
    \[D^1_{i,j}=H_{i+j}(\Hom_A(M/F_{i-1}, Z)_{\overline{p}}),\quad E^1_{i,j}=H_{i+j}(\Hom_A(F_i/F_{i-1}, Z)_{\overline{p}}).\]
    It is a right half-plane spectral sequence $E^\ast_{i,j}$ with $E^2=0$. Since $\lim_i D^1_{i, \ast-i}=0$ and $\lim^1_i D^1_{i, \ast-i}=0$ \cite[Proposition1.8]{Boa99}, the spectral sequence \emph{converges conditionally} \cite[Definition 5.10]{Boa99} to $\colim_i D^1_{i, \ast-i}=H_{\ast}(\Hom_A(M, Z)_{\overline{p}})$. Using the (co)limit comparison theorem \cite[Theorem 7.2]{Boa99} with the trivial spectral sequence (seen as induced by the identically zero exact couple), we get that 
    \[H_{\ast}(\Hom_A(M, Z)_{\overline{p}})=0.\]
    This shows that $M_\bullet$ is semi-projective.
    \end{proof}
    
    We can now prove the following main theorem.
    \begin{letterthm}\label{thm:char_cofib_Amod}
    For a pDG $A_\bullet$-module $M_\bullet$, the following statements are equivalent.
    \begin{enumerate}[label=(\arabic*), noitemsep]
      \item $M_\bullet$ is semi-projective,
      \item $M_\bullet$ is cofibrant,
      \item $M_\bullet$ is a retract of semi-free pDG $A_\bullet$-module.
    \end{enumerate}
    \end{letterthm}
    \begin{proof}
    By proposition \ref{prop:sproj-cofib}, we have $(1)\Rightarrow (2)$. 
    \begin{itemize}[label= -]
    \item $(2) \Rightarrow (3)$ By \cite[Proposition 2.1.18]{Hov99}, a cofibrant object $M_\bullet$ is a retract of an $\mathcal{I}^{\pGM}_A$-cell complex. It suffices to show that an $\mathcal{I}^{\pGM}_A$-cell complex is semi-free. Recall that a perverse chain complex $Y_\bullet$ is an $\mathcal{I}^{\pGM}_A$-cell complex if there exists an ordinal $\lambda$ and a map $g : 0 \to Y_\bullet$ that is a transfinite composite of a $\lambda$-sequence $Y_\bullet$ such that $Y_\bullet^0 = 0$ and, for a successor ordinal $\alpha+1 < \lambda$, $Y_\bullet^{\alpha+1}$ is obtained as the pushout in a diagram
    \[\begin{tikzcd}
    \bigsqcup A^q_\bullet \arrow[d, "\bigsqcup i_q"'] \arrow[r, "j"] & Y_\bullet^\alpha \arrow[d] \\
    \bigsqcup B^q_\bullet \arrow[r, "k"]                             & Y_\bullet^{\alpha+1}      
    \end{tikzcd}\]
    where the $i_q : A_\bullet^q \to B_\bullet^q$ run through some set of maps in 
    \[\mathcal{I}^{\pGM}_A=\{(A \boxtimes F_{\overline{p}}(\mathbb{S}^m))_\bullet\to (A \boxtimes F_{\overline{p}}(\mathbb{D}^{m+1}))_\bullet : m\in \mathbb{N}, \, \overline{p}\in \pGM \}.\] 
    The quotients $(Y_{\alpha+1}/Y_\alpha)_\bullet$ are isomorphic as perverse chain complexes to direct sums of $(A\boxtimes F_{\overline{p}}(\mathbb{S}^m))_\bullet$, so they are of the form $(A \boxtimes K_k)_\bullet$ with $K_{k,\bullet}$ free and of zero differential.   
    \item $(3) \Rightarrow (1)$ By proposition \ref{prop:semifree-semiproj}, we know that a semi-free pDG $A_\bullet$-module is semi-projective. Furthermore, the retract of semi-projective pDG $A_\bullet$-module is also semi-projective.
    \end{itemize}    
    \end{proof}
    \subsection{Hochschild (co)homology using derived functors}
    
    We now study cases where the bar construction of a pDGA is a cofibrant approximation of it. We say that a pDGA $A_\bullet$ is \emph{augmented} if there exist a morphism a perverse chain complexes $\eps:A_\bullet \to F_{\overline{0}}(\mathbb{S}^0)_\bullet$.
    \begin{proposition}\label{prop:bar-cofib-approx}
    If $A_\bullet$ is an augmented pDGA that is cofibrant in $(Ch(R))^{\pGM}$, then the map 
    \[q_A:\mathbb{B}(A)_\bullet\to A_\bullet \text{ defined by }
    q_A(a[a_1, \ldots, a_k]b)=\begin{cases}
        0 & \text{if } k >0 \\
        a.b & \text{if } k=0 
    \end{cases}\]
    is a cofibrant approximation of $A_\bullet$ in $pMod(A^e)$. In other words, it is an acyclic fibration and $\mathbb{B}(A)_\bullet$ is cofibrant in $pMod(A^e)$.  
    \end{proposition}

    \begin{proof}
    The map $q_A$ is a fibration since it is degree-wise surjective. Let's now see that $q_A$ is a weak equivalence i.e.\ a perversity-wise quasi-isomorphism of chain complex. Let $\overline{r}\in P_X$ and consider the map 
    \[q_A:\mathbb{B}(A)_{\overline{r}}\to A_{\overline{r}}.\]
    It is surjective, so it is enough to show that $\ker(q_A)$ has trivial homology to conclude that $q_A$ induces an isomorphism on homology. To do so, consider the following map $h:\mathbb{B}(A)_{\overline{r}}\to \mathbb{B}(A)_{\overline{r}}$ given by 
    \[h(a[a_1| \ldots| a_k]b)=\begin{cases}
        1[a a_1| \ldots| a_k]b & \text{if } k >0 \\
        1[a]b & \text{if } k=0. 
    \end{cases}\]
    This is a contracting homotopy on $\ker(q_A)$ i.e.\ we have $\id_{\ker(q_A)}=D\circ h+h \circ D$ on $\ker(q_A)$. The map $D$ denotes the differential on $\mathbb{B}(A)_{\overline{r}}$, recall that it is given by $D=d_0+d_1$ with $d_0(a[a_1|\ldots|a_k]b)+d_1(a[a_1|\ldots|a_k]b)=$ 
    \[d(a)[a_1|\ldots|a_k]b 
    -\sum_{i=1}^{k}(-1)^{\eps_i}a[a_1|\ldots|d(a_i)|\ldots|a_k]b+(-1)^{\eps_{k+1}}a[a_1|\ldots|a_k]d(b)+\]
    \[(-1)^{\Real{a}}aa_1[a_2|\ldots|a_k]b+\sum_{i=2}^{k}(-1)^{\eps_i}a[a_1|\ldots|a_{i-1}a_i|\ldots|a_k]-(-1)^{\eps_k}a[a_1|\ldots|a_{k-1}]a_kb\]
    where $\eps_i=\Real{a}+\sum_{j<i}\Real{s(a_j)}$.\\
    Consider an element $a[a_1|a_2|\ldots|a_k]b$ in $\mathbb{B}_k(A)^q_{\overline{r}}$ of length $k>0$ and degree
    \[\degr(a[a_1|\ldots|a_k]b)=\Real{a}+\Real{b}+ \sum_{i=1}^{k}\Real{s(a_i)}=\Real{a}+\Real{b}+ \sum_{i=1}^{k}\Real{a_i}+k=q.\] 
    We compute
    \[D \circ h (a[a_1|a_2|\ldots|a_k]b)=D(1[a|a_1|a_2|\ldots|a_k]b)=\]
    \[d(1)[a|a_1|\ldots|a_k]b 
    -\sum_{i=0}^{k}(-1)^{\eps_i-1}1[a|\ldots|d(a_i)|\ldots|a_k]b+(-1)^{\eps_{k+1}-1}1[a|\ldots|a_k]d(b)\]
    \[+a[a_1|\ldots|a_k]b+\sum_{i=1}^{k}(-1)^{\eps_i-1}1[a|\ldots|a_{i-1}a_i|\ldots|a_k]-(-1)^{\eps_k-1}1[a|\ldots|a_{k-1}]a_kb\]
    and $h\circ D(a[a_1|\ldots|a_k]b)$ is given by
    \[h(d(a)[a_1|\ldots|a_k]b 
    -\sum_{i=1}^{k}(-1)^{\eps_i}a[a_1|\ldots|d(a_i)|\ldots|a_k]b+(-1)^{\eps_{k+1}}a[a_1|\ldots|a_k]d(b)+\]
    \[(-1)^{\Real{a}}aa_1[a_2|\ldots|a_k]b+\sum_{i=2}^{k}(-1)^{\eps_i}a[a_1|\ldots|a_{i-1}a_i|\ldots|a_k]-(-1)^{\eps_k}a[a_1|\ldots|a_{k-1}]a_kb)\]
    which is equal to
    \[1[d(a)|a_1|\ldots|a_k]b 
    -\sum_{i=1}^{k}(-1)^{\eps_i}1[a|a_1|\ldots|d(a_i)|\ldots|a_k]b+(-1)^{\eps_{k+1}}1[a|\ldots|a_k]d(b)+\]
    \[(-1)^{\Real{a}}1[aa_1|a_2|\ldots|a_k]b+\sum_{i=2}^{k}(-1)^{\eps_i}1[a|\ldots|a_{i-1}a_i|\ldots|a_k]-(-1)^{\eps_k}1[a|\ldots|a_{k-1}]a_kb.\]
    Finally, we get \[(D\circ h + h\circ D)(a[a_1|\ldots|a_k]b)=d(1)[a|a_1|\ldots|a_k]b+a[a_1|\ldots|a_k]b=a[a_1|\ldots|a_k]b.\]
    For $a[]b\in B_0(A, \overline{A}, A)^q_{\overline{r}}$, we have
    \[(D\circ h + h\circ D)(a[]b)=D(1[a]b)+1[d(a)]b+(-1)^{\Real{a}}1[a]d(b)\]
    \[-1[d(a)]b+(-1)^{\Real{a}-1}1[a]d(b)+a[]b+(-1)^{\Real{a}}ab+1[d(a)]b+(-1)^{\Real{a}}1[a]d(b).\]
    Hence, $(D\circ h + h\circ D)(a[]b)=a[]b+(-1)^{\Real{a}}ab$. This proves that $h$ is a contracting homotopy on $\ker(q_A)$, it has trivial homology. The map $q_A$ is indeed a quasi-isomorphism. \\
    We just have to show that $\mathbb{B}(A)_\bullet$ is a cofibrant object in $pMod(A^e)$. $A_\bullet$ is equipped with two maps $\eta:(F_{\overline{0}}(\mathbb{S}^0))_\bullet\to A_\bullet$ and $\eps:A_\bullet \to F_{\overline{0}}(\mathbb{S}^0)_\bullet$ such that their composite $\eps \circ \eta$ is the identity. Then we have $\overline{A}_\bullet=A_\bullet/\eta(F_{\overline{0}}(\mathbb{S}^0))_\bullet=\ker(\eps)_\bullet$. $\overline{A}_\bullet$ is cofibrant since it is a retract of $A_\bullet$: 
    \[\overline{A}_\bullet=\ker(\eps)_\bullet\hookrightarrow A_\bullet \twoheadrightarrow A_\bullet/\eta(F_{\overline{0}}S^0)_\bullet=\overline{A}_\bullet.\]
    The suspension $s\overline{A}_\bullet$ is clearly cofibrant. Since $(Ch(R))^{P_X}$ is a monoidal model category (Remark \ref{rem:quillen_adj_pch}), the tensor product of two cofibrant perverse chain complexes is cofibrant. Hence, for $k\in \mathbb{N}$, $(s\overline{A})^{\boxtimes k}_\bullet$ is cofibrant in $(Ch(R))^{P_X}$. This implies that $\mathbb{B}_k(A)_\bullet=(A \boxtimes (s\overline{A})^{\boxtimes k} \boxtimes A)_\bullet$ is cofibrant as a pDG $A_\bullet^e$-module. By induction, we show that $\mathbb{B}_{\leq k, \bullet}=\oplus_{i=0}^k(A \boxtimes (s\overline{A})^{\boxtimes i} \boxtimes A)_\bullet$  is cofibrant equipped with the differential $d=d_0+d_1$ (where $d_0$ is induced by the differential of $A_\bullet$ and $d_1$ is induced by the multiplication of two consecutive terms in a tensor product). Finally, the bar construction is cofibrant as a colimit of cofibrant objects $\mathbb{B}_{\leq k,\bullet}$.
    \end{proof}

    We've shown (under some reasonable assumptions) that the two-sided bar construction $\mathbb{B}(A)_\bullet$ is a cofibrant approximation of the pDGA $A_\bullet$, this implies that
    \[HH^\bullet_\ast(A,M)=H_\ast(\mathbb{B}(A)\boxtimes_{A^e}M)_\bullet=\Tor^{A^e}_\ast(A,M)^\bullet\]
    and 
    \[HH^\ast_\bullet(A,M)=H^\ast(\Hom_{A^e}(\mathbb{B}(A), M)_\bullet)=\Ext^\ast_{A^e}(A,M)_\bullet.\]
    
    When we work on a field $R$, we don't have to put restrictions on $A_\bullet$. We'll need the following lemma which is inspired by \cite[Lemma 6.3]{FHT01}. We omit the proof as it is analogous to the one given by Félix, Halperin and Thomas. 

    \begin{lemma}\label{lem:sfreequotients-sfree}
    Let $M_\bullet$ be a pDG $A_\bullet$-module. If it is the union of an increasing sequence of submodules 
    \[0=F_{-1, \bullet}\subset F_{0, \bullet} \subset F_{1, \bullet} \subset \ldots \] 
    such that each $F_{k, \bullet}/F_{k-1, \bullet}$ is $A_\bullet$-semi-free. Then $M_\bullet$ itself is $A_\bullet$-semi-free.    
    \end{lemma}
    We now prove that the bar construction is always cofibrant when the coefficient ring is a field.    
    \begin{proposition}\label{prop:bar-cofib-field}
    Let $A_\bullet$ be a perverse DGA. If $R$ is a field then the map $q_A$
    is a cofibrant approximation of $A_\bullet$ in $pMod(A^e)$. 
    \end{proposition}
    \begin{proof}
     As shown in the previous proof, $q_A$ is an acyclic fibration. To prove that $\mathbb{B}(A)_\bullet$ is a cofibrant object in $pMod(A^e)$, we're going to show that it is $A_\bullet^e$-semi-free. Recall that its differential is given by $d=d_0+d_1$. Consider the following filtration on $\mathbb{B}(A)_\bullet$
    \[0 \subset \mathbb{B}_{\leq 0, \bullet}\subset \mathbb{B}_{\leq 1, \bullet} \subset \ldots \mathbb{B}_{\leq k, \bullet} \subset \ldots \text{ with } \mathbb{B}_{\leq k, \bullet}=\oplus_{i=0}^k(A \boxtimes (s\overline{A})^{\boxtimes i} \boxtimes A)_\bullet.\]
    Then, for $k\in \mathbb{N}$, the quotient $\mathbb{B}_{\leq k, \bullet}/\mathbb{B}_{\leq k-1, \bullet}$ is isomorphic as a perverse DG $A_\bullet^e$-module to $(A \boxtimes (s\overline{A})^{\boxtimes k} \boxtimes A)_\bullet$. Note that the differential on the quotient is given by $d_0$ since $d_1(\mathbb{B}_{\leq k, \bullet})\subset \mathbb{B}_{\leq k-1, \bullet}$. \\
    The following filtration of $(s\overline{A})^{\boxtimes k}_\bullet$ shows that it is semi-free:
    \[0 \subset \ker (d_0)_\bullet \subset (s\overline{A})^{\boxtimes k}_\bullet.\]
    Indeed, the differential $d_0$ on each successive quotient is trivial and since $R$ is a field, the complexes are degree-wise free. This shows that $(s\overline{A})^{\boxtimes k}_\bullet$ is $R$-semi-free, hence, \[\mathbb{B}_{\leq k, \bullet}/\mathbb{B}_{\leq k-1, \bullet}\simeq (A\boxtimes(s\overline{A})^{\boxtimes k}\boxtimes A)_\bullet\] is $A_\bullet^e$-semi-free. We conclude using Lemma \ref{lem:sfreequotients-sfree}. 
    \end{proof}

\section{Derived Poincaré duality algebras}\label{sect:DPDA}
In \cite{Men09}, the Hochschild cohomology of $C^\ast(M; \mathbb{F})$, the singular cochains  of a compact, simply-connected, oriented, smooth manifold $M$, is equipped with a BV-algebra structure where the operator $\Delta$ is given by the dual of \emph{Connes boundary map} $B^\vee$. Actually, $B^\vee$ is defined on the linear dual of the Hochschild chain complex $(HC_\ast(C^\ast(M; \mathbb{F})))^\vee\simeq HC^\ast(C^\ast(M; \mathbb{F}), C^\ast(M; \mathbb{F})^\vee)$ and $\Delta$ is given by
\[\Delta= \Dual^{-1}\circ B^\vee \circ \Dual\]
where
\[\Dual: HH^\ast(C^\ast(M; \mathbb{F})) \xrightarrow{\simeq} HH^\ast(C^\ast(M; \mathbb{F}), C^\ast(M; \mathbb{F})^\vee)\]
is an isomorphism of $HH^\ast(C^\ast(M; \mathbb{F}))$-modules. We would like to obtain a similar result for the Hochschild cohomology of the blown-up intersection cochain complex. First, one can consider a more general context. The cochain complex $C^\ast(M; \mathbb{F})$ equipped with $\Dual$ is an example of \emph{derived Poincaré duality algebra} by Abbaspour's terminology \cite[Section 4]{Abb15}. We now present a similar notion for pDGAs. Recall that we denote by $DA_\bullet$ the dual of $A_\bullet$ in the category of perverse chain complexes (Example \ref{ex:pmonoidal_chcplx}).

\begin{defi}
    We say that a pDGA $A_\bullet$ is a \emph{perverse derived Poincaré duality algebra} (pDPDA) if it isomorphic to its dual $DA_\bullet$ in the derived category of $A_\bullet$-bimodules.
\end{defi}

\begin{rem}\label{rem:pDPDA}
    In other words, if $A_\bullet$ is pDPDA then there exist, in the category of a $A_\bullet$-bimodules, a cofibrant approximation $q_A: P_\bullet\xrightarrow{\simeq} A_\bullet$ and a quasi-isomorphism of $c:P_\bullet\xrightarrow{\simeq} DA_\bullet$ that fit in the following zigzag:
    \[A_\bullet \xleftarrow[\simeq]{q_A} P_\bullet \xrightarrow[\simeq]{c} DA_\bullet.\]
    In particular, we can take the bar construction $P_\bullet=\mathbb{B}(A)_\bullet$ (when it is cofibrant as a pDG $A_\bullet$-bimodule). We have the following commutative diagram where the vertical maps are isomorphisms
    \[\adjustbox{scale=0.8}{\begin{tikzcd}[column sep = huge]
	{\Hom_{A^e}(\mathbb{B}(A),A)_\bullet} & {\Hom_{A^e}(\mathbb{B}(A),\mathbb{B}(A))_\bullet} & {\Hom_{A^e}(\mathbb{B}(A),DA)_\bullet} \\
	{\Hom_{Ch(R)^{\pGM}}(T(s\overline{A}),A)_\bullet} & {\Hom_{Ch(R)^{\pGM}}(T(s\overline{A}),\mathbb{B}(A))_\bullet} & {\Hom_{Ch(R)^{\pGM}}(T(s\overline{A}),A)_\bullet}
	\arrow["{\Hom_{A^e}(\mathbb{B}(A),q_A)}"', from=1-2, to=1-1]
	\arrow["{\Hom_{A^e}(\mathbb{B}(A),c)}", from=1-2, to=1-3]
	\arrow["{\Hom_{Ch(R)^{\pGM}}(T(s\overline{A}),c)}"', from=2-2, to=2-3]
	\arrow["{\Hom_{Ch(R)^{\pGM}}(T(s\overline{A}),q_A)}.", from=2-2, to=2-1]
	\arrow["\simeq"{marking}, draw=none, from=2-1, to=1-1]
	\arrow["\simeq"{marking}, draw=none, from=2-3, to=1-3]
	\arrow["\simeq"{marking}, draw=none, from=2-2, to=1-2]
    \end{tikzcd}}\]
    By Remark \ref{rem:quillen_adj_pch}, the bottom arrows are all quasi-isomorphisms. Hence, when we take homology, we get an isomorphism:
    \[HH^\ast_\bullet(A) \xleftarrow[\simeq]{HH^\ast(A, q_A)} HH^\ast_\bullet(A, P) \xrightarrow[\simeq]{HH^\ast(A,c)} HH^\ast_\bullet(A, DA).\]
    Notice that $[q_A]$, which is the unit of the left $HH^\ast_\bullet(A)$-module $HH^\ast_\bullet(A)$, is sent to $[Id_P]$ and finally to $[c]$. Hence, this isomorphism corresponds to the action of $HH^\ast_\bullet(A)$ on $[c]$ (given at the end of subsection \ref{subsect:Hoch_via_bar}). If $[c]\in HH^{j}_{\overline{q}}(A, DA)$, it is an isomorphism of $HH^\ast_\bullet(A)$-modules of degree $j$ and of perverse degree $\overline{q}$.
\end{rem}

In the next subsection, we show how to extend the perverse Gerstenhaber algebra structure on $HH^\ast_\bullet(A)$ (given in Theorem \ref{thm:Hoch_is_Gerst}) into a \emph{perverse Batalin-Vilkovisky algebra} one when $A_\bullet$ is a pDPDA. 
    \begin{defi}\label{def:pBV}
        A \emph{perverse Batalin-Vilkovisky} (pBV) algebra is a perverse Gerstenhaber algebra $(A_\bullet, \cup, [-,-])$ equipped  with a degree -1, perverse degree $\overline{0}$ operator $\Delta: A_\bullet^\ast\to A_\bullet^{\ast-1}$ such that
        \begin{itemize}
            \item $\Delta\circ \Delta =0$,
            \item $(-1)^{\Real{a}}[a,b]=\Delta(a\cup b)-\Delta(a)\cup b - (-1)^{\Real{a}}a\cup \Delta(b)$ for any $a,b\in A_\bullet$.
        \end{itemize}
        In other words, $\Delta$ is a differential and its deviation from being a derivation with respect to $\cup$ in encoded by the Lie bracket $[-,-]$.
    \end{defi}
In the latter subsections, we will give a sufficient condition for $A_\bullet$ to be a pDPDA and check that it holds for the blown-up intersection cochain complex of a pseudomanifold. 

    \subsection{BV algebra structure on Hochschild cohomology}
    Throughout this subsection, $A_\bullet$ will denote a pDGA. To make $HH^\ast_\bullet(A)$ a perverse BV algebra we will rely on Menichi's proof for differential graded algebras \cite[Proposition 12]{Men09} and use the notion of \emph{Tamarkin-Tsygan calculus} \cite[Section 3.2]{TT05}
    \begin{defi}
        A \emph{Tamarkin-Tsygan calculus} is given by a perverse Gerstenhaber algebra $(\mathcal{V}^\ast_\bullet, \cup, [-,-])$ and a perverse graded module $\Omega^\ast_\bullet$ equipped with
        \begin{itemize}
            \item a structure of perverse graded module over the pDGA $(\mathcal{V}^\ast_\bullet, \cup)$ (the corresponding action is denoted by $i_a$ for $a\in \mathcal{V}^\ast_\bullet$),
            \item a structure of perverse graded module over the perverse graded Lie algebra $((s\mathcal{V})^\ast_\bullet, [-,-])$ (the corresponding action is denoted by $L_a$ for $a\in \mathcal{V}^\ast_\bullet$),
            \item a degree -1 differential $B:\Omega^\ast_\bullet \to \Omega^{\ast-1}_\bullet$
        \end{itemize}
        such that, for any $a,b\in \mathcal{V}^\ast_\bullet$
        \begin{enumerate}[label=\roman*)]
            \item $i_{[a,b]}=L_a\circ i_b - (-1)^{\Real{b}(\Real{a}+1)} i_b\circ L_a$,
            \item $L_{a\cup b}=L_a\circ i_b + (-1)^{\Real{a}}i_a\circ L_b$, 
            \item $L_a=B\circ i_a - (-1)^{\Real{a}} i_a\circ B$
        \end{enumerate}
        where $\circ$ denotes the composition.
    \end{defi}
   
    We now define various operators on the Hochschild chain complex. Using the fact that they form a calculus, we will exhib a BV-algebra structure. 
    \begin{defi}
        Let $\overline{p}\in\pGM$ and consider a cochain $f\in HC^\ast_{\overline{p}}(A)$. There exists morphisms of chain complexes
        \begin{align*}
            i_f:&HC_\ast^\bullet(A) \to HC_{\ast+\Real{f}}^{\bullet+\overline{p}}(A), \\
            L_f:&HC_\ast^\bullet(A) \to HC_{\ast+\Real{f}-1}^{\bullet+\overline{p}}(A) 
        \end{align*}
        given for $a=a_0[a_1|\ldots|a_m]\in HC_\ast^\bullet(A)$ by
        \begin{align*}
            i_f(a_0[a_{1,m}])&:= \sum_{k=1}^m (-1)^{\Real{a_0}\Real{f}}a_0f[a_{1,k}]\otimes[a_{k+1}|\ldots|a_m], \\
            L_f (a_0[a_{1,m}])&:= \sum_{l=1}^m \left( 
            \sum_{k=1}^{m-l} (-1)^{(\Real{f}-1)(\sum_{i=0}^k \Real{s(a_i)})} a_0[a_{1,k}|f[a_{k+1,k+l}]|\ldots|a_m] \right.\\
            &\left.+\sum_{k=m-l+1}^m (-1)^{\Real{f}-1+\sum_{i\leq k}\Real{s(a_i)}\sum_{i> k}\Real{s(a_i)}}
            f[a_{k+1,m}|a_{0,l+k-m-1}]\otimes [a_{l+k-m,k}] \right)
        \end{align*}
    \end{defi}
    In our case, $B$ will be \emph{Connes' boundary map}. 
    \begin{defi}
        The \emph{Connes' boundary map} is a degree $-1$ and perverse degree $\overline{0}$ morphism of perverse chain complexes
        \[B:HC_\ast^\bullet(A) \to HC_{\ast-1}^\bullet(A)\]
        given by
        \[B(a_0[a_1|\ldots|a_m]):=\sum_{i=0}^m (-1)^{\sum_{k<i}\Real{s(a_k)}\sum_{k\geq i}\Real{s(a_k)}}1[a_i|\ldots|a_m|a_0|\ldots|a_{i-1}]\]
        for $a_0[a_1|\ldots|a_m]\in HC_\ast^\bullet(A)$ .
    \end{defi}
    \begin{rem}
        Similarly to \cite[Chapter 2]{Lod98}, one can show that
        \[B\circ D_\ast + D_\ast \circ B=0.\]
        Hence, the Connes' boundary map induces a morphism on Hochschild homology.
    \end{rem}
    \begin{defi}
        The \emph{dual of Connes' boundary map} is defined on $D(HC_\ast(A))_\bullet^\ast$, the linear dual of the Hochschild chain complex, by
        \[(B^\vee \phi)(a_0[a_{1,m}]):=(-1)^{\Real{\phi}}\sum_{i=0}^m (-1)^{\sum_{k<i}\Real{s(a_k)}\sum_{k\geq i}\Real{s(a_k)}}\phi(1[a_i|\ldots|a_m|a_0|\ldots|a_{i-1}])\]
        for $a_0[a_1|\ldots|a_m]\in HC_\ast^\bullet(A)$ and $\phi\in D(HC_\ast(A))_\bullet^\ast.$
    \end{defi}
    \begin{rem}
        By the internal tensor-hom adjunction (Remark \ref{rem:internal_tenhom_adj}), we have
        \[\Hom_{(Ch(R))^{\pGM}}(HC_\ast(A) , F_{\overline{0}}(\mathbb{S}^0))_\bullet \simeq HC^\ast_\bullet(A, \Hom_{(Ch(R))^{\pGM}}(A , F_{\overline{0}}(\mathbb{S}^0))).\]
        In other words,
        \[D(HC_\ast(A))_\bullet^\ast \simeq HC^\ast_\bullet(A, DA).\]
    \end{rem}
    The following proposition is the perverse analogue of a result due to Daletskii, Gel'fand and Tsygan, see \cite{DGT89}. A detailed proof for $A_\infty$-algebras is given in \cite[Theorem 1.1]{CLY22}. Chen, Lyu and Yang's computations are also valid for perverse objects.  
    \begin{proposition}\label{prop:Hoch_calculus}
        $(HH^\ast_\bullet(A), HH_\ast^\bullet(A), \cup, [-,-], i_\ast, L_\ast, B)$ is a Tamarkin-Tsygan calculus.
    \end{proposition}
    As a corollary, one gets the following result.
    \begin{lemma}[{\cite[Lemma 16]{Men09}}]\label{lem:lemma_Menichi}
    For any $f,g\in HH^\ast_\bullet(A)$ and $c\in HH^\ast_\bullet(A, DA)$, we have
        \begin{align*}
            [f,g].c=&(-1)^{\Real{f}}B^\vee((f\cup g).c)-f.B^\vee(g.c) \\
            &+(-1)^{(\Real{f}-1)(\Real{g}-1)}g.B^\vee(f.c)+(-1)^{\Real{g}}(f\cup g).B^\vee(c).
        \end{align*}
    \end{lemma}
    The above given result is the dual version of a lemma proved by Ginzburg.
    \begin{lemma}[{\cite[proof of Theorem 3.4.3]{Gin06}}]\label{lem:Ginzburg}
        If $(\mathcal{V}^\ast_\bullet, \Omega^\ast_\bullet, \cup, [-,-], i_\ast, L_\ast, B)$ is a Tamarkin-Tsygan calculus, then, for any $a,b\in \mathcal{V}^\ast_\bullet$, we have
        \[i_{[a,b]}=(-1)^{\Real{a}} B\circ i_{a\cup b} - i_a \circ B \circ i_g + (-1)^{(\Real{a}-1)(\Real{b}-1)} i_g\circ B \circ i_a + (-1)^{\Real{b}} i_{a\cup b}\circ B.\]
    \end{lemma}
    Under some assumptions, we can now exhib a BV-algebra structure on $HH^\ast_\bullet(A)$. The next result is analogue to \cite[Proposition 12]{Men09}.
  
    \begin{letterthm}\label{thm:DPDA_is_BV}
    Let $A_\bullet$ be a derived Poincaré duality algebra and let $\phi:HH^\ast_\bullet(A)\to HH^\ast_\bullet(A,DA)$ be the associated isomorphism of $HH^\ast_\bullet(A)$-bimodules. We define a map $\Delta: HH^\ast_\bullet(A) \to HH^{\ast-1}_\bullet(A)$ by setting 
    \[\Delta=\phi^{-1}\circ B^\vee \circ \phi.\]
    If $\Delta(1)=0$, then the Gerstenhaber algebra $HH^\ast_\bullet(A)$ equipped with $\Delta$ becomes a BV-algebra.
    \end{letterthm}

    \begin{proof}  
    We denote the cup product and the Lie bracket on $HH^\ast_\bullet(A)$ by $\cup$ and $[-,-]$ respectively. We clearly have $\Delta\circ \Delta=0$. We're left with checking that the deviation of $\Delta$ being a derivation with respect to $\cup$ is given by $[-,-]$. Let $q_A: P_\bullet\xrightarrow{\simeq} A_\bullet$ be a cofibrant approximation of $A_\bullet$ and let $c:P_\bullet\xrightarrow{\simeq} DA_\bullet$ be a quasi-isomorphism of $A_\bullet$-modules. The isomorphism $\phi$ corresponds to the action of $HH^\ast_\bullet(A)$ on $[c]$, see Remark \ref{rem:pDPDA}. In other words, for any $f\in HH^\ast_\bullet(A)$ we have 
    \[\Delta(f).[c] := B^\vee(f.[c]).\]
    Hence, the condition $\Delta(1)=0$ is equivalent to asking that $B^\vee([c])=0$.
    By Lemma \ref{lem:lemma_Menichi}, we get for $f, g\in HH^\ast_\bullet(A)$
        \begin{align*}
            [f,g].c=&(-1)^{\Real{f}}\Delta(f\cup g).c-f.(\Delta(g).c)+(-1)^{(\Real{f}-1)(\Real{g}-1)}g.(\Delta(f).c)\\
            =&(-1)^{\Real{f}}\Delta(f\cup g).c-(f \cup \Delta(g)).c +(-1)^{(\Real{f}-1)(\Real{g}-1)}(g\cup\Delta(f)).c\quad.
        \end{align*}
        This implies that
        \begin{align*}
            [f,g]=&(-1)^{\Real{f}}\Delta(f\cup g)-f \cup \Delta(g)+(-1)^{(\Real{f}-1)(\Real{g}-1)}g\cup\Delta(f)\\
            =&(-1)^{\Real{f}}\Delta(f\cup g)-f \cup \Delta(g)+(-1)^{\Real{f}-1}\Delta(f)\cup g\\
            =&(-1)^{\Real{f}}(\Delta(f\cup g)-(-1)^{\Real{f}}f \cup \Delta(g)-\Delta(f)\cup g.
        \end{align*}
    \end{proof}

\subsection{Sufficient condition to be a DPDA}
    Let $M$ be a compact, simply-connected, oriented, smooth manifold. In order to show that $C^\ast(M; \mathbb{F})$ is a derived Poincaré duality algebra, Menichi \cite{Men09} uses the fact that, by Poincaré duality, the action of $C^\ast(M; \mathbb{F})$ on a fundamental cycle $\zeta\in C_\ast(M; \mathbb{F})$ gives a quasi-isomorphism of $C^\ast(M; \mathbb{F})$-modules between $C^\ast(M; \mathbb{F})$ and $C_\ast(M; \mathbb{F})\simeq C^\ast(M; \mathbb{F})^\vee$. Then, he lifts this map into a cocycle in the Hochschild cochain complex $HC^\ast(C^\ast(M; \mathbb{F}), C^\ast(M; \mathbb{F})^\vee)$. Similarly, by Proposition \ref{prop:dual-qiso-bup}, we have a quasi-isomorphism of perverse (right) $\widetilde N^\ast_{\overline{\bullet}}(X;\mathbb{F})$-modules from $\widetilde N^\ast_{\bullet}(X;\mathbb{F})$ to its dual $\mathbb{D}\widetilde N^\bullet_\ast(X;\mathbb{F})$. However, it is not clear how to lift it to $HH^\ast_\bullet(\widetilde N^\ast_{\bullet}(X;\mathbb{F}), \mathbb{D}\widetilde N^\bullet_\ast(X;\mathbb{F}))$. More generally, we ask ourselves if it is possible to lift a quasi-isomorphism of $A_\bullet$-modules between a pDGA $A_\bullet$ and its dual $DA_\bullet$ into a cocycle in $HC^\ast_\bullet(A, DA)$. Things are simplified when $A_\bullet$ is commutative. 
    
    \begin{proposition}\label{prop:Commu_DPDA}
    Let $A_\bullet$ be a commutative pDGA. If there is a cycle $M\in DA_\bullet$ which induces an isomorphism of left $H(A)_\bullet$-modules
    \begin{align*}
        H(A)_\bullet &\xrightarrow{\simeq} H(DA)_\bullet \\
        a &\mapsto a.[M]
    \end{align*}
    then $A_\bullet$ is a derived Poincaré duality algebra.
    \end{proposition}
    \begin{proof}
    Since $A_\bullet$ is a commutative pDGA, the multiplication
    \begin{align*}
        \mu:    &(A\boxtimes A^{op})_\bullet    \to A_\bullet \\
        ~       &a\boxtimes a'         \mapsto a.a'
    \end{align*}
    is a morphism of pDGA and induces a lift
    \begin{align*}
        \widetilde\mu:  & \Hom_A(A, DA)_\bullet \to \Hom_{A^e}(A, DA)_\bullet \\
         ~          & (A_\bullet\xrightarrow{g} DA_\bullet) \mapsto (B(A,A,A)_\bullet  \xrightarrow{q_A} A_\bullet \xrightarrow{g} DA_\bullet). 
    \end{align*}
    Denote by $\xi_M$ the quasi-isomorphism of $A_\bullet$-modules
    \begin{align*}
        \xi_M:  & A_\bullet\to DA_\bullet \\
        ~       & a \mapsto a.M.   
    \end{align*}
    Set $c=\widetilde\mu(\xi_M)$, it is a quasi-isomorphism in $\Hom_{A^e}(A, DA)_\bullet\simeq HC^\ast_\bullet(A, DA)$. Hence, $A_\bullet$ is a pDPDA and as noted in Remark \ref{rem:pDPDA}, we have an isomorphism of $HH^\ast_\bullet(A)$-modules:
    \begin{align*}
        HH^\ast_\bullet(A) &\xrightarrow{\simeq} HH^\ast_\bullet(A, DA) \\
        f            &\mapsto f.[c].
    \end{align*}       
    One notices that this isomorphism sends a class $[f]\in HH^\ast_\bullet(A)$ to $[\xi_M\circ f]$.
    \end{proof}
    We're going to picture this result differently. Let $ev(1_A):\Ext_A(A, DA)_\bullet\to H(DA)_\bullet$ be the evaluation at the unit $1_A$ of $A_\bullet$ and let $i_A:A_\bullet\hookrightarrow (A\boxtimes A^{op})_\bullet$ be the inclusion in the first factor. When $A_\bullet$ is commutative, we're able to exhib a section to $\Ext_{i_A}(A, DA)$. 
    \[\begin{tikzcd}
	{H(DA)_\bullet} & {\Ext_{A}(A, DA)_\bullet} & {HH_\bullet(A, DA):=\Ext_{A^e}(A, DA)_\bullet}
	\arrow["{H(ev(1_A))}"', "\simeq", from=1-2, to=1-1]
	\arrow["{\Ext_{i_A}(A, DA)}", curve={height=-18pt}, from=1-3, to=1-2]
	\arrow["{\Ext_{\mu}(A, DA)}", curve={height=-18pt}, dashed, from=1-2, to=1-3]
    \end{tikzcd}\]
    Two problems arise when $A_\bullet$ is not commutative:
    \begin{itemize}
        \item it is not clear how to get such a section,
        \item if such a section exists, does it preserve quasi-isomorphisms ?
    \end{itemize}
    We think we should be able to deal with the first problem using operadic arguments. This will be done in a subsequent paper. We now treat the second problem. We denote by $\eval$ the composite
    \[eval:HH^\ast_\bullet(A,DA):=\Ext^\ast_{A^e}(A,DA)_\bullet\xrightarrow{\Ext_{i_A}(A, DA)} \Ext_A^\ast(A, DA)_\bullet\simeq H(DA)_\bullet.\]
    \begin{proposition}\label{prop:Hochdual}
    Let $A_\bullet$ be a pDGA and let $c\in HC^\ast_\bullet(A,DA)$ be a cocycle such that
    the morphism of $H(A)_\bullet$-modules
    \begin{align*}
        H(A)_\bullet &\xrightarrow{\simeq} H(DA)_\bullet \\
        a &\mapsto a.\eval([c])    
    \end{align*}
    is an isomorphism.
    Then $c$ is a quasi-isomorphism and hence, $A_\bullet$ is a derived Poincaré dualité algebra.
    \end{proposition}
    
    \begin{proof}
        Let $q_A: P_\bullet\xrightarrow{\simeq} A_\bullet$ be a cofibrant approximation of $A_\bullet$ as an $A^e_\bullet$-module. We consider $s_A:A_\bullet\xrightarrow{\simeq} P_\bullet$ a section of $q_A$ in the category of $A_\bullet$-modules. \\
        Notice that $\eval$ is the map induced in homology by the following composite:
        \[\begin{array}{ccccc}
            \Hom_{A^e}(P, DA)_\bullet& \xrightarrow{\Hom(s_A, DA)} & \Hom_A(A, DA)_\bullet&\xrightarrow[\simeq]{ev(1_A)} & DA_\bullet \\
            f & \longmapsto & f\circ s_A &\longmapsto & f\circ s_A(1_A) \\
        \end{array}\]
        By hypothesis, the map $c\circ s_A$ which sends $a\in A_\bullet$ to $a.(c\circ s_A(1_A))$ is a quasi-isomorphism. Since $s_A$ is also quasi-isomorphism, by the two out of three property, $c$ is also a quasi-isomorphism. We have the following zigzag of quasi-isomorphisms:
        \[A_\bullet \xleftarrow[\simeq]{q_A} P_\bullet \xrightarrow[\simeq]{c} DA_\bullet.\]
    \end{proof}
    
    \subsection{Application to the blown-up intersection cochain complex}
    Let $X$ be a pseudomanifold. To be able to apply Proposition \ref{prop:Hochdual} to the blown-up intersection cochain complex, we need to lift $[\Gamma_X]$ (Proposition \ref{prop:dual-qiso-bup}) into a class $[Y]\in HH^\ast_\bullet(\widetilde N(X;\mathbb{F}), \mathbb{D}\widetilde N(X;\mathbb{F}))$ such that $\eval([Y])=[\Gamma_X]$. 
    \[
    \begin{tikzcd}
    {\Ext_{\widetilde N}(\widetilde N, \mathbb{D}\widetilde N)_\bullet} \arrow[rd, "H(ev(1_A))"'] \arrow[rr, "?"] &                          & {HH_\bullet(\widetilde N, \mathbb{D}\widetilde N):=\Ext_{\widetilde N^e}(\widetilde N, \mathbb{D}\widetilde N)_\bullet} \arrow[ld, "\eval"] \\
                                                                                                    & H(\mathbb{D}\widetilde N)_\bullet. &                                                                                                                       
    \end{tikzcd}
    \]
    As we mentioned in the previous subsection, this will be done in a subsequent paper. Now, we work on $\mathbb{Q}$ and consider $\widetilde A^\ast_{PL, \overline{\bullet}}(X)$, the blown-up of Sullivan's polynomial forms of $X$ \cite[Example 1.34]{CST18ration}. Recall that we have a De Rham theorem analogue for the blown-up complexes (Proposition \ref{prop:integ_polyforms}). By the commutativity of $\widetilde A^\ast_{PL, \overline{\bullet}}(X)$ we will be able to lift $[\Gamma_X]$. 
    First, using the following result, we obtain a quasi-isomorphism between $\widetilde A^\ast_{PL, \overline{\bullet}}(X)$ and its dual.
    \begin{proposition}\label{prop:DGAqiso-gives-qiso-dual}
        Let $A_\bullet, B_\bullet$ be two pDGA. If there is a quasi-isomorphism $\gamma$ between $A_\bullet$ and $B_\bullet$ then it induces a quasi-isomorphism
        \[\phi:\Hom_B(B,DB)_\bullet\to \Hom_A(A,DA)_\bullet.\]
        Furthermore, if $\gamma$ is an isomorphism of pDGA then $\phi$ preserves quasi-isomorphisms.
    \end{proposition}
    \begin{proof}
    Let $\gamma:A_\bullet\to B_\bullet$ be a quasi-isomorphism, it induces the following map:
    \[\begin{array}{rccccccc}
       \phi:&\Hom_B(B, DB)_\bullet &\to& DB_\bullet &\longrightarrow& DA_\bullet &\longrightarrow & \Hom_A(A, DA)_\bullet \\
       ~ & f &\mapsto& f(1_B) &\mapsto& f(1_B)\circ \gamma &\mapsto& \left(\alpha \mapsto \alpha.(f(1_B)\circ \gamma)\right)
    \end{array}\]
    where $\alpha.(f(1_B)\circ \gamma)$ is the linear form which maps $a\in A_\bullet$ to $f(1_B)( \gamma (a.\alpha))$.   \\
    $\phi$ induces an isomorphism in homology since all theses maps are quasi-isomorphisms (for the middle one, we can use the universal coefficient theorem and the 5 lemma).\\
    Consider the following diagram
    \[\begin{tikzcd}
    B_\bullet \arrow[r, "f"]                           & DB_\bullet \arrow[d, "\gamma^\vee"] \\
    A_\bullet \arrow[u, "\gamma"] \arrow[r, "\phi(f)"] & DA_\bullet.                         
    \end{tikzcd}\]
    To ensure it commutes, we need $\gamma$ to be an isomorphim of pDGA. Indeed, notice that the composition $\gamma^\vee \circ f\circ \gamma$ takes $\alpha\in A$ to the linear form
    \[a\mapsto (f(\gamma(\alpha))(\gamma(a)).\]
    Since, $f$ is a $B_\bullet$-module morphism, we have
    \[(f(\gamma(\alpha))(\gamma(a))=(\gamma(\alpha).f(1_B))(\gamma(a))=f(1_B)(\gamma(a).\gamma(\alpha)).\]
    On the other hand, $\phi(\alpha)$ is the linear form which maps $a\in A $ to $f(1_B)( \gamma (a.\alpha))$. Hence, the above diagram commutes if
    \[f(1_B)( \gamma (a.\alpha))=f(1_B)(\gamma(a).\gamma(\alpha)),\]
    in other words, when $\gamma$ is a morphism of pDGA. 
    \end{proof}
    
    \begin{corollaire}
    Let $X$ be an $n$-dimensional, compact, second countable and oriented pseudomanifold. Then, there exists a quasi-isomorphism of perverse $\widetilde{A}^{\ast}_{PL,\bullet}(X)$-bimodules:
    \begin{align*}
    \Dual_{A_{PL}}:\widetilde{A}^{\ast}_{PL,\bullet}(X) &\to \Hom(\widetilde{A}^{n-\ast}_{PL,\overline{t}-\bullet}(X), R) \\
    a &\mapsto a.\Gamma_X 
    \end{align*}
    for a certain $\Gamma_X\in \widetilde{A}^n_{PL,\overline{t}}(X)$.
    \end{corollaire}
    \begin{proof}
        The existence of such a quasi-isomorphism follows from the previous proposition, Proposition \ref{prop:integ_polyforms} and Proposition \ref{prop:dual-qiso-bup}. It is a quasi-isomorphism of perverse $\widetilde{A}^{\ast}_{PL,\bullet}(X)$-bimodules since $\widetilde{A}^{\ast}_{PL,\bullet}(X)$ is commutative. In order to get an explicit description, one should consider the following commutative diagram
    \[\begin{tikzcd}
    {\widetilde A^\ast_{PL, \bullet}(X)} \arrow[d, "\int"'] \arrow[r] & {\Hom(\widetilde A^{n-\ast}_{PL, \overline{t}-\bullet}(X),R)}                        \\
    \widetilde N^\ast_\bullet(X;R) \arrow[d, "DP_X"'] \arrow[r, "\Dual_N"]       & {\Hom(\widetilde N^{n-\ast}_{\overline{t}-\bullet}(X;R), R)} \arrow[u, "\int^\vee"'] \\
    I^\bullet C_{n-\ast}(X;R) \arrow[r, "Bid"]                          & {\Hom(I_\bullet C^{n-\ast}(X;R),R)}. \arrow[u, "\Inter^\vee"']                       
    \end{tikzcd}\]
    We denote again by $\Gamma_X$ the image of the unit $1\in A^0_{PL, \overline{0}}$ by the composite 
    \[\int^\vee\circ \Inter^\vee\circ Bid\circ DP_X\circ \int.\] For $w\in\widetilde{A}^{n}_{PL,\overline{t}}(X)$, we have
    \[\Gamma(w)=\int_{\zeta}w\]
    with $\zeta$ a fundamental cycle.
    The quasi-isomorphism of $\widetilde{A}^{\ast}_{PL,\overline{\bullet}}(X)$-bimodule is given by
    \[\alpha \mapsto \alpha.\Gamma :(w\mapsto \int_{\zeta}\alpha.w).\]
    \end{proof}    
    
    \begin{corollaire}
        Let $X$ be an $n$-dimensional, compact, second countable and oriented pseudomanifold. The blown-up of Sullivan's polynomial forms of $X$, $\widetilde A^\ast_{PL, \overline{\bullet}}(X)$, is a derived Poincaré duality algebra.
    \end{corollaire}    
    We would like to show that there exists a Batalin-Vilkovisky algebra structure on $HH^\ast_\bullet(\widetilde A_{PL, \bullet}(X))$. We give a general result for a commutative pDGA. 
    \begin{letterthm}\label{thm:commuDPDA_is_BV}
    If $A_\bullet$ is a commutative pDPDA then then the perverse Gerstenhaber algebra $HH^\ast_\bullet(A)$ can be endowed with a perverse Batalin-Vilkovisky algebra structure.
    \end{letterthm}
    \begin{proof}
    There exists a quasi-isomorphism $c:B(A,A,A)_\bullet \to DA_\bullet$. To apply Theorem \ref{thm:DPDA_is_BV}, it suffices to show that $B^\vee([c])=0$. Recall that we have the following commutative diagram 
    \begin{center}
    \begin{tikzcd}
    {\Ext_A(A,DA)_\bullet} \arrow[rr, "\widetilde\mu"] &                                & {\Ext_{A^e}(A,DA)_\bullet=:HH_\bullet^\ast(A,DA)} \arrow[ld, "\eval"] \\
                                                          & H(DA)_\bullet \arrow[lu, "H(ev(1_A))^{-1}", "\simeq"'] &                                                            
    \end{tikzcd}    
    \end{center}
    and that $[c]=\widetilde\mu \circ H(ev(1_A))^{-1} ([\zeta])$ where $\zeta=\phi(1)$.
    Let's give an explicit description of the lift. The composite $\widetilde\mu \circ ev(1_A)^{-1}$ is of the form
        \begin{align*}
        \widetilde\mu \circ ev(1_A)^{-1}:  & DA_\bullet \to \Hom_A(A, DA)_\bullet\to \Hom_{A^e}(P, DA)_\bullet \\
         ~          & \xi \mapsto (a \xmapsto{g_\xi} a.\xi) \mapsto g_\xi\circ q_A. 
        \end{align*} 
        where $q_A: P_\bullet\xrightarrow{\simeq} A_\bullet$ is a cofibrant approximation of $A_\bullet$ as an $A^e_\bullet$-module. In particular, one can take the following quasi-isomorphism (Proposition \ref{prop:bar-cofib-approx})
            \[q_A:B(A, \overline{A}, A)_\bullet\to A_\bullet \text{ given by }
            q_A(a[a_1, \ldots, a_k]b)=\begin{cases}
        0 & \text{if } k >0 \\
        a.b & \text{if } k=0. 
    \end{cases}\]
    Note that $q_A$ can be written as the composite of $\id_A \boxtimes proj \boxtimes \id_A$ and the product on $A_\bullet$:
    \[B(A, \overline{A}, A)_\bullet=(A\boxtimes T(s\overline{A}) \boxtimes A)_\bullet \xmapsto{\id_A \boxtimes proj \boxtimes \id_A} (A\boxtimes F_{\overline{0}}(S^0) \boxtimes A)_\bullet \simeq (A\boxtimes A)_\bullet \to A\]
    where $proj:T(s\overline{A})_\bullet\to F_{\overline{0}}(S^0)_\bullet$ is given by
        \[proj(a[x]b)=\begin{cases}
        x & \text{if } x\in F_{\overline{0}}(S^0) \\
        0 & \text{else.} 
    \end{cases}\]
    The following sequence of isomorphisms
    \[HC_\bullet(A, DA) \simeq \Hom_{Ch(R)^{\pGM}}(T(s\overline{A}), DA)_\bullet \simeq \Hom_{Ch(R)^{\pGM}}(A\boxtimes T(s\overline{A}), F_{\overline{0}}(\mathbb{S}^0))_\bullet \]
    sends a morphism of $A_\bullet$-bimodules
    \begin{align*}
        \phi:   &B(A,\overline{A},A)_\bullet \to DA_\bullet \\
        ~       &a[x]b \mapsto \phi(a[x]b)
    \end{align*}
    to a morphism $\overline{\phi}\in D(HC^\bullet_\ast(A))_\bullet$ given by
    \begin{align*}
        \overline{\phi}:    &(A\boxtimes T(s\overline{A}))_\bullet \to F_{\overline{0}}(\mathbb{S}^0)_\bullet \\
                ~           &a[x] \mapsto \phi(1[x]1)(a).
    \end{align*}   
    In particular, for $\phi=\widetilde\mu \circ ev(1_A)^{-1}(\zeta)=g_\zeta\circ q_A$, we have 
    \[\overline{\phi}(a[x])=
    \begin{cases}
        g_\zeta(x)(a)=x.\zeta(a) & \text{if } x\in F_{\overline{0}}(S^0) \\
        0 & \text{else.} 
    \end{cases}\]
    Since we work on the normalized Hochschild (co)chains, using the definition of the dual of Connes boundary map, one can see that $B^\vee(c)=B^\vee(g_\zeta\circ q_A)=0$.
    \end{proof}
    One immediately gets the following result.
    \begin{corollaire}\label{coro:BV_on_APL}
    Let $X$ be a compact, oriented, second countable pseudomanifold. There exists a Batalin-Vilkovisky algebra structure on $HH^\ast_\bullet(\widetilde A_{PL, \bullet}(X))$.
    \end{corollaire}
    In the next section, we will show that the perverse BV-algebra structure is invariant under quasi-isomorphism of pDGAs, hence, we get such a structure on $HH^\ast_\bullet(\widetilde N^\ast_{\bullet}(X;\mathbb{Q}))$.

\section{Topological invariance}\label{sect:top_invar}
    In this section, we would like to show that when $X$ and $Y$ are two homeomorphic or stratified homotopy equivalent pseudomanifolds, there exists an isomorphism of perverse BV-algebras
        \[HH^\ast_\bullet(\widetilde N^\ast_\bullet(X; \mathbb{Q})) \simeq HH^\ast_\bullet(\widetilde N^\ast_\bullet(Y; \mathbb{Q})).\]
    There are two things we must prove before:
    \begin{enumerate}
        \item for any field  $\mathbb{F}$, there is an isomorphism of perverse Gerstenhaber algebras
        \[HH^\ast_\bullet(\widetilde N^\ast_\bullet(X; \mathbb{F})) \simeq HH^\ast_\bullet(\widetilde N^\ast_\bullet(Y; \mathbb{F}))\]
        \item and one can endow $HH^\ast_\bullet(\widetilde N^\ast_\bullet(X; \mathbb{Q}))$ with a perverse BV algebra structure. 
    \end{enumerate}    
    This will be done in the next two subsections. 
    
    \subsection{Preserving the Gerstenhaber algebra structure}
    Let $f:A_\bullet \to B_\bullet$ be a quasi-isomorphism of pDGAs. In this subsection, we ask ourselves under which assumptions does $f$ induce an isomorphism of perverse Gerstenhaber algebras between $HH_\bullet(A)$ and $HH_\bullet(B)$. 
    \begin{proposition}
    Let $f:A_\bullet \to B_\bullet$ be a quasi-isomorphism of pDGAs. If $R$ is a field or if $A_\bullet$ and $B_\bullet$ are cofibrant objects in $Ch(R)^{\pGM}$ that are augmented as pDGAs, then $f$ induces morphisms on the Hochschild cochain complexes
    \[HC_\bullet(A) \xrightarrow{HC(A,f)} HC_\bullet(A,B) \xleftarrow{HC(\tilde f,B)} HC_\bullet(B)\]
    which are quasi-isomorphisms of pDGAs.
    \end{proposition}
    \begin{proof}
        We first describe the morphisms $HC(A,f)$ and $HC(\tilde f,B)$. Let $q_A: \mathbb{B}(A)_\bullet \to A_\bullet$ and $q_B: \mathbb{B}(B)_\bullet\to B_ \bullet$ be cofibrant approximations of $A_\bullet$ and $B_\bullet$, respectively. They are given in Proposition \ref{prop:bar-cofib-approx}. Notice that $f$ induces a morphism between the bar constructions
        \[\begin{array}{rccc}
            \tilde f:& \mathbb{B}(A)_\bullet &\to &\mathbb{B}(B)_\bullet \\
            ~        & a_0[a_1|\ldots|a_n]a_{n+1} &\mapsto & (-1)^{\sum\limits_{k=1}^{n+1}\Real{f}\eps_k} f(a_0)[f(a_1)|\ldots|f(a_n)]f(a_{n+1}).      
        \end{array}\]
        where $\eps_k=\Real{a_0}+\sum_{j=1}^{k-1}\Real{s(a_j)}$.
        It fits in the following commutative diagram of perverse chain complexes
        \[
        \begin{tikzcd}
        \mathbb{B}(A)_\bullet \arrow[d, "q_A"', two heads] \arrow[r, "\tilde f"] & \mathbb{B}(B)_\bullet \arrow[d, "q_B", two heads] \\
        A_\bullet \arrow[r, "f"']                                                &  B_\bullet                                        
        \end{tikzcd}
        \]
        and hence, it is a quasi-isomorphism of perverse chain complexes. 
        Notice that $T(s\overline{A})_\bullet \simeq (F_{\overline{0}}(\mathbb{S}^0)\boxtimes T(s\overline{A}) \boxtimes F_{\overline{0}}(\mathbb{S}^0))_\bullet$ is a subcomplex of $\mathbb{B}(A)_\bullet$. We also denote by $\tilde f$ the restriction $\tilde f : T(s\overline{A})_\bullet \to  T(s\overline{B})_\bullet$. The morphisms given in the statement of the proposition fit in the following diagram
        \[\adjustbox{scale=0.8}{\begin{tikzcd}[column sep=huge]
        	{HC_\bullet(A)} & {HC_\bullet(A,B)} & {HC_\bullet(B)} \\
        	{\Hom_{A^e}(\mathbb{B}(A), A)_\bullet} & {\Hom_{A^e}(\mathbb{B}(A), B)_\bullet} & {\Hom_{B^e}(\mathbb{B}(B), B)_\bullet} \\
        	{\Hom_{Ch(R)^{\pGM}}(T(s\overline{A}), A)_\bullet} & {\Hom_{Ch(R)^{\pGM}}(T(s\overline{A}), B)_\bullet} & {\Hom_{Ch(R)^{\pGM}}(T(s\overline{B}), B)_\bullet}
        	\arrow["{\Hom_{A^e}(\mathbb{B}(A), f)}", from=2-1, to=2-2]
        	\arrow[from=2-3, to=2-2]
        	\arrow["{\scriptscriptstyle{\Hom(\tilde f, B)}}"', from=3-3, to=3-2]
        	\arrow["{\scriptscriptstyle{\Hom(T(s\overline{A}), f)}}", from=3-1, to=3-2]
        	\arrow["\simeq"{description}, draw=none, from=2-1, to=3-1]
        	\arrow["\simeq"{description}, draw=none, from=2-3, to=3-3]
        	\arrow["\simeq"{description}, draw=none, from=3-2, to=2-2]
        	\arrow["{HC(A,f)}", from=1-1, to=1-2]
        	\arrow["{HC(\tilde f, B)}"', from=1-3, to=1-2]
        	\arrow["\simeq"{description}, draw=none, from=1-1, to=2-1]
        	\arrow["\simeq"{description}, draw=none, from=1-3, to=2-3]
        	\arrow["\simeq"{description}, draw=none, from=1-2, to=2-2]
        \end{tikzcd}}\]
        where perverse chain complexes on the same column are isomorphic. Since the bar constructions are cofibrant perverse chain complexes, by Remark \ref{rem:quillen_adj_pch}, the morphisms on the bottom row are quasi-isomorphisms. Hence, $HC(A,f)$ and $HC(\tilde f, B)$ are also quasi-isomorphisms. 
        
        We now show that we have morphisms of pDGAs. $HC^\ast_\bullet(A)$ and $HC^\ast_\bullet(B)$ are equipped with the cup product defined in Definition \ref{def:cup_prod}. One can define an associative product (denoted again by $\cup$) on $HC_\bullet(A,B)$ which makes it a pDGA. This multiplication is given for $g_1,g_2\in HC^\ast_\bullet(A,B)$ and $[a_1|\ldots|a_k]\in T(s\overline{A})_\bullet$ by 
        \[g_1\cup g_2[a_1|\ldots|a_k]=\sum_{i=1}^{k-1}(-1)^{\Real{g_2}\overline{\eps}_i}g_1[a_1| \ldots | a_{i}]g_2[a_{i+1} | \ldots |  a_{k}]\]    
        where $\overline{\eps_i}=\sum_{j \leq i} \Real{s(a_j)}=\sum_{j \leq i} (\Real{a_j}-1)$. Using the fact that $f$ is a morphism of pDGA, we see that for $\alpha_1, \alpha_2\in HC_\bullet(A,A)$
        \[HC(A,f)(\alpha_1\cup \alpha_2)=f\circ(\alpha_1\cup \alpha_2)=f\circ\alpha_1\cup f\circ\alpha_2=HC(A,f)(\alpha_1)\cup HC(A,f)(\alpha_2).\]
        In other words, $HC(A,f)$ is a morphism of pDGA. Using the definition of $\tilde f$ and by writing things out, one sees that $HC(\tilde f, B)$ is also a morphism of pDGA. 
    \end{proof}
    In what follows, we consider the composition
    \[HH(f):HH_\bullet (A) \xrightarrow{HH(A,f)} HH_\bullet(A, B) \xrightarrow{HH(\tilde f,B)^{-1}} HH_\bullet(B).\]
    The next result adapts \cite[Proposition 3.3]{FMT05} of Félix, Menichi and Thomas to pDGAs.
    \begin{proposition}\label{prop:HH_iso_Gert_topinvar}
        Let $f:A_\bullet \to B_\bullet$ be a quasi-isomorphism of pDGAs. The morphism $HH(f)$ is an isomorphism of Gerstenhaber algebras if one of the following conditions is verified
    \begin{itemize}
        \item $R$ is a field and $B_\bullet$ is a cofibrant object in $Ch(R)^{\pGM}$, 
        \item or $A_\bullet$ and $B_\bullet$ are cofibrant objects in $Ch(R)^{\pGM}$ that are augmented as pDGAs.
    \end{itemize}
    \end{proposition}

    \begin{proof}
        We already know that $HH(f)$ is an isomorphism of pDGA, we just have to check that it is a morphism of perverve graded Lie algebras. In \cite[Theorem 3.4]{Hov09}, Hovey shows that there is a model category structure on the category of perverse differential graded algebras. Hence, the morphism $f:A_\bullet \to B_\bullet$ factors as 
        \[A_\bullet \xhookrightarrow{i} C_\bullet \xtwoheadrightarrow{p} B_\bullet\]
        in the category of pDGAs with $i$ a cofibration and $p$ an acyclic fibration. 
        Notice that we have the following commutative diagram
        \[\adjustbox{scale=0.85}{
        \begin{tikzcd}
        {HC_\bullet(A)} \arrow[d, "\id"] \arrow[r, "{HC(A,i)}"]                         & {HC_\bullet(A,C)} \arrow[d, "\id"]        & {HC_\bullet(C)} \arrow[l, "{HC(\tilde i, C)}"'] \arrow[r, "{HC(C, p)}"] & {HC_\bullet(C,B)} \arrow[d, "\id"']              & {HC_\bullet(B)} \arrow[l, "{HC(\tilde p, B)}"'] \arrow[d, "\id"']                           \\
        {HC_\bullet(A)} \arrow[r, "{HC(A,i)}"'] \arrow[rr, "{HC(A,f)}"', bend right=60] & {HC_\bullet(A,C)} \arrow[r, "{HC(A,p)}"'] & {HC_\bullet(A,B)}                                                          &            {HC_\bullet(C,B)} \arrow[l, "{HC(\tilde i, B)}"] & {HC_\bullet(B).} \arrow[l, "{HC(\tilde p, B)}"] \arrow[ll, "{HC(\tilde f, B)}", bend left=60]
        \end{tikzcd}}\]
        This implies that $HH(f)=HH(p)\circ HH(i)$. We are going to show that $HH(p)$ and $HH(i)$ are isomorphisms of Gerstenhaber algebra. Notice that $i:A_\bullet \hookrightarrow C_\bullet$ has a retraction $r$ and that $p: C_\bullet \to B_\bullet$ admits a section $s$ in the category $Ch(R)^{\pGM}$. Indeed, since $A_\bullet$ is fibrant and $B_\bullet$ is cofibrant as perverse chain complexes the following diagrams have lifts
        \[\begin{tikzcd}
        A_\bullet \arrow[d, "i", hook] \arrow[r, "\id"] & A_\bullet &                                                      & C_\bullet \arrow[d, "p", two heads] \\
        C_\bullet \arrow[ru, "r"', dashed]              &           & B_\bullet \arrow[r, "\id"'] \arrow[ru, "s", dashed] & B_\bullet.                          
        \end{tikzcd}\]
        They induce sections in the following diagram
        \[\adjustbox{scale=0.9}{\begin{tikzcd}
        {HC_\bullet(A)} \arrow[r, "{HC(A,i)}"] & {HC_\bullet(A,C)} \arrow[r, "{HC(\tilde r, C)}"', bend right=60] & {HC_\bullet(C)} \arrow[l, "{HC(\tilde i, C)}"', two heads] \arrow[r, "{HC(C,p)}", two heads] & {HC_\bullet(C,B)} \arrow[l, "{HC(C,s)}", bend left=60] & {HC_\bullet(B).} \arrow[l, "{HC(\tilde p, B)}"']
        \end{tikzcd}}\]
        Let $\alpha_1,\alpha_2 \in HH_\bullet(A)$ and set $\gamma_j=HH(i)(\alpha_j)$ for $j=1,2$. We want to show that
        \[HH(i)[\alpha_1, \alpha_2]=[\gamma_1, \gamma_2].\]
        Let $a_1, a_2$ be cycles in $HC_\bullet(A)$ such that their class in homology correspond to $\alpha_1$ and $\alpha_2$ respectively. There exists cycles $c_1,c_2\in HC_\bullet(C)$ whose classes in homology are respectively $\gamma_1, \gamma_2$ and such that
        \[HC(A,i)(a_j)=i\circ a_j=HC(\tilde i, C)(c_j)= c_j\circ \tilde i\]
        for $j=1,2$. Indeed, one can take $c_j=HC(\tilde r, C)\circ HC(A,i)(a_j)$. We compute
        \[\begin{array}{rcl}
        HC(A, i)[a_1, a_2]&=&i\circ (a_1\circ a_2) - (-1)^{(\Real{a_1}-1)(\Real{a_2}-1)}i\circ (a_2\circ a_1)\\
        ~ &                =&(c_1\circ c_2)\circ \tilde i- (-1)^{(\Real{a_1}-1)(\Real{a_2}-1)} (c_2\circ c_1)\circ \tilde i\\
        ~ &                =&(c_1\circ c_2)\circ \tilde i- (-1)^{(\Real{c_1}-1)(\Real{c_2}-1)} (c_2\circ c_1)\circ \tilde i  \\
        ~ &                =&HC(\tilde i, C)[c_1, c_2].
        \end{array}\]
        In homology, this gives us 
        \[HH(i)[\alpha_1, \alpha_2] = HH(\tilde i, C)^{-1}\circ HC(A,i)[\alpha_1, \alpha_2]=[\gamma_1, \gamma_2].\]
        By using similar arguments, we can show that for cycles $\gamma_1,\gamma_2\in HH_\bullet(C)$ 
        \[HH(p)[\gamma_1, \gamma_2]=[HH(p)(\gamma_1), HH(p)(\gamma_2)].\]
        Finally, this proves that $HH(f)=HH(p)\circ HH(i)$ is a morphism of Gerstenhaber algebras. 
    \end{proof} 
    Using the topological invariance of the blown-up intersection cochain complex (Proposition \ref{prop:bup_prop}) and the fact that it is a cofibrant perverse cochain complex on a field (Corollary \ref{coro:ex_cofib_cplx_field}), we get the following result. 
    \begin{corollaire}\label{coro:top_invar_Gerst}
        If $X$ and $Y$ are two homeomorphic or stratified homotopy equivalent pseudomanifolds then there exists an isomorphism of perverse Gerstenhaber algebras
        \[HH^\ast_\bullet(\widetilde N^\ast_\bullet(X; \mathbb{F})) \simeq HH^\ast_\bullet(\widetilde N^\ast_\bullet(Y; \mathbb{F})).\]
    \end{corollaire} 

    \subsection{Preserving the pBV algebra structure}
    As we said in the introduction to this section, we now prove that under some assumptions on the pseudomanifold $X$ there exists a perverse BV-algebra structure on $HH^\ast_\bullet(\widetilde N^\ast_\bullet(X; \mathbb{Q}))$ . We will need the following lemma.
    \begin{lemma}\label{lem:BV_transfer} 
        Let $(A_\bullet, \cup_A, [-,-]_A)$ and $(B_\bullet, \cup_B, [-,-]_B)$ be two perverse Gerstenhaber algebras which are isomorphic. If one of them is a perverse BV-algebra then so is the other one.
    \end{lemma}
    \begin{proof}
        Suppose that $A_\bullet$ is a perverse BV-algebra and denote by $\Delta_A$ the $\Delta$ operator on $A_\bullet$. Let $f:A_\bullet \to B_\bullet$ be an isomorphism of perverse Gerstenhaber algebras. We set $\Delta_B:=(-1)^{\Real{f}}f\circ \Delta_A\circ f^{-1}$. We clearly have $\Delta_B \circ \Delta_B=0$. We just need to check that $[-,-]_B$ encodes the deviation of $\Delta_B$ from being a derivation with respect to $\cup_B$.
        Let $b_1, b_2 \in B_\bullet$, we set $a_1=f^{-1}(b_1)$ and $a_2=f^{-1}(b_2)$. We compute $[b_1,b_2]_B$, it is equal to 
        \begin{align*}
        f([a_1, a_2]_A) &=(-1)^{\Real{a_1}}f(\Delta_A(a_1\cup_A a_2) -a_1\cup_A \Delta_A(a_2) -(-1)^{\Real{a_1}}\Delta_A(a_1)\cup_A a_2) \\
                ~       &=(-1)^{\Real{a_1}+\Real{f}}\Delta_B(b_1\cup_B b_2) -b_1\cup_B \Delta_B(b_2) -(-1)^{\Real{a_1}+\Real{f}}\Delta_B(b_1)\cup_B b_2) \\
                ~       &=(-1)^{\Real{b_1}}\Delta_B(b_1\cup_B b_2) -b_1\cup_B \Delta_B(b_2) -(-1)^{\Real{b_1}}\Delta_B(b_1)\cup_B b_2)                
        \end{align*}
    \end{proof}
    \begin{proposition}\label{prop:BV_on_BUP}
        Let $X$ be a compact, oriented, second countable pseudomanifold. There exists a perverse Batalin-Vilkovisky algebra structure on $HH^\ast_\bullet(\widetilde N^\ast_\bullet(X; \mathbb{Q}))$.
    \end{proposition}
    \begin{proof}
        Recall that Proposition \ref{prop:integ_polyforms} gives quasi-isomorphisms of pDGAs 
        \[f_1:\widetilde A_{PL, \bullet}^\ast(X) \to (\widetilde{A_{PL}\otimes C})^\ast_\bullet(X) \text{ and } f_2:\widetilde N_\bullet^\ast(X; \mathbb{Q}) \to (\widetilde{A_{PL}\otimes C})^\ast_\bullet(X).\]
        Hence, by Proposition \ref{prop:HH_iso_Gert_topinvar}, we have isomorphisms of perverse Gerstenhaber algebras
        \[HH^\ast_\bullet(\widetilde A^\ast_{PL, \bullet}(X))\simeq HH^\ast_\bullet((\widetilde{A_{PL}\otimes C})^\ast_\bullet(X))\simeq HH^\ast_\bullet(\widetilde N^\ast_\bullet(X; \mathbb{Q}))\]              
        By Corollary \ref{coro:BV_on_APL}, $HH^\ast_\bullet(\widetilde A_{PL, \bullet}(X))$ is a pBV algebra. Using the previous lemma, we can endow $HH^\ast_\bullet(\widetilde N^\ast_\bullet(X; \mathbb{Q}))$ with a pBV algebra structure.
    \end{proof}
    While there is a canonical Gerstenhaber algebra structure on $HH^\ast_\bullet(A)$ for $A_\bullet$, the BV structure depends on the choice of a class in $HH^\ast_\bullet(A,DA)_\bullet$. Hence, one needs additional hypothesis to ensure that a quasi-isomorphism $f:A_\bullet \to B_\bullet$ of pDGAs preserves the pBV algebra structures on the Hochschild cohomologies. 

    \begin{proposition}\label{prop:HH_iso_BV_topinvar} 
        Let $f:A_\bullet \to B_\bullet$ be a quasi-isomorphism of commutative pDGAs. We denote by $f^\vee: DB_\bullet \to DA_\bullet$ the dual morphism. Suppose that there are cycles $M\in DA_\bullet$ and $N\in DB_\bullet$ such that the morphisms
        \[\begin{array}{rclcrcl}
          H(A)_\bullet   & \xrightarrow{\simeq} & H(DA)_\bullet & \quad & H(B)_\bullet   & \xrightarrow{\simeq} & H(DB)_\bullet\\
          a   & \mapsto & a.[M] & \quad & b   & \mapsto & b.[N]  
        \end{array}\]
        are isomorphisms of left $H(A)_\bullet$-modules and left $H(B)_\bullet$-modules respectively. If $f^\vee(N)=M$ then $HH(f)$ is an isomorphism of BV-algebras.
    \end{proposition}

    \begin{proof}
    By Proposition \ref{prop:Commu_DPDA}, there exist cocycles $c\in HC^\ast_\bullet(A,DA)$ and $d\in HC^\ast_\bullet(B,DB)$ such that the morphisms 
    \[\begin{array}{rclcrcl}
        HH^\ast_\bullet(A) &\xrightarrow{\simeq} &HH^\ast_\bullet(A, DA) & \quad & HH^\ast_\bullet(B) &\xrightarrow{\simeq} &HH^\ast_\bullet(B, DB) \\
        a            &\mapsto &a.[c]&\quad & b            &\mapsto &b.[d]
    \end{array}\]   
    are isomorphisms of $HH^\ast_\bullet(A,A)$-modules and $HH^\ast_\bullet(B,B)$-modules respectively. 
    Recall that, by Theorem \ref{thm:DPDA_is_BV}, the pBV algebra structures on $HH^\ast_\bullet(A,A)$ and $HH^\ast_\bullet(B,B)$ are given by the following operators: 
    \[\Delta_A: HH^\ast_\bullet(A) \to HH^{\ast-1}_\bullet(A) \text{ and } \Delta_B: HH^\ast_\bullet(A) \to HH^{\ast-1}_\bullet(A)\] which are defined for $h\in HH^\ast_\bullet(A), g\in HH^\ast_\bullet(B)$
    by \[\Delta_A(h).[c] := B^\vee(h.[c]) \text{ and } \Delta_B(g).[d] := B^\vee(g.[d]).\]
    We just need to check that the following diagram commutes
    \[\begin{tikzcd}
    {HH^\ast_\bullet(A)} \arrow[d, "HH(f)"'] \arrow[r, "\Delta_A"] & {HH^{\ast-1}_\bullet(A)} \arrow[d, "HH(f)"] \\
    {HH^\ast_\bullet(B)} \arrow[r, "\Delta_B"]                     & {HH^{\ast-1}_\bullet(B)}.                   
    \end{tikzcd}\]
    We go back to the definition of the morphisms at the cochain level, we have the following diagram
    \[
    \begin{tikzcd}
    {HC_\bullet^\ast(A)} \arrow[d, "{HC(A,f)}"', ""{name=l1}] \arrow[r, "{\_\,.c}"]          & {HC_\bullet^\ast(A,DA)} \arrow[r, "B^\vee"]                                 & {HC_\bullet^{\ast-1}(A,DA)}                                 & {HC_\bullet^{\ast-1}(A)} \arrow[l, "{\_\,.c}"'] \arrow[d, "{HC(A,f)}"]                                  \\
    {HC_\bullet^\ast(A,B)} \arrow[r, "\xi_N\circ\_"]                              & {HC_\bullet^\ast(A,DB)} \arrow[u, "{HC(A,f^\vee)}"', ""{name=r1}] \arrow[r, phantom, "\scriptstyle{(3)}"]    & {HC_\bullet^{\ast-1}(A,DB)} \arrow[u, "{HC(A,f^\vee)}"']    & {HC_\bullet^{\ast-1}(A,B)} \arrow[l, "\xi_N\circ\_"']                        \\
    {HC_\bullet^\ast(B)} \arrow[u, "{HC(\tilde f, B)}", ""{name=l2}] \arrow[r, "{\_ \,. d}"] & {HC_\bullet^\ast(B,DB)} \arrow[u, "{HC(\tilde f, B)}"', ""{name=r2}] \arrow[r, "B^\vee"] & {HC_\bullet^{\ast-1}(B,DB)} \arrow[u, "{HC(\tilde f, B)}"'] & {HC_\bullet^{\ast-1}(B)} \arrow[l, "{\_ \,. d}"'] \arrow[u, "{HC(\tilde f, B)}"']
    \arrow[phantom,from=l1,to=r1,"\scriptstyle{(2)}"]
    \arrow[phantom,from=l2,to=r2,"\scriptstyle{(1)}"]
    \end{tikzcd}
    \]
    The commutativity of the right hand side will follow from the commutativity of the left hand side. As it was mentioned at the end of the proof of Proposition \ref{prop:Commu_DPDA}, the action of $c$ and $d$ correspond to composing with $\xi_M$ and $\xi_N$ respectively. Hence, the left hand size of the above diagram is given by
    \[\begin{tikzcd}
    {HC_\bullet^\ast(A)} \arrow[d, "{HC(A,f)}"', ""{name=l1}] \arrow[r, "\xi_M\circ\_"]        & {HC_\bullet^\ast(A,DA)}                                 \\
    {HC_\bullet^\ast(A,B)} \arrow[r, "\xi_N\circ\_"]                                & {HC_\bullet^\ast(A,DB)} \arrow[u, "{HC(A,f^\vee)}"', ""{name=r1}]    \\
    {HC_\bullet^\ast(B)} \arrow[u, "{HC(\tilde f, B)}", ""{name=l2}] \arrow[r, "\xi_N\circ\_"] & {HC_\bullet^\ast(B,DB)}. \arrow[u, "{HC(\tilde f, B)}"', ""{name=r2}]
    \arrow[phantom,from=l1,to=r1,"\scriptstyle{(2)}"]
    \arrow[phantom,from=l2,to=r2,"\scriptstyle{(1)}"]
    \end{tikzcd}\]
    The diagram (1) is clearly commutative. The commutativity of (2) follows from the following commutative diagram of $A_\bullet$-bimodules
    \[
    \begin{tikzcd}
    A_\bullet \arrow[r, "\xi_M"] \arrow[d, "f"'] & DA_\bullet                      \\
    B_\bullet \arrow[r, "\xi_N"']                & DB_\bullet. \arrow[u, "f^\vee"']
    \end{tikzcd}
    \]
    We know study the middle part (3). Using the isomorphisms 
    \[HC_\bullet^\ast(A,DA)\simeq D(HC_\ast(A))^\ast_\bullet\text{ and } HC^\bullet_\ast(A)\simeq (A\boxtimes T(s\overline{A}))_\ast^\bullet,\] we are led to show that the following diagram commutes
    \[\begin{tikzcd}
    D(A\boxtimes T(s\overline{A}))^\ast_\bullet \arrow[r, "B^\vee"]                                          & D(A\boxtimes T(s\overline{A}))^{\ast-1}_\bullet                                          \\
    D(B\boxtimes T(s\overline{B}))^\ast_\bullet \arrow[u, "(f\boxtimes \tilde f)^\vee"] \arrow[r, "B^\vee"'] & D(B\boxtimes T(s\overline{B}))^{\ast-1}_\bullet \arrow[u, "(f\boxtimes \tilde f)^\vee"']
    \end{tikzcd}\]
    where for $g\in D(B\boxtimes T(s\overline{B}))^\ast_\bullet$ and $a_0[a_1|\ldots|a_m]\in (A\boxtimes T(s\overline{A}))_\ast^\bullet$, $(f\boxtimes \tilde f)^\vee \circ g$ is given by
    \begin{align*}
        (f\boxtimes \tilde f)^\vee \circ g (a_0[a_1|\ldots|a_m])&(-1)^{(m+1)\Real{f}\Real{g}} g\circ (f\boxtimes \tilde f) (a_0[a_1|\ldots|a_m]) \\
        ~                                                       &(-1)^{\sum_{k=1}^m\Real{f}\eps_k+(m+1)\Real{f}\Real{g}} g(f(a_0)[f(a_1)|\ldots|f(a_m)])
    \end{align*}
    with $\eps_k=\Real{a_0}+\sum_{j=1}^{k-1}\Real{s(a_j)}$. This is indeed the case. Let $g\in D(B\boxtimes T(s\overline{B}))^\ast_\bullet$ and $a_0[a_1|\ldots|a_m]\in (A\boxtimes T(s\overline{A}))_\ast^\bullet$. 
    
    On one side we have
    \begin{align*}
        (B^\vee\circ &(f\boxtimes \tilde f)^\vee)(g)(a_0[a_{1,m}]) = (-1)^{(m+1)\Real{f}\Real{g}}B^\vee(g\circ(f\boxtimes \tilde f))(a_0[a_1|\ldots|a_m]) \\
        & =-(-1)^{((m+1)\Real{f}-1)(\Real{g}-1)}g\circ (f\boxtimes \tilde f)\left(B(a_0[a_1|\ldots|a_m])\right) \\        
        & =-(-1)^{((m+1)\Real{f}-1)(\Real{g}-1)}g\circ (f\boxtimes \tilde f)\left(\sum_{i=0}^m (-1)^{\sum\limits_{k<i}\Real{s(a_k)}\sum\limits_{k\geq i}\Real{s(a_k)}} 1[a_i|\ldots|a_m|a_0|\ldots]\right) \\
        & =-(-1)^{((m+1)\Real{f}-1)(\Real{g}-1)}\sum_{i=0}^m (-1)^{\sum\limits_{k<i}\Real{s(a_k)}\sum\limits_{k\geq i}\Real{s(a_k)}}(-1)^{\left(\sum\limits_{k=i}^m\eta_{i,k}+\sum\limits_{k=1}^{i-1}(\eta_{i,m}+\eps_k)\right)\Real{f}} \\
        &\qquad g(1[f(a_i)|\ldots|f(a_m)|f(a_0)|\ldots]).
    \end{align*}
    with $\eta_{k,l}=\eps_{l+1}-\eps_k$. On the other side, we have
    \begin{align*}
        ((f\boxtimes &\tilde f)^\vee \circ B^\vee)(g)(a_0[a_1|\ldots|a_m]) = (-1)^{\Real{g}} (f\boxtimes \tilde f)^\vee (g\circ B)(a_0[a_1|\ldots|a_m]) \\
        & =(-1)^{\Real{g}+(\Real{g}-1)(m+1)\Real{f}}(-1)^{\sum_{k=1}^m\Real{f}\eps_k} (g\circ B)(f(a_0)[f(a_1)|\ldots|f(a_m)])\\
        & =-(-1)^{((m+1)\Real{f}-1)(\Real{g}-1)}(-1)^{\sum_{k=1}^m\Real{f}\eps_k} \sum_{i=0}^m (-1)^{\sum\limits_{k<i}\Real{s(f(a_k))}\sum\limits_{k\geq i}\Real{s(f(a_k))}} \\
        &\qquad g(1[f(a_i)|\ldots|f(a_m)|f(a_0)|\ldots]).
    \end{align*}
    The only extra signs that appear in this second expression are $(-1)^{\Real{f}^2}$ but they cancel themselves out with $(-1)^{\Real{f}}$. 
    \end{proof}
    The next result follows from the topological invariance of the blown-up intersection cochain complex (Proposition \ref{prop:bup_prop}) and Poincaré duality (Theorem \ref{thm:Poincaré}). We can proceed just like we did in the proof of Proposition \ref{prop:BV_on_BUP} and use the blown-up of Sullivan's polynomial forms. It is a commutative pDGA, so we can apply the previous proposition.
    \begin{corollaire}\label{coro:top_invar_BV}
        If $X$ and $Y$ are two compact, oriented, second countable pseudomanifolds which are homeomorphic or stratified homotopy equivalent then there exists an isomorphism of perverse BV algebras
        \[HH^\ast_\bullet(\widetilde N^\ast_\bullet(X; \mathbb{Q})) \simeq HH^\ast_\bullet(\widetilde N^\ast_\bullet(Y; \mathbb{Q})).\]
    \end{corollaire} 

    \subsection{Poincaré homology sphere}
    In this subsection, we will prove the following result.
    \begin{proposition}\label{prop:Poincaré_homology_sphere}
        Let $M$ be the Poincaré $\mathbb{Q}$-homology sphere and let $\Sigma M$ be its suspension. We have an isomorphism of pBV-algebras
        \[HH^\ast_\bullet(\widetilde N^\ast_\bullet(\Sigma M; \mathbb{Q})) \simeq HH^\ast_\bullet(F_{\overline{0}}C^\ast(\mathbb{S}^4; \mathbb{Q}))\]
        where $C^\ast(\mathbb{S}^4; \mathbb{Q})$ denotes the singular cochains of $\mathbb{S}^4$ with coefficients in $\mathbb{Q}$.
    \end{proposition}
    In order to apply Proposition \ref{prop:HH_iso_BV_topinvar}, we must first give a quasi-isomorphism between $\widetilde N^\ast_\bullet(\Sigma M; \mathbb{Q})$ and $F_{\overline{0}}C^\ast(\mathbb{S}^4; \mathbb{Q})_\bullet$. We can use the \emph{Thom-Pontryagin collapse map} \cite[Chapter II - Section 16]{Bre93}. 
    \begin{proposition}
        Let $X$ be a closed pseudomanifold of formal dimension $n$. Then for any $\overline{p}\in \pGM$, the \emph{Thom-Pontryagin collapse map} $c:X\to \mathbb{S}^n$ induces an isomorphism
        \[I^{\overline{p}}H_n(X; R) \simeq H_n(\mathbb{S}^n; R)\]
        if and only if $X$ is orientable.
    \end{proposition}
    \begin{proof}
        The pseudomanifold $X$ is a filtered space (Definition \ref{def:filtered_space}), it is equipped with a filtration by closed subspaces $\{X_i\}_{0\leq n}$. Let $x$ be a point in the regular part $X_n\setminus X_{n-1}$ (which is an $n$-manifold). There exists an open neighborhood $U_x$ of $x$ in $X_n\setminus X_{n-1}$ and an homeomorphism $\phi$ from $U_x$ to $\mathbb{D}^n$, the open disk in $\mathbb{R}^n$. The $n$-sphere $\mathbb{S}^n$ in $\mathbb{R}^{n+1}$ can be seen as the \emph{Alexandroff compactification} of $\mathbb{D}^n$ i.e. it is obtained by adjoining a single point, which we denote $\infty$. The Thom-Pontryagin collapse at $x$ is the continuous map $c_x: X \to \mathbb{S}$ which sends a point $p\in U_x$ to $\phi(p)$ and the other points are sent to $\infty$. Note that this is a stratified map if we equip $\mathbb{S}^n$ with the trivial filtration. It induces, for every $\overline{p}\in \pGM$, a morphism $c_{x, \ast}:I^{\overline{p}}H_n(x)\to H_n(\mathbb{S}^n)$ which fits in the following commutative diagram.  \[\begin{tikzcd}
        	{I^{\overline{p}}H_n(X; R)} & {H_n(\mathbb{S}^n; R)} \\
        	{I^{\overline{p}}H_n(X, X\setminus U_x)} & {H_n(\mathbb{S}^n, \infty)} \\
        	{I^{\overline{p}}H_n(V, V\setminus U_x)} & {H_n(V, V\setminus U_x)} 
        	\arrow["{c_{x, \ast}}", from=1-1, to=1-2]
        	\arrow["\simeq"', "(1)", from=1-1, to=2-1]
        	\arrow["\simeq"', "(2)", from=2-1, to=3-1]
        	\arrow["\simeq"',"(3)", from=3-1, to=3-2]
        	\arrow["\simeq"',"(4)", from=3-2, to=2-2]
        	\arrow["\simeq"{description}, draw=none, from=2-2, to=1-2]
        \end{tikzcd}\] 
        where $V$ is an open in $X_n\setminus X_{n-1}$ which contains $U_x$.
        Note the following facts:
        \begin{itemize}
            \item $(1)$ is an isomorphism if and only if $X$ is orientable.
            \item $(2)$ is an isomorphism by excision.
            \item $(3)$ is an isomorphism because $V$ is a submanifold.
            \item $(4)$ is also an isomorphism by considering the long exact sequence induced in homology by $C_\ast(V\setminus U_X; R) \hookrightarrow C_\ast(V; R) \twoheadrightarrow C_\ast(V; R)/C_\ast(V\setminus U_X; R)$.
        \end{itemize}
        Hence, $c_{x, \ast}$ is an isomorphism.
    \end{proof}
    We now proceed to the proof of Proposition \ref{prop:Poincaré_homology_sphere}. 
    \begin{proof}
        We denote by $M$ the Poincaré $\mathbb{Q}$-homology sphere. We apply the previous proposition to $X=\Sigma M$. For every $\overline{p}\in\pGM$ there exists an isomorphism
        \[I^{\overline{p}}H_4(\Sigma M; \mathbb{Q}) \simeq H_4(\mathbb{S}^4; \mathbb{Q}).\]
        Furthermore, by Theorem \ref{thm:Poincaré} (Poincaré duality), we have a quasi-isomorphism
        \[\widetilde N^0_{\overline{p}}(\Sigma M; \mathbb{Q})\simeq H^0(\mathbb{S}^4; \mathbb{Q}).\]
        After additionnal computation (see \cite[Section 1.5]{CST18ration}), we can show that $\widetilde N^\ast_{\overline{p}}(\Sigma M; \mathbb{Q})$ and $H^\ast(\mathbb{S}^4; \mathbb{Q})$ are quasi-isomorphic for any $\overline{p}\in\pGM$. In other words, we have a quasi-isomorphism of perverse chain complexes
        \[\widetilde N^\ast_\bullet(\Sigma M; \mathbb{Q})\simeq F_{\overline{0}}C^\ast(\mathbb{S}^4; \mathbb{Q})_\bullet.\]
        It's even a quasi-isomorphism of pDGA. Going back to the proof of the previous proposition, one notices that the collapse map sends a fundamental class of $\Sigma M$ to a fundamental class of $\mathbb{S}^4$.  
        By applying Proposition \ref{prop:HH_iso_BV_topinvar} to the blown-up of Sullivan's polynomial forms we get an isomorphism of perverse BV algebras
        \[HH^\ast_\bullet(\widetilde A_{PL, \bullet}(\Sigma M)) \simeq HH^\ast_\bullet(F_{\overline{0}}C^\ast(\mathbb{S}^4; \mathbb{Q})).\]
        Finally, by Proposition \ref{prop:BV_on_BUP}, we get the sought isomorphism
        \[HH^\ast_\bullet(\widetilde N^\ast_\bullet(\Sigma M; \mathbb{Q})) \simeq HH^\ast_\bullet(F_{\overline{0}}C^\ast(\mathbb{S}^4; \mathbb{Q})).\]        
    \end{proof}

\section{Tensor product and Hochschild cohomology}\label{sect:tensor}
    
    In \cite{LZ14}, Le and Zhou show that if $A$ and $B$ are finite dimensional symmetric algebras over a field then we have an isomorphism of BV algebras
    between $HH^\ast(A\otimes B)$ and $ HH^\ast(A)\otimes HH^\ast(B)$. Following their approach, we will see what assumptions are needed on pDGAs $A_\bullet$ and $B_\bullet$ in order to get an isomorphism of pBV algebras
    \begin{equation}\label{eq:tensorprod-Hoch-coho}
        HH^\ast_\bullet(A\boxtimes B) \simeq (HH^\ast(A) \boxtimes HH^\ast(B))_\bullet.
    \end{equation}
    We first endow both of these objects with a perverse BV-algebra structure. 
    
    \subsection{Defining pBV algebra structures}
    Let $A_\bullet$ and $B_\bullet$ be two pDGAs. Instead of showing that there is a pBV algebra structure on $(HH^\ast(A) \boxtimes HH^\ast(B))_\bullet$ we will define a pBV algebra structure on the tensor product of two pBV algebras. The following statements are analogous to results known by Manin \cite[Chap III - Section 9.4]{Man99} and his proofs apply to the perverse objects as the perversities don't appear in the computations.
    \begin{lemma}
        Let $(A^\ast_\bullet, \Delta, \cup)$ be a pDGA where
        \begin{itemize}
            \item $\Delta$ is a differential of degree $-1$ and perverse degree $\overline{0}$
            \item and $\cup$ is an associative and graded commutative product of degree $0$ and perverse degree $\overline{0}$.
        \end{itemize}
        We define a product 
        \[[-,-]:(A \boxtimes A)_\bullet^\ast \to A_\bullet^{\ast - 1} \]
        by setting for $a,b\in A^\ast_\bullet$
        \[[a,b]:=(-1)^{\Real{a}}\Delta(a\cup b)-a\cup\Delta(b)-(-1)^{\Real{a}}\Delta(a)\cup b.\]
        If for every $a\in A^\ast_\bullet$, the operator $[a, -]:A^\ast_\bullet \to A^{\ast+\Real{a}-1}_\bullet$ is a derivation with respect to $\cup$, then $(A_\bullet^\ast, \cup, [-,-], \Delta)$ is a perverse Batalin-Vilkovisky algebra. 
    \end{lemma}
    \begin{proof}
        One just needs to check, by straightforward computation, that $[-,-]$ is a Lie algebra (skew commutativity and Jacobi identity). All the other properties expected from a perverse BV algebra are verified (Definition \ref{def:pBV}).
    \end{proof}
    \begin{proposition}
        Let $(A_\bullet, \cup_A, [-,-]_A, \Delta_A)$ and $(B_\bullet, \cup_B, [-,-]_B, \Delta_B)$ be two pBV algebras. Then, there exists a pBV algebra structure on their tensor product $(A\boxtimes B)_\bullet$ where the cup product $\cup$, the Gerstenhaber bracket $[-, -]$ and the $\Delta$ operator are given for homogeneous elements $a_1\otimes b_1, a_2\otimes b_2\in (A\boxtimes B)_\bullet$ by
        \begin{itemize}
            \item $(a_1\otimes b_1)\cup (a_2\otimes b_2)=(-1)^{\Real{a_2}\Real{b_1}}(a_1\cup_A a_2)\otimes (b_1\cup_B b_2)$,
            \item $[a_1\otimes b_1, a_2\otimes b_2]=(-1)^{(\Real{a_2}-1)\Real{b_1}}[a_1, a_2]_A\otimes (b_1\cup b_2) + (-1)^{\Real{a_2}(\Real{b_1}-1)}(a_1\cup_1 a_2)\otimes [b_1, b_2]_B$,
            \item $\Delta(a_1\otimes b_1)=\Delta_A(a_1)\otimes b_1 +(-1)^{\Real{a_1}}a_1\otimes \Delta_B(b_1)$.
        \end{itemize}
    \end{proposition}
    \begin{proof}
        We easily notice the following facts: 
        \begin{itemize}
            \item  the associativity and the graded commutativity of $\cup$ is inherited from $\cup_A$ and $\cup_B$,
            \item $\Delta$ is a differential.
        \end{itemize}
        We can remark that
        \begin{align*}
            [a_1\otimes b_1, a_2\otimes b_2]    &=(-1)^{\Real{a_1}+\Real{b_1}}\Delta((a_1\otimes b_1) \cup (a_2\otimes b_2))-(a_1 \otimes b_1) \cup\Delta(a_2\otimes b_2)\\
            ~                                   &-(-1)^{\Real{a_1}+\Real{b_1}}\Delta(a_1\otimes b_1)\cup (a_2\otimes b_2).
        \end{align*}
        Furthermore, we can show that for any $a_1\otimes b_1\in (A\boxtimes B)_\bullet$, the operator $[a_1\otimes b_1, -]$ is a derivation with respect to $\cup$. Hence, by the previous lemma, $(A\boxtimes B)_\bullet$ is endowed with a pBV algebra structure.
    \end{proof}

    We now move on to the Hochschild cohomology of a tensor product. There is no general way of extending the perverse BV-algebra structures found on $HH^\ast_\bullet(A)$ and $HH^\ast_\bullet(B)$ to $HH^\ast_\bullet(A\boxtimes B)$. We will state the main result for pDGAs over a field $\mathbb{F}$ and we will require additional assumptions on either $A_\bullet$ or $B_\bullet$. The following lemmas will be useful.
    \begin{lemma}\label{lemma:Bourbaki_tensorhom}
        Let $A_\bullet, B_\bullet, C_\bullet$ and $D_\bullet$ be perverse cochain complexes. If either one of the pairs $(A_\bullet, C_\bullet),\, (A_\bullet, B_\bullet)$ or $(C_\bullet, D_\bullet)$ is of the form $(F_{\overline{0}}(V)_\bullet, F_{\overline{0}}(W)_\bullet)$ with $V^\ast$ and $W^\ast$ chain complexes whose components are finite dimensional $\mathbb{F}$-vector spaces then we have an isomorphism of perverse chain complexes
        \[(\Hom_{Ch(R)^{\pGM}}(A, B) \boxtimes \Hom_{Ch(R)^{\pGM}}(C, D))_\bullet \simeq \Hom_{Ch(R)^{\pGM}}(A\boxtimes C, B\boxtimes D)_\bullet.\]
    \end{lemma}
    \begin{proof}
        We suppose that $C_\bullet=F_{\overline{0}}(V)_\bullet$ and $D_\bullet=F_{\overline{0}}(W)_\bullet$ where $V$ and $W$ are chain complexes whose components are are finite dimensional vector spaces. The other cases are treated in a similar way. Let $\overline{r}\in \pGM$ and $k\in \mathbb{Z}$, notice that we have
        \[\Hom_{Ch(R)^{\pGM}}(A, F_{\overline{0}}(V))_{\overline{r}}^k=\Hom_{Ch(R)}(A_{\overline{r}}, V)^k \text{ and } (A\boxtimes F_{\overline{0}}(V))_{\overline{r}}^k=(A_{\overline{r}}\otimes V)^k.\]
        The first equality implies that $\Hom_{Ch(R)^{\pGM}}(F_{\overline{0}}(V), F_{\overline{0}}(W))_{\overline{r}}^k=F_{\overline{0}}(\Hom_{Ch(R)}(V, W))^k_{\overline{r}}$. Using these relations, we find that the term $(\Hom_{Ch(R)^{\pGM}}(A, B) \boxtimes \Hom_{Ch(R)^{\pGM}}(C, D))_{\overline{r}}^k$ is equal to 
        \[(\Hom_{Ch(R)^{\pGM}}(A, B)_{\overline{r}}\otimes \Hom_{Ch(R)}(V, W))^k\]
        \[=\lim\limits_{\overline{r}\leq \overline{q}-\overline{p}} \prod_{j_2-j_1-i_2+i_1=k} \Hom_R(A_{\overline{p}}^{i_1}, B_{\overline{q}}^{i_2})\otimes \Hom_R(V^{j_1}, W^{j_2}).\]
        Furthermore, we have 
        \[\Hom_{Ch(R)^{\pGM}}(A\boxtimes C, B\boxtimes D)_{\overline{r}}^k=\lim\limits_{\overline{r}\leq \overline{q}-\overline{p}} \prod_{j-i=k} \Hom_R((A_{\overline{p}}\otimes V)^i, (B_{\overline{q}}\otimes W)^j)\]
        \[=\lim\limits_{\overline{r}\leq \overline{q}-\overline{p}} \prod_{j_2-j_1-i_2+i_1=k} \Hom_R(A_{\overline{p}}^{i_1}\otimes V^{j_1}, B_{\overline{q}}^{i_2}\otimes W^{j_2}).\]        
        Since $V^{j_1}$ and $W^{j_2}$ are finite dimensional vector spaces, we have an isomorphism
        \[\Hom_R(A_{\overline{p}}^{i_1}, B_{\overline{q}}^{i_2})\otimes \Hom_R(V^{j_1}, W^{j_2}) \simeq \Hom_R(A_{\overline{p}}^{i_1}\otimes V^{j_1}, B_{\overline{q}}^{i_2}\otimes W^{j_2}).\]
        This concludes the proof. 
    \end{proof}   
        The next lemma follows by noticing that over a field all chains complexes are cofibrant and that the left adjoint $F_{\overline{0}}(-)$ is part of a Quillen adjunction. 
    \begin{lemma}
        If $V^\ast$ is a chain complex whose components are finite dimensional $\mathbb{F}$-vector spaces then the perverse chain complex $F_{\overline{0}}(V)_\bullet$ and its linear dual are cofibrant objects in $(Ch(R))^{\pGM}$.
    \end{lemma}
    
    \begin{proposition}\label{prop:BV on Hoch of tensor}
        Let $A_\bullet$ and $B_\bullet$ be commutative pDGAs over a field $\mathbb{F}$ such that one of them is equal to $F_{\overline{0}}(V)_\bullet$ with $V^\ast$ a commutative DGA whose components are finite dimensional vector spaces. Suppose that there is an isomorphism of left $H(A)_\bullet$-modules $H(A)_\bullet\to H(DA)_\bullet$ and an isomorphism of left $H(B)_\bullet$-modules $H(B)_\bullet\to H(DB)_\bullet$. Then there exists an isomorphism of left $H(A\boxtimes B)_\bullet$-modules $H(A\boxtimes B)_\bullet \xrightarrow{\simeq}  H(D(A\boxtimes B))_\bullet$. 
    \end{proposition}
    \begin{proof}
    We use the notation from Proposition \ref{prop:Commu_DPDA}. Suppose that there are cycles $M\in DA_\bullet$ and $N\in DB_\bullet$ such that the morphisms
        \[\begin{array}{crclccrcl}
          H(\xi_M):&H(A)_\bullet   & \xrightarrow{\simeq} & H(DA)_\bullet & \quad & H(\xi_N):& H(B)_\bullet   & \xrightarrow{\simeq} & H(DB)_\bullet\\
          ~ & a   & \mapsto & a.[M] & \quad & ~ & b   & \mapsto & b.[N]  
        \end{array}\]
        are isomorphisms of left $H(A)_\bullet$-modules and left $H(B)_\bullet$-modules respectively. 
    We obtain a morphism $(A\boxtimes B)_\bullet \to D(A\boxtimes B)_\bullet$ by factoring through $(DA\boxtimes DB)_\bullet$. Suppose that $B_\bullet=F_{\overline{0}}(V)_\bullet$ with $V^\ast$ a commutative DGA whose components are finite dimensional vector spaces, then by the previous lemma $B_\bullet$ and $DB_\bullet$ are cofibrant objects in $(Ch(R))^{\pGM}$. Note that there is a morphism of left $(A\boxtimes B)_\bullet$-module $\xi_M\boxtimes \xi_N:(A\boxtimes B)_\bullet \to (DA\boxtimes DB)_\bullet$ that sends $a\otimes b$ to $\xi_M(a)\otimes \xi_N(b)$. We can show that it is a quasi-isomorphism by reasoning perversity wise and using the fact that all chain complexes are cofibrant over a field. Furthermore, by Lemma \ref{lemma:Bourbaki_tensorhom}, there is isomorphism from $(DA\boxtimes DB)_\bullet$ to $D(A\boxtimes B)_\bullet$.       
    \end{proof}
    Using Theorem \ref{thm:commuDPDA_is_BV}, we get the following result.  
    \begin{corollaire}
        Under the assumptions of the previous proposition, there exists a perverse BV-algebra structure on $HH^\ast_\bullet(A\boxtimes B)$.
    \end{corollaire}
    \begin{exemple}\label{ex:BV_prod_mani_pseudomani}
         Let $Y$ be a compact, oriented, second countable pseudomanifold and let $M$ be a compact, simply-connected, oriented smooth manifold. We consider their product $X=Y\times M$, it is a pseudomanifold endowed with the product filtration. We present a pBV algebra structure on $HH^\ast_\bullet(\widetilde N^\ast_\bullet(X; \mathbb{Q}))$. As we can see in the proof of Proposition \ref{prop:BV_on_BUP}, it suffices to show that $HH^\ast_\bullet(\widetilde A^\ast_{PL, \bullet}(X))$ is a pBV algebra.  
         
         Notice that $\widetilde A^\ast_{PL, \bullet}(M) \simeq F_{\overline{0}}(A_{PL}^\ast(M))^\ast_\bullet$ since $M$ is a manifold. The conditions we have on $Y$ and $M$ ensure that there exist pBV algebra structures on $HH^\ast_\bullet(F_{\overline{0}}(A^\ast_{PL}(M)))$ and $HH^\ast_\bullet(\widetilde N^\ast_\bullet(Y; \mathbb{Q}))$. In general, the components of the cochain complex $A^\ast_{PL}(M)$ are not finite dimensional vector spaces. However, if $\Lambda_M$ is a minimal model of  $A^\ast_{PL}(M)$, it is a cochain complex of finite dimensional vector spaces and we have $\widetilde A^\ast_{PL, \bullet}(M) \simeq F_{\overline{0}}(\Lambda_M^\ast)^\ast_\bullet$. By the above given corollary, we can endow $HH^\ast_ \bullet((\widetilde A^\ast_{PL, \bullet}(Y)\boxtimes \widetilde A^\ast_{PL, \bullet}(M)))$ with a pBV algebra structure. 
         
         We denote by $pr_M:X\to M$ and $pr_Y:X\to Y$ the canonical projections. They are stratified maps so they induce morphisms ($pr_M^\ast$ and $pr_Y^\ast$) on the associated blown-up cochain complexes. Composing the product with $pr_Y^\ast\boxtimes pr_M^\ast$ gives a perverse chain complex morphism
         \[(\widetilde A^\ast_{PL, \bullet}(Y)\boxtimes \widetilde A^\ast_{PL, \bullet}(M))_\bullet \to \widetilde A^\ast_{PL, \bullet}(X).\]
         It is even a morphism of pDGAs since we work with commutative algebras. Using arguments similar to \cite[Proposition 13.2]{CST18BUP-Alpine}, we can show that it is a quasi-isomorphism. 
         Lemma \ref{lem:BV_transfer} implies that $HH^\ast_\bullet(\widetilde A^\ast_{PL, \bullet}(X))$ is a pBV algebra.
    \end{exemple}

    \subsection{Comparing the pBV algebra structures}
    We have endowed both sides of (\ref{eq:tensorprod-Hoch-coho}) with a perverse BV-algebra structure. We now explain how to get an isomorphism. We consider perverse chain complexes over a field $\mathbb{F}$. Let $A_\bullet$ and $B_\bullet$ be two pDGAs. Using Proposition \ref{prop:bar-cofib-field} and Remark \ref{rem:quillen_adj_pch}, we notice that the perverse chain complexes $\mathbb{B}(A\boxtimes B)_\bullet$ and $\left(\mathbb{B}(A)\boxtimes\mathbb{B}(B)\right)_\bullet$ are two cofibrant resolutions of $(A\boxtimes B)_\bullet$ as pDG $(A\boxtimes B)^e$-modules. Hence, there exists a quasi-isomorphism of $(A\boxtimes B)_\bullet$-bimodules $(\mathbb{B}(A)\boxtimes \mathbb{B}(B))_\bullet \to \mathbb{B}(A\boxtimes B)_\bullet$
    which induces a quasi-isomorphism
    \[\Hom_{(A\boxtimes B)^e}(\mathbb{B}(A\boxtimes B), A\boxtimes B)_\bullet \simeq  \Hom_{(A\boxtimes B)^e}(\mathbb{B}(A)\boxtimes \mathbb{B}(B), A\boxtimes B)_\bullet.\]
    Furthermore, if we assume that either $A_\bullet$ or $B_\bullet$ is of the form $F_{\overline{0}}(V)_\bullet$ with $V^\ast$ a chain complex whose components are finite dimensional vector spaces then we have an isomorphism
    \[\Hom_{(A\boxtimes B)^e}(\mathbb{B}(A)\boxtimes \mathbb{B}(B), A\boxtimes B)_\bullet \simeq (\Hom_{A^e}(\mathbb{B}(A),A)\boxtimes \Hom_{B^e}(\mathbb{B}(B),B))_\bullet.\]    
    These two morphisms give the sought isomorphism (\ref{eq:tensorprod-Hoch-coho}). However, to show that this morphism preserves the perverse BV-algebra structure, we need to give an explicit description of the map that goes from one cofibrant resolution to the other. We'll be using the \emph{Alexander-Whitney map} and the \emph{Eilenberg-Zilber map} \cite[Chap VII - Section 8]{MacL95}. They are well-defined for bisimplicial objects in an abelian category. 
    \begin{defi}
    Let $A_\bullet$ and $B_\bullet$ be two pDGAs. The \emph{Alexander-Whitney map} $AW_{\ast,\bullet}: \mathbb{B}(A\boxtimes B)_\bullet\to (\mathbb{B}(A)\boxtimes \mathbb{B}(B))_\bullet$ is a morphism of perverse chain complexes whose length $k\in\mathbb{N}$ component is denoted $AW_{k,\bullet}$. \\
    It is given for $k=0$ by
    \begin{align*}
        AW_{0, \bullet}: &(A\boxtimes B \boxtimes A \boxtimes B)_\bullet \to (A\boxtimes A \boxtimes B \boxtimes B)_\bullet \\
        ~     &a_1\otimes b_1\otimes a_2 \otimes b_2 \mapsto (-1)^{\Real{a_2}\Real{b_1}}a_1\otimes a_2 \otimes b_1 \otimes b_2.       
    \end{align*}
    and for $k>0$ by
    \[AW_{k, \bullet}:\mathbb{B}_k(A\boxtimes B)_\bullet \to \bigoplus_{i=0}^k (A\boxtimes (s\overline{A})^{\boxtimes i}\boxtimes A)\boxtimes (B\boxtimes (s\overline{B})^{\boxtimes (k-i)}\boxtimes B)\]
    \begin{align*}
        &AW_{k, \bullet}(a_0\otimes b_0[a_1\otimes b_1|\ldots |a_k \otimes b_k]a_{k+1}\otimes b_{k+1})\\
        &=(-1)^{\Real{b_0}\eta_1+\sum\limits_{j=1}^k(\Real{b_j}-1)\eta_j}\sum_{i=0}^k (-1)^{\sum\limits_{j=i+1}^k \eps_j + \sum\limits_{j=0}^{i-1}\overline{\eps}_j}a_0[a_{1,i}]a_{i+1}\ldots a_ka_{k+1}\otimes b_0b_1\ldots b_i[b_{i+1,k}]b_{k+1}.      
    \end{align*}
    where $\eta_j=\Real{a_{k+1}}+\sum_{l=j+1}^k \Real{s(a_l)}$, $\eps_i=\Real{a_0}+\sum_{k=1}^i \Real{s(a_k)}$ and $\overline{\eps}_i=\Real{b_0}+\sum_{k=1}^i \Real{b_k}$. 
    \end{defi}
    To define the Eilenberg-Zilber map, we first need to present $\mathcal{S}_{s,t}$ the set of $(s,t)$-\emph{shuffles}. It is given by 
    \[\mathcal{S}_{s,t}:=\{\sigma\in \mathfrak{S}_{s+t}| \sigma(1)<\sigma(2)<\ldots<\sigma(s) \text{ and } \sigma(s+1)<\sigma(s+2)<\ldots<\sigma(s+t)\}.\]    
    The degree of a shuffle $\sigma\in \mathcal{S}_{s,t}$ is set to be 
    \[\Real{\sigma}:=Card\{(i,j) |  1\leq i < j \leq s+t \text{ and } \sigma(i)>\sigma(2)\} \]
    Note that the symmetric group $\mathfrak{S}_m$ acts on $(A\boxtimes T^{\boxtimes m}(s\overline{A})\boxtimes A)_\bullet$ by
    \[\sigma.(a_0[a_1|\ldots|a_m]a_{m+1})=a_0[a_{\sigma^{-1}(1)}|\ldots |a_{\sigma^{-1}(m)}]a_{m+1}.\]   
    \begin{defi}
    Let $A_\bullet$ and $B_\bullet$ be two pDGAs. The \emph{Eilenberg-Zilber map} $EZ_{\ast, \bullet}: (\mathbb{B}(A)\boxtimes \mathbb{B}(B))_\bullet \to \mathbb{B}(A\boxtimes B)_\bullet$ is a morphism of perverse chain complexes which is defined length-wise by
    \begin{align*}
        EZ_{0, \bullet}: &(A\boxtimes A \boxtimes B \boxtimes B)_\bullet \to (A\boxtimes B \boxtimes A \boxtimes B)_\bullet  \\
        ~     &a_1\otimes a_2\otimes b_1 \otimes b_2 \mapsto (-1)^{\Real{a_2}\Real{b_1}}a_1\otimes b_1 \otimes a_2 \otimes b_2.       
    \end{align*}
    \begin{align*}
        &EZ_{k, \bullet}(1[a_1|\ldots|a_{k-i}]1 \otimes 1[b_1|\ldots b_i]1)\\
        &=\sum_{\sigma\in \mathcal{S}_{s,t}} (-1)^{\Real{\sigma}}\sigma.(1\otimes 1\otimes [a_{1}\otimes 1|\ldots |a_{k-i}\otimes 1|1\otimes b_{1}|\ldots|1\otimes b_{i}]\otimes 1\otimes 1).      
    \end{align*}
    \end{defi}
    The \emph{Eilenberg-Zilber theorem} \cite[Theorem 8.5.1]{Wei95} gives the following result.
    \begin{proposition}
        Let $A_\bullet$ and $B_\bullet$ be two pDGAs. The Alexander-Whitney map $AW_{\ast, \bullet}$ and the Eilenberg-Zilber map $EZ_{\ast, \bullet}$ induce a chain homotopy equivalence
        \[\mathbb{B}(A\boxtimes B)_\bullet \underset{EZ_{\ast, \bullet}}{\overset{AW_{\ast, \bullet}}{\rightleftarrows}} (\mathbb{B}(A)\boxtimes \mathbb{B}(B))_\bullet.\]
    \end{proposition}
    The result above is proved just like \cite[Proposition 8.6.13]{Wei95} by considering simplicial perverse chain complexes instead of simplicial modules. We now consider the morphisms induced on the Hochschild cohomology. 
    \begin{proposition}
        Let $A_\bullet$ and $B_\bullet$ be two pDGAs such that either one of them is equal to $F_{\overline{0}}(V)$ where $V^\ast$ is chain complex whose components are finite dimensional vector spaces. There is an isomorphism of chain complexes
        \[(HH^\ast(A)\boxtimes HH^\ast(B))_\bullet \to HH^\ast_\bullet(A\boxtimes B).\]
    \end{proposition}
    \begin{proof}
    We suppose that either $A_\bullet$ or $B_\bullet$ is of the form $F_{\overline{0}}(V)_\bullet$ where the components of $V^\ast$ are vector spaces of finite dimension. We have an isomorphism
    \[\Psi:\Hom_{(A\boxtimes B)^e}(\mathbb{B}(A)\boxtimes \mathbb{B}(B), A\boxtimes B)_\bullet \simeq (\Hom_{A^e}(\mathbb{B}(A),A)\boxtimes \Hom_{B^e}(\mathbb{B}(B),B))_\bullet.\]
    Consider the following morphisms
    \[AW^\ast_\bullet:=\Hom_{(A\boxtimes B)^e}(AW_{\ast, \bullet}, A\boxtimes B)\circ \Psi^{-1} \text{ and } 
    EZ^\ast_\bullet:=\Psi\circ\Hom_{(A\boxtimes B)^e}(EZ_{\ast, \bullet}, A\boxtimes B).\]
    By remark \ref{rem:quillen_adj_pch}, they are quasi-isomorphisms. 
    \[AW^\ast_\bullet:(\Hom_{A^e}(\mathbb{B}(A),A)\boxtimes \Hom_{B^e}(\mathbb{B}(B),B))_\bullet\to \Hom_{(A\boxtimes B)^e}(\mathbb{B}(A\boxtimes B), A\boxtimes B)_\bullet\]  
    \[EZ^\ast_\bullet:\Hom_{(A\boxtimes B)^e}(\mathbb{B}(A\boxtimes B), A\boxtimes B)_\bullet \to (\Hom_{A^e}(\mathbb{B}(A),A)\boxtimes \Hom_{B^e}(\mathbb{B}(B),B))_\bullet\]      
    \end{proof}

    To show that $H(AW^\ast_\bullet)$ and $H(EZ^\ast_\bullet)$ are isomorphisms of perverse BV-algebras, one just needs to check that they are morphisms of perverse graded algebras and that they preserve the $\Delta$ operator. 
    \begin{lemma}
        Let $(A_\bullet, \cup_A, [-,-]_A, \Delta_A)$ and $(B_\bullet, \cup_B, [-,-]_B, \Delta_B)$ be two perverse BV-algebras. Suppose that there exists an isomorphism of perverse chain complexes $f:A_\bullet \to B_\bullet$ such that for any $a, b \in A_\bullet$
        $f(a\cup_A b)=f(a)\cup_B f(b)$ and $f(\Delta_A(a))=(-1)^{\Real{f}}\Delta_B(f(a))$ then $f$ is an isomorphism of perverse BV-algebras.
    \end{lemma}
    \begin{proof}
        Let $a_1, a_2 \in A_\bullet$, we just have to check that $f([a_1,a_2]_A)=[f(a_1),f(a_2)]_B$. To do so, we write the brackets using the differentials and the cup products. We have 
        \[f([a_1,a_2]_A)=(-1)^{\Real{a_1}}f(\Delta_A(a_1\cup_A a_2) -a_1\cup_A \Delta_A(a_2) -(-1)^{\Real{a_1}}\Delta_A(a_1)\cup_A a_2)\]
        and $[f(a_1),f(a_2)]_B$ is equal to
        \[(-1)^{\Real{a_1}+\Real{f}}\Delta_B(f(a_1)\cup_A f(a_2)) -(-1)^{\Real{a_1}+\Real{f}}f(a_1)\cup_B \Delta_B(f(a_2)) -\Delta_B(f(a_1))\cup_B f(a_2).\]
        One concludes by using the fact that $f$ preserves the cup product and the differential.
    \end{proof}
    
    \begin{proposition}
        Let $A_\bullet$ and $B_\bullet$ be two pDGA such that either one of them is equal to $F_{\overline{0}}(V)$ where $V^\ast$ is chain complex whose components are finite dimensional vector spaces. There is an isomorphism of perverse graded algebras
        \[H(AW^\ast):(HH^\ast(A)\boxtimes HH^\ast(B))_\bullet \to HH^\ast_\bullet(A\boxtimes B).\]
    \end{proposition}
    \begin{proof}
    Let $f,f'\in HC_\bullet(A,A)$ and $g,g'\in HC_\bullet(B,B)$. We want to show that
    \[AW^\ast(f\boxtimes g \cup f'\boxtimes g')=AW^\ast(f\boxtimes g) \cup AW^\ast(f'\boxtimes g').\]
    One can rewrite the term on the left as $(-1)^{\Real{f'}\Real{g}}AW^\ast(f\cup f' \boxtimes g\cup g')$. Furthermore, by definition, $AW^\ast:=\Hom_{(A\boxtimes B)^e}(AW_\ast, A\boxtimes B)$. The equality we want to show is thus the following
    \[(-1)^{\Real{f'}\Real{g}}(f\cup_A f' \boxtimes g\cup_B g')\circ AW_\ast=(f\boxtimes g)\circ AW_\ast \cup (f'\boxtimes g')\circ AW_\ast.\]
    Let $q_A$ be the cofibrant approximation of $A_\bullet$ given in Proposition \ref{prop:bar-cofib-approx}. The following diagrams are commutative 
    \[
    \begin{tikzcd}
    \mathbb{B}(A\boxtimes B) \arrow[rd, "q_{A\boxtimes B}"'] \arrow[r, "AW_\ast"] & \mathbb{B}(A)\boxtimes \mathbb{B}(B) \arrow[d, "q_A\boxtimes q_B"] & \mathbb{B}(A) \arrow[rd, "q_A"', two heads] \arrow[r, "diag_A"] & \mathbb{B}(A)\boxtimes_A \mathbb{B}(A) \arrow[d, "q_A\boxtimes_A q_A"] \\
                                                                       & A\boxtimes B                                         &                                                        & A                                                                        
    \end{tikzcd}
    \]
    where $diag_A$ is defined for $a_0[a_1|\ldots|a_m]a_{m+1}\in \mathbb{B}(A)$ by
    \[diag_A(a_0[a_1|\ldots|a_m]a_{m+1}) = \sum_{i=0}^m (a_0[a_1|\ldots|a_i]\otimes 1) \otimes_A (1\otimes [a_{i+1}|\ldots|a_m]a_{m+1}).\]
    Notice that $f\cup_A f'=(f\boxtimes_A f') \circ diag_A$. We set $C_\bullet:=(A\boxtimes B)_\bullet$. All the triangles in the following diagram are commutative.
    \[
    \adjustbox{scale=0.8}{\begin{tikzcd}[column sep=tiny, row sep=scriptsize]
	{\mathbb{B}(C)} & {\mathbb{B}(A)\boxtimes \mathbb{B}(B)} \\
	\\
	{\mathbb{B}(C)\boxtimes_{C} \mathbb{B}(C)} & {C} & {(\mathbb{B}(A)\boxtimes_A \mathbb{B}(A))\boxtimes (\mathbb{B}(B)\boxtimes_B \mathbb{B}(B))} \\
	\\
	& {((\mathbb{B}(A)\boxtimes \mathbb{B}(B))\boxtimes_{C} (\mathbb{B}(A)\boxtimes \mathbb{B}(B))} & C
	\arrow["{q_{C}}"{description}, from=1-1, to=3-2]
	\arrow["{q_A\boxtimes q_B}"{description}, from=1-2, to=3-2]
	\arrow["{AW_\ast}"{description}, from=1-1, to=1-2]
	\arrow["{diag_{C}}"{description}, from=1-1, to=3-1]
	\arrow["{\mu_{C}}"{description}, from=3-1, to=3-2]
	\arrow["{AW_\ast\boxtimes AW_\ast}"{description}, from=3-1, to=5-2]
	\arrow["{\mu_A\boxtimes \mu_B}"{description}, from=3-3, to=3-2]
	\arrow["{(q_A\boxtimes q_B) \boxtimes_C (q_A\boxtimes q_B)}"{description}, from=5-2, to=3-2]
	\arrow["\tau"{description}, from=3-3, to=5-2]
	\arrow["{diag_A\boxtimes diag_B}"{description}, from=1-2, to=3-3]
    \arrow["(f\boxtimes g)\boxtimes_C (f'\boxtimes g')"{description}, from=5-2, to=5-3]
    \arrow["(f\boxtimes_A f')\boxtimes (g\boxtimes_B g')"{description}, from=3-3, to=5-3]
    \end{tikzcd}}
    \]
    The perverse chain complexes $\mathbb{B}(C)_\bullet$ and $((\mathbb{B}(A)\boxtimes \mathbb{B}(B))\boxtimes_{C} (\mathbb{B}(A)\boxtimes \mathbb{B}(B)))_\bullet$ are both cofibrant approximations of $C$. Hence, there exists a unique (upto homotopy equivalence) quasi-isomorphism between the two. This implies that $\tau \circ (diag_A\boxtimes diag_B)\circ AW_\ast$ is homotopy equivalent to $(AW_\ast \boxtimes_C AW_\ast) \circ diag_C$. Finally, in homology, we have the following equality. 
    \begin{align*}
        (-1)^{\Real{f'}\Real{g}}(f\cup f' \boxtimes g\cup g')\circ AW_\ast  &= (f\boxtimes_A f')\boxtimes (g\boxtimes_B g') \circ (diag_A\boxtimes diag_B)\circ AW_\ast \\
        ~                                                                   &= (f\boxtimes g)\boxtimes_C (f'\boxtimes g') \circ \tau \circ (diag_A\boxtimes diag_B)\circ AW_\ast \\        
        ~                                                                   &= (f\boxtimes g)\boxtimes_C (f'\boxtimes g') \circ  (AW_\ast \boxtimes_C AW_\ast) \circ diag_C\\
                                                                            &= (f\boxtimes g)\circ AW_\ast \cup (f'\boxtimes g')\circ AW_\ast.
    \end{align*}
    \end{proof}
    In order to show that $H(EZ^\ast_\bullet)$ preserves the $\Delta$ operator, we will need to introduce the \emph{shuffle product} \cite[Section 4.2.1]{Lod98}. Notice that the symmetric group $\mathfrak{S}_m$ acts on $(A\boxtimes T^{\boxtimes m}(s\overline{A}))_\bullet$. For $\sigma \in \mathfrak{S}_m$, the action is given by
    \[\sigma.(a_0[a_1|\ldots|a_m])=a_0[a_{\sigma^{-1}(1)}|\ldots |a_{\sigma^{-1}(m)}].\]   
    \begin{defi}
        Let $A_\bullet$ and $B_\bullet$ be two pDGAs. The \emph{shuffle product}, 
        \[sh:(HC_\ast(A)\boxtimes HC_\ast(B))^\ast_\bullet \to HC^\bullet_\ast(A\boxtimes B)\] 
        is a morphism of degree $0$ and perverse degree $\overline{0}$ which is defined for $a_0[a_{1,p}]\in HC_\ast^\bullet(A)$ and $b_0[b_{1,m}]\in HC_\ast^\bullet(B)$ by
    \[sh(a_0[a_{1,p}]\otimes b_0[b_{1,m}])=\sum_{\sigma\in \mathcal{S}_{p,m}} (-1)^{\Real{\sigma}}\sigma.((a_0\otimes b_0) \otimes [a_1\otimes 1|\ldots| a_p\otimes 1|1\otimes b_1|\ldots|1\otimes b_m].\]     
    \end{defi}
    \begin{proposition}\label{prop:HH_tensorprod_BViso}
        Let $A_\bullet$ and $B_\bullet$ be two perverse CDGA satisfying the hypothesis of \emph{Proposition \ref{prop:BV on Hoch of tensor}}. Then, there is an isomorphism of perverse BV-algebras
        \[H(EZ^\ast):HH^\ast_\bullet(A\boxtimes B)\to (HH^\ast(A)\boxtimes HH^\ast(B))_\bullet.\]
    \end{proposition}
    \begin{proof}
        Let $[f]\in HH^\ast_\bullet(A\boxtimes B)$, we would like to show that
        \[(\Delta_A\boxtimes \id_B \pm \id_A\boxtimes \Delta_B)\circ EZ^\ast([f])=EZ^\ast(\Delta_{A\boxtimes B}([f])).\]
        Since the $\Delta$ operator corresponds to the dual of Connes' boundary map, it suffices to show that the following diagram commutes in homology
        \[
        \adjustbox{scale=0.8}{\begin{tikzcd}[column sep=tiny]
        HC^\ast_\bullet(A\boxtimes B) \arrow[rr, "EZ^\ast"] \arrow[d, "\xi_P\circ-"{description}] &             & (HC^\ast(A)\boxtimes HC^\ast(B))^\ast_\bullet \arrow[d, "(\xi_M\circ -)\boxtimes (\xi_N\circ -)"{description}]           \\
        HC^\ast_\bullet(A\boxtimes B, D(A\boxtimes B)) \arrow[d, "\simeq"', "\phi_{A\boxtimes B}"] \arrow[rr, "(1)", phantom] &             & (HC^\ast(A,DA)\boxtimes HC^\ast(B,DB))^\ast_\bullet \arrow[d, "\simeq", "\phi_A\boxtimes \phi_B"']           \\
        D(HC_\ast(A\boxtimes B))^\ast_\bullet \arrow[r, "sh^\vee"] \arrow[d, "B^\vee_{A\boxtimes B}"{description}, ""{name=1}]  & D(HC_\ast(A)\boxtimes HC_ \ast(B))^\ast_\bullet \arrow[d, "(B_A\boxtimes \id \pm \id\boxtimes B_B)^\vee"{description}, ""{name=2}] & \big(D(HC_\ast(A))\boxtimes D(HC_\ast(B))\big)^\ast_\bullet \arrow[l, "\simeq", phantom] \arrow[d, "(B_A^\vee\boxtimes \id \pm \id\boxtimes B_B^\vee)"{description}, ""{name=3}] \\
        D(HC_\ast(A\boxtimes B))^{\ast-1}_\bullet  \arrow[r, "sh^\vee"]            & D(HC_\ast(A)\boxtimes HC_ \ast(B))^{\ast-1}_\bullet  & \big(D(HC_\ast(A))\boxtimes D(HC_\ast(B))\big)^{\ast-1}_\bullet \arrow[l, "\simeq", phantom] \\
        HC^{\ast-1}_\bullet(A\boxtimes B, D(A\boxtimes B)) \arrow[u, "\simeq", "\phi_{A\boxtimes B}"'] &             & (HC^\ast(A,DA)\boxtimes HC^\ast(B,DB))^{\ast-1}_\bullet \arrow[u, "\simeq"', "\phi_A\boxtimes \phi_B"]           \\
        HC^{\ast-1}_\bullet(A\boxtimes B) \arrow[rr, "EZ^\ast"] \arrow[u, "\xi_P\circ-"{description}] &             & (HC^\ast(A)\boxtimes HC^\ast(B))^{\ast-1}_\bullet \arrow[u, "(\xi_M\circ -)\boxtimes (\xi_N\circ -)"{description}]          
        \arrow[phantom,from=1,to=2,"\scriptstyle{(2)}"]
        \arrow[phantom,from=2,to=3,"\scriptstyle{(3)}"]
        \end{tikzcd}}
        \]
        where $\phi_A: HC^\ast_\bullet(A, DA)\to D(HC_\ast(A))^\ast_\bullet$ is the isomorphism which is given for $f\in HC^\ast_\bullet(A, DA)$ and $a=a_0[a_{1,p}]\in HC_\ast(A)$ by
        \[\phi(f)(a_0[a_{1,p}]):=(-1)^{\Real{a_0)(\Real{a}-\Real{a_0})}}f([a_{1,p}])(a_0)\]
        where $\Real{a}=\Real{a_0}+\sum_{i=1}^p \Real{s(a_i)}$.
        Proving the commutativity of the above diagram amounts to noticing two things:
        \begin{itemize}
            \item $EZ^\ast$ defined on the Hochschild cochain complex corresponds to the dual of the the shuffle product $sh$,
            \item Connes' boundary map is a derivation with respect to the shuffle product.
        \end{itemize}
        We begin by showing that (1) is commutative. Let $f\in HC^\ast_\bullet(A\boxtimes B)$, $a=a_0[a_{1,p}]\in HC_\ast^\bullet(A)$ and $b=b_0[b_{1,m}]\in HC_\ast^\bullet(B)$. We write
        \[sh(a_0[a_{1,p}]\otimes b_0[b_{1,m}])=\sum_{\sigma\in \mathcal{S}_{p,m}} (-1)^{\Real{\sigma}} c_0[c_{\sigma^{-1}(1)}|\ldots | c_{\sigma^{-1}(p+m)}]\]
        and
        \begin{align*}
        &EZ_{p+m, \bullet}(1[a_1|\ldots|a_p]1 \otimes 1[b_1|\ldots b_m]1)\\
        &=\sum_{\sigma\in \mathcal{S}_{s,t}} (-1)^{\Real{\sigma}}1\otimes 1\otimes [c_{\sigma^{-1}(1)}|\ldots|c_{\sigma^{-1}(p+m)}]\otimes 1\otimes 1.      
        \end{align*}
        where $c_0=a_0\otimes b_0$ and
        $c_i=\begin{cases}
            1 \otimes a_i \text{ if } 1\leq i \leq p\\
            b_{i-n} \otimes 1 \text{ if } n+1\leq i \leq p+m.
        \end{cases}$
        
        We have
        \begin{align*}
            sh^\vee\circ\phi_{A\boxtimes B}\circ \xi_P&\circ f(a\otimes b)   = \phi_{A\boxtimes B}\circ \xi_P\circ f\circ sh(a\otimes b) \\
            ~                                                               &= \sum_{\sigma\in \mathcal{S}_{p,m}} (-1)^{\Real{\sigma}} \phi_{A\boxtimes B}\circ \xi_P\circ f(c_0[c_{\sigma^{-1}(1)}|\ldots | c_{\sigma^{-1}(p+m)}]) \\
            ~                                                               &= \sum_{\sigma\in \mathcal{S}_{p,m}} (-1)^{\Real{\sigma}} (-1)^{\Real{c_0)(\Real{c}-\Real{c_0})}} \xi_P\circ f([c_{\sigma^{-1}(1)}|\ldots | c_{\sigma^{-1}(p+m)}])(c_0)            
        \end{align*}
        where $\Real{c}=\Real{c_0}+\sum_{i=1}^{p+m} \Real{s(c_i)}$.
        The other side of the diagram gives the following term
        \begin{align*}
            &(\phi_A\boxtimes \phi_B)\circ ((\xi_M\circ-)\boxtimes (\xi_N\circ-))\circ EZ^\ast\circ f(a\otimes b)   \\
            &= (-1)^{sg}((\xi_M\circ-)\boxtimes (\xi_N\circ-))\circ EZ^\ast\circ f([a_{1,n}]\otimes [b_{1,m}])(a_0\otimes b_0) \\
            &= (-1)^{sg} ((\xi_M\circ-)\boxtimes (\xi_N\circ-))\circ f\circ EZ_\ast([a_{1,n}]\otimes [b_{1,m}])(a_0\otimes b_0) \\
            &= (-1)^{sg}\sum_{\sigma\in \mathcal{S}_{n,m}} (-1)^{\Real{\sigma}} ((\xi_M\circ-)\boxtimes (\xi_N\circ-))\circ f([c_{\sigma^{-1}(1)}|\ldots | c_{\sigma^{-1}(n+m)}])(c_0)            
        \end{align*}
        where $sg=(\Real{a_0}+\Real{b_0})(\Real{a}-\Real{a_0}+\Real{b}-\Real{b_0})$.
        In the proof of Proposition \ref{prop:BV on Hoch of tensor}, we have shown that the diagram
        \[\begin{tikzcd}
        (A\boxtimes B)_\bullet \arrow[rr, "\xi_P"] \arrow[rd, "\xi_M\boxtimes \xi_N"'] &                                                & D(A\boxtimes B)_\bullet \\
                                                                               & (DA\boxtimes DB)_\bullet \arrow[ru, "\simeq"'] &                        
        \end{tikzcd}\]
        is commutative. This implies that the two terms we have computed are equal. 
        We move on to (2). It is well known that Connes' boundary map is a derivation with respect to the shuffle product \cite[Corollary 4.3.4]{Lod98}. For $x\in HH_\ast^\bullet(A)$ and $y\in HH_\ast^\bullet(B)$, we have
        \[B_{A\boxtimes B}(sh(x,y))=sh(B_A(x),y)+(-1)^{\Real{x}}sh(x,B_B(y)).\]    
        Dually, this gives the commutativity of (2) in homology. Finally, the diagram (3) is clearly commutative.
    \end{proof}
    \begin{exemple}
        We consider the pseudomanifold $X=Y\times M$ from Example \ref{ex:BV_prod_mani_pseudomani}. By the previous proposition, we have an isomorphism of perverse BV-algebras
        \[HH^\ast_\bullet(\widetilde N_\bullet^\ast(X; \mathbb{Q})) \simeq \big(HH^\ast(\widetilde N_\bullet^\ast(Y; \mathbb{Q})) \boxtimes HH^\ast(\widetilde N_\bullet^\ast(M; \mathbb{Q}))\big)_\bullet.\]
        This example presents a way of constructing a pseudomanifold $X$ such that the pBV algebra structure found on  $HH^\ast_\bullet(\widetilde N_\bullet^\ast(X; \mathbb{Q}))$ is non trivial.
    \end{exemple}

\appendix
\section{Characterisation of cofibrant objects in \texorpdfstring{$(Ch(\mathbb{F}))^{\pGM}$}{pCh(F)}}\label{appendix:chara-cofib-field}
    In this appendix, we consider perverse chain complexes over a field $\mathbb{F}$. The projective model category structure of $Ch(\mathbb{F})$ can be lifted to the category $(Ch(\mathbb{F}))^{\pGM}$. By applying \cite[Theorem 5.1.3.]{Hov99}, we get the following result. 
    \begin{thm}
        The category of perverse chain complexes $(Ch(\mathbb{F}))^{\pGM}$ is equipped with a model structure induced by the (projective) model structure on chain complexes. For this structure, a perverse chain complex $Z^\bullet$ is cofibrant if and only if for any $\overline{p}\in \pGM$, the map
        \[\colim_{\overline{q}<\overline{p}} Z^{\overline{q}}\to Z^{\overline{p}}\]
        is a degree-wise split injection with cofibrant cokernel. 
    \end{thm}
    Notice that, since we work over a field, all chain complexes are cofibrant and all injections are split. We are going to give a sufficient condition on $Z^\bullet$ that ensures that the above map is be a degree-wise injection. 
    
    \begin{defi}
        We say that a perverse chain complex $Z^\bullet$ satisfies the \emph{minimum condition} if for any perversities $\overline{p}, \overline{q}_1, \overline{q}_2$ such that $\overline{q}_1< \overline{p}$ and $\overline{q}_2< \overline{p}$, we have 
    \[i_1(Z^{\overline{q}_1})\cap i_2(Z^{\overline{q}_2})=i(Z^{\min\{\overline{q}_1, \overline{q}_2\}})\]
    where $\min\{\overline{q}_1, \overline{q}_2\}:=\max\{\overline{r}\in\pGM~|~\overline{r}\leq\overline{q}_1 \text{ and } \overline{r}\leq\overline{q}_2\}$ and the maps $i, i_1, i_2$ are the structure morphisms
    \[\begin{tikzcd}
                                                              & Z^{\overline{q}_1} \arrow[rd, "i_1"] &                 \\
    {Z^{\min\{\overline{q}_1, \overline{q}_2\}}} \arrow[rr, "i"] &                                     & Z^{\overline{p}}. \\
                                                              & Z^{\overline{q}_2} \arrow[ru, "i_2"] &                
    \end{tikzcd}\]
    \end{defi}
    
    \begin{proposition}
        If a perverse chain complex $Z^\bullet$ satisfies the minimum condition and if its structure morphisms are injections then, for any $\overline{p}\in \pGM$, the map
        \[\colim_{\overline{q}<\overline{p}} Z^{\overline{q}}\to Z^{\overline{p}}\]
        is a degree-wise injection. In other words, $Z^\bullet$ is a cofibrant perverse chain complex. 
    \end{proposition}
    
    \begin{proof}
        The statement is clearly true for $\overline{0}$. Let $\overline{p}\in \pGM$ be a non-trivial perversity. 
        We consider the following set 
        \[Per_{\overline{p}}:=\{\overline{q}\in \pGM \,|\, \overline{q}<\overline{p}\}.\]
        Since we work with positive perversities, $\min Per_{\overline{p}}=\overline{0}$ and the set is finite.  \\
        For $\overline{q}\in Per_{\overline{p}}$, a non-trivial perversity, let $M_{\overline{q}}$ be the set of maximal elements of $Per_{\overline{q}}$. We consider the perversity 
        \[\overline{m}_{\overline{q}}:=\min M_{\overline{q}}.\]
        where we define the minimum of $S\subset\pGM$ as $\min S:=\max\{\overline{r}\in \pGM |\forall \overline{s}\in S\; \overline{r}\leq \overline{s}\}$. We denote by $k_{\overline{q}}:Z^{\overline{m}_{\overline{q}}}\to Z^{\overline{q}}$ the structure morphism and let $\overline{Z}_{\overline{q}}$ be the supplementary of $k_{\overline{q}}\left(Z^{\overline{m}_{\overline{q}}}\right)$ in $Z^{\overline{q}}$.  \\
        The minimum condition implies that     
        \[\colim_{\overline{q}<\overline{p}} Z^{\overline{q}}=Z^{\overline{0}}\bigoplus_{\overline{q}\in Per_{\overline{p}}\setminus\overline{0}} \overline{Z}_{\overline{q}}.\]
        The map $\colim_{\overline{q}<\overline{p}} Z^{\overline{q}}\to Z^{\overline{p}}$ is a injection since it is induced by the structure morphisms which are injections by hypothesis.
    \end{proof}
    Note that asking for the structure morphisms to be injections is a necessary condition to have a cofibrant perverse chain complex. 
    \begin{proposition}
        The structure morphisms of a cofibrant perverse chain complexes are cofibrations of chain complexes. In particular, they are injections.
    \end{proposition}
    \begin{proof}
        Let $Z^\bullet$ be a cofibrant perverse chain complex and $\overline{q}< \overline{p}$ two \emph{consecutive} perversities. This means that there is no perversity $\overline{r}$ such that 
        \[\overline{q}< \overline{r} < \overline{p}.\]
        We want to show that the structure morphism $s:Z^{\overline{q}}\to Z^{\overline{p}}$ has the left lifting property with respect to trivial fibrations. \\
        Let $f:A_\ast\twoheadrightarrow B_\ast$ be a trivial fibration of chain complexes. Suppose we have a commutative diagramme of the following form 
        \[\begin{tikzcd}
        Z^{\overline{q}}_\ast \arrow[r, "g_1"] \arrow[d, "s"'] & A_\ast \arrow[d, "f", two heads] \\
        Z^{\overline{p}}_\ast \arrow[r, "g_2"']           & B_\ast                          
        \end{tikzcd}
        \]
        We consider two perverse chain complexes $M^\bullet, N^\bullet$ which are defined by
        \[M^{\overline{r}}_\ast=\begin{cases}
        A_\ast & \text{ if } \overline{r}=\overline{q} \text{ ou }\overline{p}, \\
        0 & \text{ else} \\
    \end{cases}\]
    and
        \[N^{\overline{r}}_\ast=\begin{cases}
        A_\ast & \text{ if } \overline{r}=\overline{q}, \\
        B_\ast & \text{ if } \overline{r}=\overline{p},\\    
        0 & \text{ else.} \\
    \end{cases}\]
    We have a morphism of perverse chain complexes $\tilde{f}:M^\bullet\to N^\bullet$
    \[  \begin{tikzcd}
        M^\bullet \arrow[d, "\tilde f", two heads] & \ldots \arrow[r] & 0 \arrow[r] \arrow[d] & M^{\overline{q}}_\ast=A_\ast \arrow[d, "\id"'] \arrow[r, "\id"] & M^{\overline{p}}_\ast=A_\ast \arrow[d, "f"', two heads] \arrow[r] & 0 \arrow[r] \arrow[d] & \ldots \\
        N^\bullet                                  & \ldots \arrow[r] & 0 \arrow[r]           & N^{\overline{q}}_\ast=A_\ast \arrow[r, "f", two heads]          & N^{\overline{p}}_\ast=B_\ast \arrow[r]                            & 0 \arrow[r]           & \ldots
        \end{tikzcd}\]
    Since $f$ and $\id$ are trivial fibrations of chain complexes, $\tilde f$ is a trivial fibration in the category of perverse chain complexes. \\
    We have a commutative diagram of perverse chain complexes
        \[\begin{tikzcd}
        0 \arrow[r] \arrow[d] & M^\bullet \arrow[d, "\tilde f", two heads] \\
        Z^\bullet \arrow[r, "\tilde g"']           & N^\bullet                          
        \end{tikzcd}
        \]
    where $\tilde g:Z^\bullet\to N^\bullet$ is trivial except in perversities $\overline{q}$ and $\overline{p}$:
        \[\tilde g^{\overline{q}}=g_1, \quad \tilde g^{\overline{p}}=g_2.\]
    Since $Z^\bullet$ is cofibrant, there exists a lift $h:Z^\bullet\to M^\bullet$ such that \[\tilde f \circ h = \tilde g.\]
    In particular, we have $h^{\overline{q}}, h^{\overline{p}}$ maps of chains complexes such that
        \[\id \circ h^{\overline{q}}=g_1 \text{ and } f\circ h^{\overline{p}} =g_2.\]
    \[
    \begin{tikzcd}
      &    & M^{\overline{q}}_\ast=A_\ast \arrow[ld] \arrow[d]       \\
    Z^{\overline{q}}_\ast \arrow[d, "s"'] \arrow[r, "g_1"'] \arrow[rru, "h^{\overline{q}}", dashed] & N^{\overline{q}}_\ast=A_\ast \arrow[d, two heads] & M^{\overline{p}}_\ast=A_\ast \arrow[ld, "f", two heads] \\
    Z^{\overline{p}}_\ast \arrow[r, "g_2"'] \arrow[rru, "h^{\overline{p}}", dashed]                 & N^{\overline{p}}_\ast=B_\ast                      &                                          
    \end{tikzcd}\]
    Since $h$ is a morphism of perverse chain complexes, its components commute with the structure morphisms:
    \[\id\circ h^{\overline{q}}= h^{\overline{p}}\circ s.\]
    To sum up, we have
    \[h^{\overline{p}}\circ s = g_1 \text{ and } f\circ h^{\overline{p}} =g_2.\]
    In other words, $h^{\overline{p}}$ is a lift for the first diagram. This proves the claim for consecutive perversities.\\
    Now suppose that $\overline{q}$ and $\overline{p}$ are not consecutive perversities. Then there exists a finite number of perversities $\{\overline{r}_i\}_{1\leq i \leq m}$ such that
    \[\overline{r}_i \text{ and } \overline{r}_{i+1} \text{ are consecutive perversities for every } i\in\{0, 1, \ldots, m\}\]
    with $\overline{r}_0=\overline{q}$ and $\overline{r}_{m+1}=\overline{p}$.\\
    Note that the class of cofibrations is stable under composition (more generally, a class of maps which verify the left lifting property with respect to another class of maps is stable under composition).  
    Since the structure morphism $s:Z^{\overline{q}}\to Z^{\overline{p}}$ is obtained as the composite of cofibrant structure morphisms
    \[s_i:Z^{\overline{r}_i}\to Z^{\overline{r}_{i+1}} \text{ with } i\in \{0,\ldots, n\}\]
    it is cofibrant.
    \end{proof}
    We now try to understand what kind of perverse chain complexes verify the minimum condition. Let $X$ be a pseudomanifold of dimension $n$. We denote by $\mathcal{S}_X$ the set of strata of $X$. Remember that a perversity $\overline{p}\in \pGM$ can be seen as a map $\overline{p}: \mathcal{S}_X\to \mathbb{N}$ (Remark \ref{rem:genperv}).
    \begin{defi}
        A chain complex $(D_\ast, d)$ is \emph{p-filtered} by $X$ if there exists a set map
        \[\lambda_X: D_\ast \to \pGM\]
        which takes each chain to a certain perversity on $X$ that verifies the following properties:
        \begin{itemize}
            \item $\lambda_X(0)=\overline{0}$,
            \item $\lambda_X(\alpha c)\leq \lambda_X(c)\quad \forall \alpha \in \mathbb{F}, c\in D_\ast,$
            \item $\lambda_X(c+c')\leq \max\{\lambda_X(c), \lambda_X(c')\} \quad \forall c, c'\in D_\ast$.
        \end{itemize}
        The \emph{p-filtration} of $D_\ast$ with respect to $\lambda_X$ is the perverse chain complex $Filt_\ast^{\overline{\bullet}}(D)$ given by
        \[Filt_k^{\overline{p}}(D):=\{c\in D_k \,|\, \lambda_X(c) \leq \overline{p} \text{ and } \lambda_X(dc) \leq \overline{p}\}\]
        for $k\in \mathbb{Z}$ and $\overline{p}\in \pGM$. We remark that the structure morphisms are given by the inclusions. 
    \end{defi}
    \begin{rem}
        These notions were already considered by Chataur, Saralegui and Tanré \cite[Definition 1.22]{CST18ration}. In their terminology, a p-filtered chain complex is a \emph{strict perverse chain complex} and its p-filtration corresponds to the associated \emph{intersection chain complex}. Hence, the blown-up intersection cochain complex $\widetilde N^\ast_\bullet(X; \mathbb{F})$ and the blown-up of Sullivan's polynomial forms $\widetilde A_{PL, \bullet}^\ast(X; \mathbb{Q})$ are examples of p-filtrations (\cite[Example 1.34]{CST18ration}). 
    \end{rem}
    \begin{exemple}
        Let $C_\ast(X; \mathbb{F})$ be the singular chain complex of $X$. For a singular chain $c\in C_\ast(X; \mathbb{F})$, we fix
        \[\lambda_X(c)=\min \{\overline{p}\in \pGM \,|\, x \text{ is } \overline{p}\text{-allowable}\}.\]
        The intersection chain complex $I^\bullet C_\ast(X;\mathbb{F})$ is the p-filtration of the singular chain complex with respect to $\lambda_X$. 
    \end{exemple}
    \begin{proposition}
        If a perverse chain complex is a p-filtration of a certain chain complex then it satisfies the minimum condition.
    \end{proposition}
    \begin{proof}
    Let $(D_\ast, d, \lambda)$ be a p-filtered chain complex and consider $Z^\bullet$ the p-filtration of $D_\ast$ with respect ot $\lambda$. Let $\overline{p}, \overline{q}_1, \overline{q}_2$ be perversities such that $\overline{q}_1< \overline{p}$ and $\overline{q}_2< \overline{p}$, we want to show that
    \[i_1(Z^{\overline{q}_1})\cap i_2(Z^{\overline{q}_2})=i(Z^{\overline{m}})\]
    where $\overline{m}:=\min\{\overline{q}_1, \overline{q}_2\}$ and the maps $i, i_1$ and $i_2$ are structure morphisms
    \[\begin{tikzcd}
                                                              & Z^{\overline{q}_1} \arrow[rd, "i_1", hook] &                 \\
    {Z^{\overline{m}}} \arrow[rr, "i", hook] &                                     & Z^{\overline{p}}. \\
                                                              & Z^{\overline{q}_2} \arrow[ru, "i_2", hook] &                
    \end{tikzcd}\] 
    By commutativity of the structure morphisms, we always have
    \[i(Z^{\overline{m}})\subset i_1(Z^{\overline{q}_1})\cap i_2(Z^{\overline{q}_2}).\]
    By definition of $Z$, we get
    \[i_1(Z^{\overline{q}_1})\cap i_2(Z^{\overline{q}_2})=\{c\in D_\ast \,|\, \lambda(c) \leq \overline{q}_1,\,\lambda(c) \leq \overline{q}_2,\,\lambda(dc) \leq \overline{q_1} \text{ and } \lambda(dc) \leq \overline{q_2}\}.\]
    This means that if $c\in i_1(Z^{\overline{q}_1})\cap i_2(Z^{\overline{q}_2})$ then 
    \[\lambda(c), \lambda(dc)\in \{\overline{r}\in \pGM | \overline{r}\leq \overline{q}_1 \text{ and }\overline{r}\leq\overline{q}_2\}.\] 
    But, 
    \[\overline{m}=\max\{\overline{r}\in \pGM | \overline{r}\leq \overline{q}_1 \text{ and }\overline{r}\leq\overline{q}_2\}\]
    whence $\lambda(c)\leq \overline{m}$ and $\lambda(dc)\leq \overline{m}$.
    In other words, 
    \[i_1(Z^{\overline{q}_1})\cap i_2(Z^{\overline{q}_2})\subset i(Z^{\overline{m}}).\]
    This proves the statement.
    \end{proof}
    Since the structure morphisms of a p-filtration are given by the inclusion, we immediatly get the following result. 
    \begin{corollaire}\label{coro:ex_cofib_cplx_field}
    Let $\mathbb{F}$ be a field and $D_\ast$ a p-filtered chain complex. The p-filtration of $D_\ast$ is a cofibrant perverse chain complex. 
    In particular, for $X$ a pseudomanifold, the perverse chain complexes $I^\bullet C_\ast(X;\mathbb{F})$, $\widetilde N^\ast_\bullet(X; \mathbb{F})$, $\widetilde A^\ast_{PL, \bullet}(X; \mathbb{Q})$ and $(\widetilde{A_{PL}\otimes C})^\ast_\bullet(X; \mathbb{Q})$ \emph{(given in Proposition \ref{prop:integ_polyforms})} are cofibrant. 
    \end{corollaire}

\emergencystretch=1em 
\printbibliography[title={Bibliography}] 
\end{document}